\documentclass[final,oneside,notitlepage,10pt]{article}

\usepackage{amsthm}
\usepackage{thmtools}
\usepackage{graphicx}
\usepackage[margin=2cm,font=normalsize,labelfont=bf]{caption}
\usepackage{verbatim}
\usepackage[shortlabels]{enumitem}
\usepackage{fancyhdr}
\usepackage[]{geometry}
\usepackage{xstring}
\usepackage{needspace}
\usepackage{color}
\usepackage{amstext}
\usepackage{nameref}
\usepackage[hypertexnames=false]{hyperref}
\usepackage{mathdots}
\usepackage[normalem]{ulem}
\usepackage{chngcntr}
\usepackage[section]{placeins}
\usepackage{tikz-cd}    
  
\makeatletter
\newcounter{denseversion}
\newcounter{comments}
\newcounter{authorcounter}
\newcounter{adresscounter}

\def\title#1{\gdef\@title{#1}}
\def\@title{}
\def\subtitle#1{\gdef\@subtitle{#1}}
\def\@subtitle{}

\def\authortagsused{0}
\def\adresstag#1{\if!#1!\else$^{\;#1\;}$\fi}
\def\@authorsep#1{
  \ifnum\value{authorcounter}=#1 and \else\unskip, \fi
}

\renewcommand{\author}[2][]{
  \stepcounter{authorcounter}
  \if!#1!\else\gdef\authortagsused{1}\fi
  \ifnum\value{authorcounter}=1
    \def\@authorstringa{#2\adresstag{#1}}
    \def\@authorstringb{#2}
    \def\@authorstringc{#2\adresstag{#1}}
  \else
    \ifnum\value{authorcounter}=2
      \g@addto@macro\@authorstringa{\@authorsep{2}#2\adresstag{#1}}
      \g@addto@macro\@authorstringb{\@authorsep{2}#2}
      \g@addto@macro\@authorstringc{\\#2\adresstag{#1}}
    \else
      \g@addto@macro\@authorstringa{\@authorsep{3}#2\adresstag{#1}}
      \g@addto@macro\@authorstringb{\@authorsep{3}#2}
      \g@addto@macro\@authorstringc{\\#2\adresstag{#1}}
    \fi
  \fi}
\def\@author{\ifnum\value{denseversion}=0\@authorstringa\else\@authorstringb\fi}

\def\@adressstringa{}
\def\@adressstringb{}
\newcommand{\adress}[2][]{
  \stepcounter{adresscounter}
  \ifnum\value{adresscounter}=1
    \g@addto@macro\@adressstringa{\ifnum\authortagsused=0\def\br{\\}\else\def\br{, }\fi\adresstag{#1}#2}
    \g@addto@macro\@adressstringb{\def\br{\\}\adresstag{#1}\parbox[t]{14cm}{#2}}
  \else
    \g@addto@macro\@adressstringa{\\[\bigskipamount]\adresstag{#1}#2}
    \g@addto@macro\@adressstringb{\\[\medskipamount]\adresstag{#1}\parbox[t]{14cm}{#2}}
  \fi}

\def\preprint#1{\gdef\@preprint{#1}}
\def\@preprint{}
\def\keywords#1{\gdef\@keywords{#1}}
\def\@keywords{}
\def\msc#1{\gdef\@msc{#1}}
\def\@msc{}
\def\email#1{
   \gdef\@email{#1}
   \g@addto@macro\@authorstringc{ {\it (#1)}}}
\def\@email{}
\def\dedication#1{\gdef\@dedication{#1}}
\def\@dedication{}

\def\mybaselinestretch#1{
  \gdef\@mybaselinestretch{#1}
  \renewcommand{\baselinestretch}{\@mybaselinestretch}}

\def\myparskip#1{
  \gdef\@myparskip{#1}
  \setlength{\parskip}{\@myparskip}}

\mybaselinestretch{1.2}
\myparskip{1.5ex}

\binoppenalty=\maxdimen
\relpenalty=\maxdimen

\normalfont
\normalsize
\newlength{\@listleftmargin}
\def\setenumstandard{
  \setlist{leftmargin=\@listleftmargin,itemsep=0pt,topsep=0pt,partopsep=0pt,parsep=\@myparskip}
  \setlist[enumerate]{align=left,labelsep=*,leftmargin=\@listleftmargin,itemsep=0pt,topsep=0pt,partopsep=0pt,parsep=\@myparskip}
}

\def\denseversion{
  \setcounter{denseversion}{1}
  \newgeometry{left=3cm,right=3cm,top=3cm}
  \mybaselinestretch{1.1}
  \myparskip{0.8ex}
  \normalfont
  \def\possiblelinebreak{}
  \fancyfoot[C]{\itshape{--$\,\,$\thepage$\,\,$--}}}

\def\possiblelinebreak{\\}

\renewcommand{\emph}[1]{\def\reserved@a{it}\ifx\f@shape\reserved@a\ul{#1}\else\textit{#1}\fi}

\def\setcrefnames{}
\newcommand{\mytableofcontents}{
   \ifnum\value{denseversion}=0
     \tableofcontents
     \setcrefnames %% \tableofcontents macht die \crefnames kaputt
   \else
     \renewcommand{\baselinestretch}{1.1}
     \setlength{\parskip}{0ex}
     \normalfont
     \begingroup
     \def\addvspace##1{\vskip0.4em}
     \tableofcontents
     \setcrefnames %% \tableofcontents macht die \crefnames kaputt
     \endgroup
     \renewcommand{\baselinestretch}{\@mybaselinestretch}
     \setlength{\parskip}{\@myparskip}
     \normalfont
   \fi}

\newlength{\zeilenlaenge}
\def\putindent#1{
  \settowidth{\zeilenlaenge}{#1}
  \ifnum\zeilenlaenge>\textwidth
    #1
  \else
    \noindent #1
  \fi
}

\hypersetup{
    linktocpage = true,
    colorlinks,
    linkcolor=[RGB]{0,0,120},
    citecolor=[RGB]{0,0,120},
    urlcolor=[RGB]{0,0,120}
}

\def\pdfdaten{
  \hypersetup{
    pdftitle = {\@title},
    pdfauthor = {\@author},
    pdfkeywords = {\@keywords},    
    bookmarksopen = true,
    bookmarksopenlevel = 1
  }}  

\def\showkeywords{\begin{flushleft}\footnotesize\textbf{Keywords}: \@keywords\end{flushleft}}
\def\showmsc{\begin{flushleft}\footnotesize\textbf{MSC 2010}: \@msc\end{flushleft}}

\def\tocsection#1{\section*{#1}\addcontentsline{toc}{section}{#1}}

\def\mytitle{}
\def\zmptitle{
  \begin{tabular}{cc}
    \begin{minipage}[c]{0.4\textwidth}
      \begin{flushleft}
        \includegraphics[width=110pt]{../../tex/zmp}
      \end{flushleft}  
    \end{minipage}&
    \begin{minipage}[c]{0.55\textwidth}
      \begin{flushright}
      {\small\sf\@preprint}
      \end{flushright}
    \end{minipage}
  \end{tabular}
  \vskip 2cm}

\def\maketitle{
  \pdfdaten
  \noindent
  \mytitle
  \begin{center}
    \LARGE\@title\\
    \if!\@subtitle!\else\smallskip\LARGE\@subtitle\\\fi
    \bigskip
    \if!\@author!\else\bigskip\large\@author\\\fi
    \ifnum\value{denseversion}=0
      \if!\@adressstringa!\else\bigskip\normalsize\@adressstringa\\\fi
      \if!\@email!\else\ifnum\value{authorcounter}=1\bigskip\normalsize\textit{\@email}\\\else\fi\fi
    \else
    \fi
    \if!\@dedication!\else\bigskip\normalsize{\@dedication}\\\fi
  \end{center}
  \ifnum\value{denseversion}=0\vskip 1.5cm\else\vskip0.5cm\fi}

\def\kobib#1{
  \begin{raggedright}
  \ifnum\value{denseversion}=0\else\small\fi
  \Oldbibliography{#1/kobib}
  \bibliographystyle{#1/kobib}
  \end{raggedright}
  \ifnum\value{denseversion}=0\else
      \noindent
      \if!\@authorstringc!\else
        \ifnum\authortagsused=0\ifnum\value{authorcounter}>1\normalsize\@authorstringc\\[\medskipamount]\else\fi\else\normalsize\@authorstringc\\[\medskipamount]\fi
      \fi
      \if!\@adressstringb!\else\normalsize\@adressstringb\\{}\fi
      \ifnum\authortagsused=0
        \ifnum\value{authorcounter}=1
          \if!\@email!\else\linebreak\normalsize\textit{\@email}\\{}\fi
        \else
        \fi
      \else
      \fi
  \fi}

\let\Oldbibliography\bibliography
\def\bibliography#1{
  \begin{raggedright}
  \ifnum\value{denseversion}=0\else\small\fi
  \Oldbibliography{#1}
  \end{raggedright}
  \ifnum\value{denseversion}=0\else
      \medskip
      \noindent
      \if!\@authorstringc!\else
        \ifnum\authortagsused=0\ifnum\value{authorcounter}>1\normalsize\@authorstringc\\[\medskipamount]\else\fi\else\normalsize\@authorstringc\\[\medskipamount]\fi
      \fi
      \if!\@adressstringb!\else\normalsize\@adressstringb\\{}\fi
      \ifnum\authortagsused=0
        \ifnum\value{authorcounter}=1
          \if!\@email!\else\linebreak\normalsize\textit{\@email}\\{}\fi
        \else
        \fi
      \else
      \fi
  \fi
}

\newenvironment{commentfigure}{\begin{comment}}{\end{comment}}
\newenvironment{sidewayscommentfigure}{\begin{minipage}}{\end{minipage}}
\newenvironment{displaycomment}{\begin{list}{}{\rightmargin=1cm\leftmargin=1cm}\item\sf\begin{small}\color{gray}}{\end{small}\end{list}}

\def\tocmark#1{}

\def\draftstamp#1{
  \def\tocmark##1{
    \ifnum\c@secnumdepth=0\section{##1}\fi
    \ifnum\c@secnumdepth=1\subsection{##1}\fi
    \ifnum\c@secnumdepth=2\subsubsection{##1}\fi
    \ifnum\c@secnumdepth=3\subsubsection{##1}\fi
  }
  \ifnum\value{comments}=0
    \gdef\@draft{DRAFT - Edited on \today\ by #1 - Comments are not displayed}
  \else
    \gdef\@draft{DRAFT - Edited on \today\ by #1 - Comments are displayed}
  \fi
  \fancyhead[C]{\footnotesize\tt\textcolor{red}{\@draft}}}

\def\skript{
  \renewenvironment{displaycomment}{}{}
  \ifnum\value{comments}=0
    \renewenvironment{example*}{\comment}{\endcomment}
    \renewenvironment{remark*}{\comment}{\endcomment}
  \else\fi
  \parindent=0mm        
}

\mybaselinestretch{}
\setenumstandard

\makeatother

\input{kobib}
\usepackage{amssymb}
\usepackage{amsmath}
\usepackage{amstext}
\usepackage{mathrsfs}
\usepackage{euscript}
\usepackage[all,cmtip]{xy}
\usepackage{xstring}
\usepackage{slashed}
\usepackage{tensor}
\usepackage{mathtools}

\def\ul{\underline}
\entrymodifiers={+!!<0pt,\fontdimen22\textfont2>}

\def\Z {\mathbb{Z}}

\def\R {\mathbb{R}}
\def\C {\mathbb{C}}

\def\hc#1{\mathrm{h}_{#1}}
\def\h {\mathrm{H}}

\def\subset{\subseteq}

\makeatletter
\renewenvironment{proof}[1][\nameProof]
  {\par\pushQED{\qed}%
   \normalfont \topsep6\p@\@plus6\p@\relax
   \trivlist
   \item[\hskip\labelsep
         \itshape
         #1\@addpunct{.}]
  \leavevmode}
  {\popQED\endtrivlist\@endpefalse}
\makeatother

\def\notebox#1#2{\begin{minipage}[b]{#1}\sloppy\renewcommand{\baselinestretch}{0.8}\footnotesize \begin{center}#2\end{center}\end{minipage}}

\def\mqquad{\hspace{-4em}}

\newcommand{\arr}[1][r]{\ar@<0.7ex>[#1]\ar@<-0.7ex>[#1]}
\newcommand{\arrr}[1][r]{\ar@<1.4ex>[#1]\ar[#1]\ar@<-1.4ex>[#1]}

\newlength{\myeqt} % Darin wird die Länge des übergebenen Textes abgespeichert
\newlength{\myeqs} % Darin wird die Länge des übergebenen Symbolds abgespeichert
\newlength{\myeqm} % Wieviel "klein" über die Breite des Symbolds hinausgehen darf
\newlength{\myeqn} % Die Standardbreite für große Boxen
\newcommand\symtext[3][\myeqn]{
  \settowidth{\myeqt}{#2}
  \settowidth{\myeqs}{$#3$}
  \addtolength{\myeqs}{\the\myeqm}
  \ifdim\myeqt>\myeqs
    \stackrel{\hspace{-#1}\notebox{#1}{\medskip #2 \\ $\downarrow$\smallskip}\hspace{-#1}}{#3}
  \else
    \stackrel{\text{#2}}{#3}
  \fi}

\def\brackets#1{\IfStrEq{#1}{-}{}{(#1)}}
\def\subindex#1{\IfStrEq{#1}{-}{}{_{#1}}}



\newlength{\myl}
\newcommand\sheaf[1]{\unitlength 0.1mm
  \settowidth{\myl}{$#1$}
  \addtolength{\myl}{-0.8mm}
  \begin{picture}(0,0)(0,0)
  \put(2,0){\text{\underline{\hspace{\myl}}}}
  \end{picture}#1\hspace{-0.15mm}}

\def\ddt#1#2#3{\left.\frac{\mathrm{d}^{\IfStrEq{#1}{1}{}{#1}}}{\mathrm{d}#2}\IfStrEq{#2}{#3}{\right.}{\right|_{#3}}}

\def\ev{\mathrm{ev}}

\newlength{\widthtmp}
\def\length#1{\settowidth{\widthtmp}{#1}\the\widthtmp}

\def\lli#1{{}_#1}

\definecolor{olivegreen}{rgb}{.33,.55,.18}

\newcommand{\ie}{i.e., }
\newcommand{\eg}{e.g., }

\def\pss#1{\prescript{}{#1}}

\renewcommand{\O}{\operatorname{O}}

\newcommand{\SO}{\operatorname{SO}}
\newcommand{\Spin}{\operatorname{Spin}}

\newcommand{\Pin}{\operatorname{Pin}}
\newcommand{\U}{\operatorname{U}}
\newcommand{\PU}{\operatorname{PU}}

\newcommand{\Cl}{\operatorname{Cl}}
\newcommand{\CCl}{\C\!\operatorname{l}}

\newcommand{\opp}{\operatorname{op}}
\newcommand{\pr}{\operatorname{pr}}
\newcommand{\lact}{\triangleright}
\newcommand{\ract}{\triangleleft}

\newcommand{\grp}{\operatorname{grpd}}

\newcommand{\id}{\operatorname{id}}
\newcommand{\unit}{\mathbf{1}}

\newcommand{\Isocat}{\mathscr{I}\mathrm{so}}
\newcommand{\Homcat}{\mathscr{H}\mathrm{om}}
\newcommand{\Endcat}{\mathscr{E}\!\mathrm{nd}}
\newcommand{\Hom}{\mathrm{Hom}}
\newcommand{\End}{\mathrm{End}}
\newcommand{\Aut}{\mathrm{Aut}}

\newcommand{\AUT}{\mathscr{A}\mathrm{ut}}
\newcommand{\Inn}{\mathrm{Inn}}

\newcommand{\Vect}{\mathscr{V}\mathrm{ect}}

\newcommand{\sVect}{\mathrm{s}\Vect}

\newcommand{\twoVect}{2\Vect}
\newcommand{\stwoVect}{\mathrm{s}\twoVect}
\newcommand{\stwoVectgrpd}[1]{\mathrm{Grpd}(\mathrm{s}\twoVect_{#1})}
\newcommand{\ssstwoVect}{\mathrm{ss\text{-}s}\twoVect}

\newcommand{\VectBdl}{\mathscr{V}\mathscr{B}\mathrm{dl}}
\newcommand{\LineBdl}{\mathscr{L}\mathscr{B}\mathrm{dl}}
\newcommand{\sVectBdl}{\mathrm{s}\VectBdl}
\newcommand{\sLineBdl}{\mathrm{s}\LineBdl}
\newcommand{\Bdl}[1]{#1\text{-}\mathscr{B}\mathrm{dl}}

\newcommand{\twoVectBdl}{2\VectBdl}
\newcommand{\twoVectBdlgrpd}[2]{\mathrm{Grpd}(2\VectBdl_{#1}\brackets{#2})}
\newcommand{\stwoVectBdl}{\mathrm{s}2\VectBdl}
\newcommand{\stwoVectBdlref}{\mathrm{s}2\VectBdl^{\mathrm{ref}}}
\newcommand{\ssstwoVectBdlref}{\mathrm{ss}\text{-}\mathrm{s}2\VectBdl^{\mathrm{ref}}}
\newcommand{\stwoVectBdlrefinv}{\mathrm{s}2\VectBdl^{\mathrm{inv}\text{-}\mathrm{ref}}}
\newcommand{\stwoVectBdlgrpd}[2]{\mathrm{Grpd}(\mathrm{s}2\VectBdl_{#1}\brackets{#2})}
\newcommand{\sstwoVectBdl}{\mathrm{ss}\text{-}2\VectBdl}
\newcommand{\ssstwoVectBdl}{\mathrm{ss\text{-}s}2\VectBdl}

\newcommand{\twoLineBdl}{2\LineBdl}
\newcommand{\stwoLineBdl}{\mathrm{s}2\LineBdl}
\newcommand{\sLineBdlgrpd}[2]{\mathrm{Grpd}(\sLineBdl_{#1}\brackets{#2})}
\newcommand{\BimodBdl}{\mathscr{B}\mathrm{im}\mathscr{B}\mathrm{dl}}
\newcommand{\sBimodBdl}{\mathrm{s}\BimodBdl}
\newcommand{\sBimodBdlimp}{\mathrm{s}\BimodBdl^{\mathrm{imp}}}
\newcommand{\Grb}{\mathscr{G}\mathrm{rb}}
\newcommand{\sGrb}{\mathrm{s}\Grb}

\newcommand{\sGrbref}{\mathrm{s}\Grb^{\mathrm{ref}}}

\newcommand{\Alg}{\Incl\mathrm{lg}}       
\newcommand{\Algbi}{\Alg^{\mathrm{bi}}}

\newcommand{\sAlg}{\mathrm{s}\Alg}    
\newcommand{\sAlggrpd}[1]{\mathrm{Grpd}(\mathrm{s}\Alg_{#1})}    
\newcommand{\sAlgbi}{\sAlg^{\mathrm{bi}}}

\newcommand{\csAlg}{\mathrm{cs}\Incl\mathrm{lg}}       
\newcommand{\cssAlg}{\mathrm{cs\text{-}s}\Alg}       
       
\newcommand{\ssAlg}{\mathrm{ss}\Alg}       
       
\newcommand{\sssAlg}{\mathrm{ss\text{-}s}\Alg}

\newcommand{\AlgBdl}{\Alg\mathscr{B}\mathrm{dl}}       
\newcommand{\sssAlgBdl}{\mathrm{ss\text{-}s}\AlgBdl}       
\newcommand{\ssAlgBdl}{\mathrm{ss}\AlgBdl}       
\newcommand{\AlgBdlbi}{\Alg\mathscr{B}\mathrm{dl}^{\mathrm{bi}}}
\newcommand{\sAlgBdlbi}{\sAlg\mathscr{B}\mathrm{dl}^{\mathrm{bi}}}
\newcommand{\sAlgBdlbigrpd}[2]{\mathrm{Grpd}(\sAlg\mathscr{B}\mathrm{dl}^{\mathrm{bi}}_{#1}\brackets{#2})} 
\newcommand{\AlgBdlbigrpd}[2]{\mathrm{Grpd}(\Alg\mathscr{B}\mathrm{dl}^{\mathrm{bi}}_{#1}\brackets{#2})} \newcommand{\sssAlgBdlbi}{\sssAlg\mathscr{B}\mathrm{dl}^{\mathrm{bi}}}
\newcommand{\ssAlgBdlbi}{\ssAlg\mathscr{B}\mathrm{dl}^{\mathrm{bi}}}      
\newcommand{\sAlgBdl}{\sAlg\mathscr{B}\mathrm{dl}}       
\newcommand{\sAlgBdlgrpd}[2]{\mathrm{Grpd}(\sAlg\mathscr{B}\mathrm{dl}_{#1}\brackets{#2})} 
\newcommand{\AlgBdlgrpd}[2]{\mathrm{Grpd}(\Alg\mathscr{B}\mathrm{dl}_{#1}\brackets{#2})}

\newcommand{\cssAlgBdlbi}{\cssAlg\mathscr{B}\mathrm{dl}^{\mathrm{bi}}} 
\newcommand{\csAlgBdlbi}{\csAlg\mathscr{B}\mathrm{dl}^{\mathrm{bi}}}       
\newcommand{\Mfd}{\mathscr{M}\mathrm{fd}}

\newcommand{\contrMfd}{c\Mfd}

\newcommand{\Incl}{\mathscr{A}}

\usepackage[latin1]{inputenc}
\usepackage[english]{babel}
\usepackage[english]{cleveref}

\def\refname{References}

\def\quot#1{``#1''}

\def\quand{\quad\text{ and }\quad}

\def\nameTheorem{Theorem}
\def\nameTheorems{Theorems}
\def\nameDefinition{Definition}
\def\nameDefinitions{Definitions}
\def\nameCorollary{Corollary}
\def\nameCorollaries{Corollaries}
\def\nameProposition{Proposition}
\def\namePropositions{Propositions}
\def\nameConstruction{Construction}
\def\nameConstructions{Constructions}
\def\nameLemma{Lemma}
\def\nameLemmas{Lemmas}
\def\nameRemark{Remark}
\def\nameRemarks{Remarks}
\def\nameProblem{Problem}
\def\nameProblems{Problems}
\def\nameExample{Example}
\def\nameExamples{Examples}
\def\nameExercise{Exercise}
\def\nameExercises{Exercises}

\def\nameProof{Proof}
\def\nameEquation{\unskip}
\def\nameEquations{\unskip}
\def\nameSection{Section}
\def\nameSections{Sections}
\def\nameAppendix{Appendix}
\def\nameAppendixs{Appendices}
\def\nameFigure{Figure}
\def\nameFigures{Figures}

\input{theorems}

\def\mathscr#1{\EuScript{#1}}

\usepackage{tikz-cd}
\usepackage{todonotes}

\title{2-vector bundles}

\author[a]{Peter Kristel}
\email{peter.kristel@umanitoba.ca}

\adress[a]{
  University of Manitoba\\
  Department of Mathematics\\
  Winnipeg, MB R3T 2N2, Canada}

\author[b]{Matthias Ludewig}
\email{matthias.ludewig@mathematik.uni-regensburg.de}

\adress[b]{
Universit\"at Regensburg\\
  Fakult\"at f\"ur Mathematik\\
93053 Regensburg, Germany}

\author[c]{Konrad Waldorf}
\email{konrad.waldorf@uni-greifswald.de}

\adress[c]{
  Universit\"at Greifswald\\
  Institut f\"ur Mathematik und Informatik\\
  D-17487 Greifswald}

\msc{}
\dedication{}
\pdfdaten

\denseversion
\allowdisplaybreaks
\usepackage{amstext}

\begin{document}

\maketitle

\begin{abstract}
\noindent
We develop a ready-to-use comprehensive theory for (super) 2-vector bundles over smooth manifolds. It is based on the bicategory of (super) algebras, bimodules, and intertwiners as a model for 2-vector spaces. We discuss symmetric monoidal structures and the corresponding notions of dualizability, and  we derive a classification in terms of Cech cohomology with values in a crossed module. One important feature of our 2-vector bundles is that they contain bundle gerbes as well as ordinary algebra bundles as full sub-bicategories, and hence provide a unifying framework for these so far distinct objects. We provide several examples of isomorphisms between bundle gerbes and algebra bundles, coming from representation theory, twisted K-theory, and spin geometry.
%\showmsc
%\howkeywords
\end{abstract}

\mytableofcontents

\setsecnumdepth{1}

\section{Introduction}

In this paper, we develop a theory of 2-vector bundles. 
Just as a vector bundle over a manifold $X$ is a collection of vector spaces (the fibres), together with information on how these fit together to a bundle structure, a 2-vector bundle will be a geometric object whose fibres are 2-vector spaces.
Here, a 2-vector space is an object of a delooping bicategory of the category of vector spaces, i.e., a symmetric monoidal bicategory such that the endomorphism category of the monoidal unit is equivalent to a vector space category.

In this paper, we choose a certain bicategory of finite-dimensional algebras, finite-dimensional bimodules and intertwiners over $k=\R$ or $\C$, which we denote by $\twoVect_k$. In fact, we work throughout with a larger bicategory $\stwoVect_k$ of \emph{super} 2-vector spaces, where algebras and bimodules are $\Z_2$-graded and all intertwiners are grading preserving; this is crucial for several applications.
The idea to consider algebras as 2-vector spaces was brought up by Schreiber in discussions at the n-Category Caf\'e \cite{Schreiber2006,Schreiber2007}, also see \cite[\S A]{Schreiber2009} and \cite[\S 4.4]{schreiber2} for early references advocating this point of view. 
There are many other possible (mostly \quot{smaller}) choices, see the Appendix of \cite{Bartlett2015} for a \quot{bestiary of 2-vector spaces}; e.g.,  Kapranov-Voevodsky 2-vector spaces \cite{kapranov1}.  Another round of examples comes from various flavors of linear categories.
In particular, in the context of TQFTs, a frequently used model for 2-vector spaces are linear finitely semisimple abelian categories. 
Our reason to use the algebra/bimodule model of 2-vector spaces is that it is, for the most part, straightforward to turn algebras and bimodules into bundles.
Another reason is that many of the examples we have in mind are of this form.

In our previous article \cite{Kristel2022} we have described how to turn (finite-dimensional) algebras and bimodules into algebra \emph{bundles} and bimodule \emph{bundles} over smooth manifolds, in such a way that a bicategory $\sAlgBdlbi_k(X)$ is obtained.  
The composition in $\sAlgBdlbi_k(X)$ is the relative tensor product of bimodule bundles over an algebra bundle. The well-definedness of such a tensor product does not come for free and depends crucially on the admitted class of bimodules and the conditions one imposes for their local triviality. Our article \cite{Kristel2022} describes and solves these issues.
In \cite{Kristel2022} and in the present work, we stick to \emph{finite-dimensional} algebras and bimodules; in our subsequent article \cite{StringRep} we  also consider versions with von Neumann algebras. 

The bicategory $\sAlgBdlbi_k(X)$ of algebra bundles is   a preliminary version of 2-vector bundles. It is preliminary because algebra bundles do not satisfy bicategorical descent -- one cannot glue locally defined algebra bundles along invertible bimodule bundles to again obtain an algebra bundle. Another way to say this is that algebra bundles -- though locally trivial as bundles -- are not locally trivial in the correct bicategorical sense, i.e., as 2-vector bundles.
In more precise terms, we have shown in \cite{Kristel2022} that the presheaf of bicategories $\sAlgBdlbi_k$, which assigns to each smooth manifold $X$ the bicategory $\sAlgBdlbi_k(X)$, is only a pre-2-stack, but not a 2-stack.
However, every pre-2-stack can be 2-stackified using the plus construction of Nikolaus and Schweigert \cite{nikolaus2}, and this is precisely our definition (\cref{def:applplus}) of super 2-vector bundles:
\begin{equation*}
        \stwoVectBdl_k := (\sAlgBdlbi_k)^+.
\end{equation*}
In particular, our super 2-vector bundles then form a 2-stack, i.e., they do satisfy descent and are locally trivial as 2-vector bundles.

Concretely, a super 2-vector bundle $\mathscr{V}$ on a manifold $X$ consists of the following data: a surjective submersion $Y \to X$,
a super algebra bundle $\mathcal{A}$ over $Y$; a bundle $\mathcal{M}$ of invertible $\pr_1^*\!\mathcal{A}$-$\pr_2^*\!\mathcal{A}$-bimodules over $Y^{[2]}$; and an invertible even intertwiner 
\begin{equation*}
\mu : \pr_{23}^* \mathcal{M} \otimes_{\pr_{2}^*\mathcal{A}} \pr_{12}^*\mathcal{M} \to \pr_{13}^* \mathcal{M}
\end{equation*}
over $Y^{[3]}$, which  satisfies a coherence condition over $Y^{[4]}$.
Here, $Y^{[k]}$ denotes the $k$-fold fibre product of $Y$ with itself over $X$, which comes with projection maps $\pr_{i_1, \dots, i_k} : Y^{[l]} \to Y^{[k]}$. Schematically, we depict a super 2-vector bundle as
\begin{equation} \label{TwoVBDiagramIntro}
\mathscr{V}= \left ( 
\begin{aligned}
\xymatrix{
\mathcal{A} \ar[d] & \mathcal{M} \ar[d] & \mu \ar@{..}[d] & \footnotesize{\text{coherence}} \ar@{..}[d] \\ 
Y \ar[d] & Y^{[2]} \ar@<0.5ex>[l]^{\pr_2}\ar@<-0.5ex>[l]_{\pr_1} & Y^{[3]} \ar@<1ex>[l]\ar@<-1ex>[l]\ar[l] & Y^{[4]} \ar@<1.5ex>[l] \ar@<-1.5ex>[l] \ar@<0.5ex>[l] \ar@<-0.5ex>[l]\\ 
X
} 
\end{aligned}
\right )\text{.}
\end{equation}
We think of the surjective submersion $Y\to X$ in terms of a generalized open cover of $X$:
any open cover $(U_i)_{i \in I}$ of $X$ gives rise to a surjective submersion $Y = \coprod_{i \in I} U_i \to X$ (in fact, a local diffeomorphism); yet,  allowing general surjective submersions is often convenient in practice.
We emphasize that  every super 2-vector bundle has its individual surjective submersion $Y \to X$, and that in general it cannot be assumed to be trivial (i.e., $Y=X$).

The isomorphism class of the typical fibre of the algebra bundle $\mathcal{A}$ in \eqref{TwoVBDiagramIntro} is not an invariant under isomorphisms of super 2-vector bundles.
However, it turns out that its Morita class is an invariant, which we call the Morita class of the super 2-vector bundle $\mathscr{V}$; it is the higher analog of the rank of an ordinary vector bundle.

For each manifold $X$, we obtain a bicategory $\stwoVectBdl_k(X)$ of super 2-vector bundles on $X$, and one may ask for a classification, i.e., a description of the set of isomorphism classes.
In this article, we classify super 2-vector bundles of fixed Morita class $A$ in terms of the first \v{C}ech cohomology with values in the automorphism 2-group $\AUT(A)$ of $A$.
As a  Lie 2-group, the automorphism 2-group $\AUT(A)$ can be presented by a smooth crossed module $A^\times_0 \stackrel{\curvearrowleft}{\to} \Aut(A)$, where $A_0^\times$, the group of even units in $A$, includes into $\Aut(A)$ as inner automorphisms, and $\Aut(A)$ acts on $A_0^\times$ in the standard way.
Precisely, our classification result is the following; see \cref{th:class2vect} in the main text.

\begin{maintheorem}
        \label{th:main}
        Let $A$ be a Picard-surjective super algebra. Then, there is a canonical bijection
        \begin{equation*}
                \mathrm{h}_0\bigl(A\text{-}\stwoVectBdl(X)\bigr) \cong \check{\mathrm{H}}^1(X, \AUT(A))\text{,}
        \end{equation*}
        \ie super 2-vector bundles over X of Morita class A are classified by the \v{C}ech cohomology of X with values in $\AUT(A)$.
\end{maintheorem}

Here, $\hc 0$ denotes the set of isomorphism classes of objects in a bicategory. Moreover, we say that a super algebra $A$ is \emph{Picard-surjective} if the natural map $\Aut(A) \to \mathrm{Pic}(A)$ is surjective. We proved in \cite[Prop.~A2]{Kristel2022} that \emph{every} super algebra is Morita equivalent to a Picard-surjective one; hence, our classification result in fact applies to \emph{all} 2-vector bundles.

Though there is a well-defined tensor product between arbitrary 2-vector bundles, the bicategory $\stwoVectBdl_k(X)$ is in general not monoidal.
The problem lies already in the bicategory $\sAlgBdlbi_k(X)$ and has been observed and discussed in \cite{Kristel2022}.
It is rooted in the fact that the well-definedness of the \emph{relative} tensor product between bimodule bundles (the composition) requires a condition (\quot{implementing}) that is not preserved under the \emph{exterior} tensor product of bimodule bundles (the monoidal structure).
%On the basis of \cite{Kristel2022}, one has two options. 
In \cite{Kristel2022}, we suggest two options to solve this problem.
Either, one can take the underlying sub-bigroupoid $\stwoVectBdlgrpd kX$; this circumvents above problem because \emph{invertible} bimodule bundles are automatically implementing. Or, one can consider the full sub-bicategory $\ssstwoVectBdl_k(X)$ over all \emph{semisimple} super 2-vector bundles; \ie ones whose super algebra bundle $\mathcal{A}$ has semisimple fibres; here, the problem does not arise as bimodule bundles between \emph{semisimple} super algebras are automatically implementing.

The invertible objects in both symmetric monoidal bicategories (i.e., the objects with tensor inverses) provide a unified picture: we show that in both cases  a super 2-vector bundle is invertible if and only if its Morita class is a \emph{central simple super algebra} (\cref{prop:dualizability}). Invertible super 2-vector bundles will also be called \emph{super 2-line bundles}; these form a symmetric monoidal bicategory denoted $\stwoLineBdl_k(X)$.
Our classification specializes to the following; see \cref{th:classline2bundles} in the main text.

\begin{maincorollary}
\label{co:main}
\label{CorollaryIntro}
For any smooth manifold $X$, there is a canonical group isomorphism
\begin{equation*}
\mathrm{h}_0(\stwoLineBdl_k(X)) \cong\mathrm{H}^0(X,\mathrm{BW}_k) \times \mathrm{H}^1(X,\Z_2) \times \check{\mathrm{H}}^2(X,\underline{k}^{\times}).
\end{equation*} 
\end{maincorollary}

Here $\mathrm{BW}_k$ is the Brauer-Wall group for the field $k=\R$ or $\C$, and the first factor on the right hand side specifies the Morita class of the super 2-line bundle on each connected component. 
This result was previously obtained by Mertsch in his PhD thesis \cite{Mertsch2020} by a direct computation, whereas we compute the \v Cech cohomology group $\check{\mathrm{H}}^1(X, \AUT(A))$ in case of a central simple super algebra and then deduce the result from \cref{th:main}.

We also study weaker notions of invertibility in symmetric monoidal bicategories, most importantly \emph{full dualizability}.
While in the bigroupoid $\stwoVectBdlgrpd kX$ dualizability, full dualizability, and invertibility are the same, we show (\cref{prop:dualizability}) that in $\ssstwoVectBdl_k(X)$ all objects, i.e., all semisimple super 2-vector bundles are fully dualizable. In this sense, full dualizability corresponds to semisimplicity, just as it is the case for 2-vector \emph{spaces} \cite{Bartlett2015}.

An important feature of our framework is that both super \emph{algebra bundles} and (super) \emph{bundle gerbes} give rise to super 2-vector bundles. This can be seen directly using the above description in \eqref{TwoVBDiagramIntro}: 
for an algebra bundle $\mathcal{A}$ on $X$, one just takes $Y = X$ and $\mathcal{M} = \mathcal{A}$ to get a 2-vector bundle, while for a bundle gerbe over $k$ with surjective submersion $Y \to X$, one just takes $\mathcal{A} = \underline{k}$, the trivial algebra bundle over $Y$ with typical fibre $k$. In particular, super 2-vector bundles provide a framework in which isomorphisms between super algebra bundles and bundle gerbes can be discussed.
Thus, questions like \quot{When is a bundle gerbe a bundle of algebras?} (see \cite{Pavlov}) obtain a well-defined meaning (and answer: if and only if its Dixmier-Douady class is torsion, see \cref{co:bundlegerbesandalgebrabundles}).
Throughout this paper, we provide a variety of results about the correlation between super 2-line bundles, super algebra bundles, and super bundle gerbes; see \cref{sec:bundlegerbes,sec:inclusionofalgebrabundles,sec:classificationof2linebundles}. \Cref{fig:venn} shows schematically a summary of these results. 
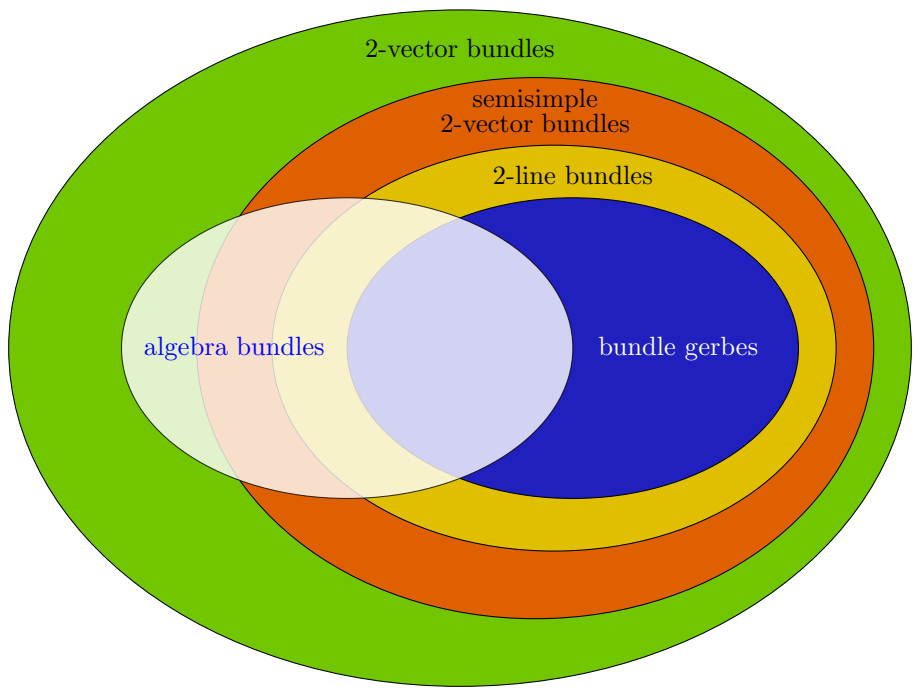
\begin{figure}[h]
\begin{center}
\begin{tikzpicture} 
[set/.style = {draw,
    text opacity = 1}]
 
\draw[set,fill={rgb:red,2;green,5;yellow,2}]    (0,0) ellipse (6cm and 4.5cm);
\draw[set,fill={rgb:red,5;green,1;yellow,2},opacity = 1]      (1,0) ellipse (4.5cm and 3.6cm);
\draw[set,fill={rgb:red,2;green,1;yellow,5},opacity = 1]   (1.25,0) ellipse (3.75cm and 2.7cm);
\draw[set,fill={rgb:red,1;green,1;blue,6},opacity = 1]   (1.5,0) ellipse (3cm and 2cm);
\draw[set,fill=white,opacity = 0.8] (-1.5,0) ellipse (3cm and 2cm);

\node at (0,4)  {2-vector bundles};
\node at (-3,0)[color=blue]  {algebra bundles};
\node at (2.9,0) [color=white]  {bundle gerbes};
\node at (1.5,2.3)   {2-line bundles};
\node at (1,3.3)   {semisimple};
\node at (1,3)   {2-vector bundles};

\end{tikzpicture}
\end{center}
\caption{A Venn diagram relating various sorts of 2-vector bundles and bundle gerbes. The intersection between algebra bundles and bundle gerbes consists of those algebra bundles whose fibres are Morita equivalent to the ground field, and at the same time of all bundle gerbes whose Dixmier-Douady class is torsion. The intersection between algebra bundles and 2-line bundles consists of all algebra bundles whose fibres are central and simple. 
}
\label{fig:venn}
\end{figure}

A geometric example of the situation that a bundle gerbe is isomorphic to an algebra bundle is the following: an oriented Riemannian manifold $X^d$ canonically carries the \emph{spin lifting gerbe} $\mathscr{G}_{\SO_d}^{\Spin_d}$, which is a real line bundle gerbe. 
This bundle gerbe is trivial if and only if $X$ admits a spin structure, and a trivialization is precisely a choice of spin structure. 
On the other hand, one can apply the Clifford algebra construction to the tangent bundle, which yields a bundle $\Cl(TX)$ of real super algebras over $X$. 
Both the spin lifting gerbe and the Clifford algebra bundle give rise to 2-vector bundles, and they may be compared as such. In fact, we prove that
\begin{equation*}
  \mathscr{G}_{\SO_d}^{\Spin_d} \cong \Cl(TX) \otimes \Cl_{-d}
\end{equation*}
as 2-vector bundles, where an isomorphism is provided by the twisted spinor bundle.
We derive this from a general statement (\cref{Thm:CanonicalIsomorphism:Representation}) about lifting gerbes and representations, and also prove some variations concerning complex scalars and non-oriented manifolds.

We remark that precursors to our super 2-vector bundles have appeared in several places, but the presentation as the 2-stackification of a presheaf of bicategories (together with the consequential fact that they form a 2-stack) is new in the present paper.
For example, in \cite{Pennig2011}, Pennig defines \quot{Morita bundle gerbes} with respect to a unital C$^{*}$-algebra $A$. 
If $A$ is finite-dimensional, this essentially reduces to our (ungraded) complex 2-vector bundles of Morita class $A$.
Pennig does not define a bicategory (let alone a 2-stack) but only considers {stable} equivalence classes.
He also classifies Morita bundle gerbes by the \v{C}ech cohomology with coefficients in a C$^{*}$-algebraic version of the 2-group $\AUT(A)$, and our classification theory was strongly motivated by Pennig's ideas and results.
It is worth pointing out that, in spite of the similarities to Pennig's approach, our methods are more conceptual as we use the modern 2-stack-theoretical framework. 
Moreover, we go a bit further, for instance, with the discussion of the tensor product of algebras under this classification.

In \cite{Ershov2016} Ershov adapts Pennig's definition to a finite-dimensional setting. Ershov's Morita bundle gerbes coincide with our ungraded complex 2-line bundles of Morita class $\C$, see \cref{re:moritaclass:c} for further comparison.
\begin{comment}
However, Ershov requires a fixed surjective submersion $Y \to M$, and the bicategory he constructs contains no morphisms that allow to change it. In fact, every Morita bundle gerbe is isomorphic to an ordinary bundle gerbe, this is our \cref{lem:bundlegerbesand2vectorbundles} or \cite[Prop.~2.10]{Ershov2016}. \end{comment}
In his lecture notes \cite{Freed2012} Freed sketches the definition of a 2-category of \emph{invertible algebra bundles} over topological groupoids $\mathcal{G}$, which coincides with (a continuous version of) our notion of super 2-line bundles if $\mathcal{G}$ is the groupoid obtained from a cover $Y \to X$, \ie where $\mathcal{G}_0 = Y$ and $\mathcal{G}_1 = Y^{[2]}$. 
Freed then obtains a classification result similar to \Cref{CorollaryIntro}, using homotopy-theoretical methods. Finally, in \cite{baas1} Baas, Dundas, and Rognes propose a definition of a \quot{charted 2-vector bundle} based on a category of 2-vector spaces that is equivalent to Kapranov-Voevodsky 2-vector spaces. 
In \cite{nikolaus2} Nikolaus and Schweigert recast charted 2-vector bundles in term of 2-stackification, showing that charted 2-vector bundles form a 2-stack. Nonetheless, we do not see any direct relation between these charted 2-vector bundles and our 2-vector bundles, and their framework does not seem to capture the geometric examples we have in mind.

\paragraph{Acknowledgements.}
We thank the Erwin Schr\"odinger International Institute for Mathematics and Physics for hosting the program \quot{Higher Structures and Field Theory} in summer 2020, at which some results of this article have been presented.
PK and KW gratefully acknowledge support from the German Research Foundation under project code WA 3300/1-1.
PK gratefully acknowledges support from the Pacific Institute for the Mathematical Sciences in the form of a postdoctoral fellowship.
ML gratefully acknowledges support from CRC 1085 ``Higher Invariants'', funded by the DFG.
We would like to thank Severin Bunk, Christoph Schweigert and Danny Stevenson for helpful discussions.

%=============================================================================

\setsecnumdepth{2}

\section{2-vector bundles}

\label{sec:super2vectorbundles}

The goal of this section is to introduce the 2-stack of 2-vector bundles in several flavours (e.g., real/complex, ungraded/graded).
We explain first what our models for 2-vector spaces are, and then turn these into bundles over manifolds. This results in a variety of pre-2-stacks which may be viewed as \quot{preliminary 2-vector bundles}.
We then apply the 2-stackification procedure of Nikolaus-Schweigert \cite{nikolaus2} to construct 2-stacks of 2-vector bundles of various flavours.

\subsection{The bicategory of 2-vector spaces} \label{Section2VectorSpaces}

In order to define 2-vector bundles, one first has to settle on a good notion of 2-vector \emph{spaces}. In general, $2$-vector spaces should form a symmetric monoidal bicategory $\mathscr{V}$ such that the monoidal category $\Endcat_{\mathscr{V}}( \mathbf{1})$ of endomorphisms of the monoidal unit $\mathbf{1}$ is isomorphic to some category of vector spaces (\eg finite-dimensional ones) over a field $k$. We describe in this section our choice of a sub-bicategory $\stwoVect_k$ of the symmetric monoidal bicategory $\sAlgbi_k$ of finite-dimensional super algebras over $k$, where $k$ is either $\R$ or $\C$.
We start by recalling the basics of the bicategory $\sAlgbi_k$.

The objects of $\sAlgbi_k$ are monoid objects in the symmetric monoidal category $\sVect_k$ of finite-dimensional super vector spaces; in more detail, these are {finite-dimensional}, $\Z_2$-graded, unital, associative algebras $A$ over $k$ (in short: super algebras). We will use the notation $A=A_0 \oplus A_1$ to denote the graded components.
If $A$ and $B$ are super algebras, then the 1-morphisms $A \to B$ in $\sAlgbi_k$ are $\Z_2$-graded, finitely generated $B$-$A$-bimodules (we will just say \emph{bimodules}).
The 2-morphisms are bimodule intertwiners, and are always required to be parity-preserving.

\begin{comment}
\begin{remark} \label{RemarkFinDimBanach}
One reason to restrict to finite-dimensional algebras and bimodules is that their underlying vector spaces have a canonical topology and smooth structure, which makes it straightforward to put a smooth structure on the underlying algebra bundles and bimodule bundles, and to talk about smoothness in general.
\\
We will also use the fact that finite-dimensional algebras always admit submultiplicative norms (norms satisfying $\|ab\| \leq \|a\|\|b\|$), in other words, they admit the structure of a Banach algebra.
Such a norm can be obtained for example by choosing a scalar product on $A$ and then embed $A$ into the Banach algebra $\End(A)$ of vector space endomorphisms by sending $a \mapsto \ell_a$, where $\ell_a : b \mapsto ab$ is the left multiplication operator.
We remark that since $\End(A)$ is a $C^*$-algebra, such a norm also satisfies the $C^*$-identity, but $A$ is not necessarily $*$-closed when viewed as subset of $\End(A)$.
\end{remark}
\end{comment}

If $M:B \to A$ and $N:C \to B$ are 1-morphisms (\ie $M$ is a $A$-$B$-bimodule and $N$ is a $B$-$C$-bimodule), then the composition in $\sAlgbi_k$ is the relative tensor product $M \otimes_B N$ over $B$, which gives a $A$-$C$-bimodule and hence a 1-morphism $C \to A$.
The reason that 1-morphisms from $B$ to $A$ are taken to be $A$-$B$-bimodules (as opposed to $B$-$A$-bimodules), is that we then have $M\circ N = M \otimes_B N$ for the composition of 1-morphisms in our 2-category; consequently we do not have to distinguish between the symbols $\circ$ and $\otimes$. 
The \emph{identity bimodule} of a super algebra $A$ is $A$, considered as an $A$-$A$-bimodule in the obvious way.
An $A$-$B$-bimodule $M$ is \emph{invertible} if there exists a $B$-$A$-bimodule $N$ such that $M \otimes_B N \cong A$ and $N \otimes_A M \cong B$.
Two super algebras $A$, $B$ are isomorphic in the bicategory $\sAlgbi_k$ if and only if there exists an invertible $A$-$B$-bimodule $M$; this relation is usually called \emph{Morita equivalence}.

\begin{remark}
        The bicategory $\Algbi_k$ of (non-super) algebras, (non-super) bimodules and intertwiners forms a (non-full) sub-bicategory, by declaring algebras and bimodules to be purely even. That way,  the ungraded situation is included in our discussion.
\end{remark}

We recall that to any $A$-$B$-bimodule $M$ one can associate another $A$-$B$-bimodule $\Pi M$ with the graded components swapped, see \cite[\S 2.1]{Kristel2022}.
We also recall the usual way of twisting bimodules by super algebra homomorphisms (\ie even algebra homomorphisms). 
Given a super algebra homomorphism $\phi:A' \to A$ and an $A$-$B$-bimodule $M$, there is an $A'$-$B$ bimodule $\pss{\phi} M$ with underlying vector space $M$, the right $B$-action as before and the left $A'$-action induced along $\phi$. 
Further, if $\psi: B' \rightarrow B$ is another super algebra homomorphism, then we obtain a $A$-$B'$-bimodule $M_\psi$ in a similar way. 
Both twistings can be performed simultaneously, resulting in an $A'$-$B'$-bimodule $\pss{\phi}M_{\psi}$.
If $N$ is another $A'$-$B'$-bimodule, then an intertwiner $u: N \to \pss{\phi}M_{\psi}$ is the same as an {intertwiner along $\phi$ and $\psi$}, \ie an even linear map satisfying 
\begin{equation} \label{IntertwiningCondition}
u(a \lact n \ract b) = \phi(a)\lact u(n) \ract \psi(b), \qquad n\in N, ~~ a\in A', ~~b\in B'.
\end{equation} 
We refer to \cite[Lem.~2.1.3]{Kristel2022} for a summary of further properties of this construction.

For any super algebra $A$, one may consider the Picard group, $\mathrm{Pic}(A)$.
By definition, its elements are the isomorphism classes of invertible $A$-$A$-bimodules, and its group multiplication is given by the relative tensor product over $A$.
If $\phi \in \Aut(A)$, then the bimodule $A_{\phi}$ is invertible.
Moreover, $A_{\phi} \otimes_{A} A_{\psi} \cong A_{\phi \circ \psi}$.
Thus, the assignment $\phi \mapsto A_{\phi}$ induces a group homomorphism
\begin{equation}\label{AutPicMap}
        \Aut(A) \rightarrow \mathrm{Pic}(A).
\end{equation}
A super algebra is called \emph{Picard-surjective}, if every invertible  $A$-$A$-bimodule $M$ is isomorphic to one of the form $A_\phi$ for some $\phi \in \Aut(A)$.
In other words, the map \eqref{AutPicMap} is surjective.

As working with bicategories can be involved due to non-strictness, it is often convenient -- when possible -- to reduce to a one-categorical context.
This leads to the notion of a \emph{framed bicategory}: a bicategory $\mathscr{B}$ together with a category $\mathscr{C}$ with the same objects and a functor $\mathscr{C} \to \mathscr{B}$ such that (a) it is the identity on the level of objects and (b) the image of every morphism of $\mathscr{C}$ has a right adjoint in $\mathscr{B}$.
Assigning to a super algebra homomorphism $\varphi:A \to B$ the $B$-$A$-bimodule $B_{\varphi}$ establishes such a framing
\begin{equation}
\label{eq:framing2vect}
\sAlg_k \to \sAlgbi_k
\end{equation}
for the bicategory of super algebras; see \cite[Lem.~2.1.6]{Kristel2022} for a more detailed discussion.

Another important feature of the bicategory $\sAlgbi_k$ is that it is symmetric monoidal in the sense of Schommer-Pries \cite[Definition 2.3]{pries1}. The monoidal structure is the \emph{graded} tensor product of super algebras, respectively, the exterior graded tensor product on bimodules over $k$. 
%
\begin{comment}
Let $M$ be an $A$-$B$-bimodule, and let $M'$ be an $A'$-$B'$-bimodule. Then the exterior tensor product $M \otimes M'$ is $(A\otimes A')$-$(B\otimes B')$-bimodule in the following way:
\begin{align*}
(a \otimes a') \lact (m\otimes m') \ract (b \otimes b') := (-1)^{|a'||m|+|b||m'|+|b||a'|}(a\lact m \ract b)\otimes (a'\lact m'\ract b')\text{.}
\end{align*}
Further, the identity is graded linear map
\begin{equation*}
\Pi(M \otimes M') \cong \Pi M \otimes M' \cong M \otimes \Pi M'\text{.}
\end{equation*}
It turns out that, just like in \cref{lem:Pi}, 
\begin{equation*}
\Pi (M \otimes M') \to M \otimes \Pi M': m \otimes m' \mapsto m \otimes m'
\end{equation*} 
is an intertwiner, and
\begin{equation*}
\Pi(M \otimes M') \to \Pi M \otimes M : m \otimes m' \mapsto (-1)^{|m'|}m \otimes m'
\end{equation*}
is an intertwiner. 
\\
Also, the identity induces an intertwiner
\begin{equation*}
A_{\varphi_A} \otimes B_{\varphi_B} \cong (A\otimes B)_{\varphi_A\otimes \varphi_B}\text{.}
\end{equation*}
\end{comment}
%
The tensor unit is the field $k$, considered as a trivially graded super algebra, $\unit := k$. This way, we obtain $\Endcat_{\sAlgbi_k}(\unit) \cong \sVect_k$ as symmetric monoidal categories, which qualifies the bicategory $\sAlgbi_k$ as a bicategory of 2-vector spaces.

\begin{remark}
\label{re:dualizable}
One can show that every super algebra $A$ is dualizable with respect to the symmetric monoidal structure of $\sAlgbi_k$. Furthermore, a super algebra is \emph{fully} dualizable if and only if it is semisimple (equivalently, separable), see \cite{Bartlett2015}. Finally, it is well-known that a super algebra is \emph{invertible} if it is central and simple.
The group of isomorphism classes of invertible objects in $\sAlgbi_k$ is called the \emph{Brauer-Wall group} of $k$. It is well known that $\mathrm{BW}_\R=\Z_8$ and $\mathrm{BW}_{\C}=\Z_2$, and that representatives are the real and complex Clifford algebras, $\Cl_n$ ($n=0,...,7$) and $\CCl_n$ ($n=0,1$), respectively.
\end{remark}

\label{SectionImplementingModules}

The passage from algebras to algebra bundles is straightforward, whereas the passage from bimodules to bimodule bundles needs extra care, see \cref{sec:algebrabundles}.
The crucial point is to achieve a well-defined relative tensor product of bimodule bundles over an algebra bundle. 
A locally trivial vector bundle structure on the fibrewise relative tensor product does
not come for free and requires some control of the bimodule actions.
This leads us to the definition of the sub-bicategory $\sVect_k \subset \sAlgbi_k$, which will be our choice of 2-vector spaces, and which we explain next.

We associate to any $A$-$B$-bimodule $M$ the group $I(M)$  of triples $(\phi,u,\psi)$  consisting of super algebra automorphisms $\phi\in \Aut(A)$ and $\psi\in \Aut(B)$ and of an invertible (even) intertwiner $u$ along $\phi$ and $\psi$, \ie $u$ satisfies the relation \eqref{IntertwiningCondition}.
It is straightforward to show \cite[\S 3.1]{Kristel2022} that $I(M)$ is a Lie subgroup of $\Aut(A) \times \mathrm{GL}(M) \times \Aut(B)$. We call $I(M)$ the \emph{group of implementers}, due to the fact that if $(\phi, u, \psi) \in I(M)$, then $u$ \emph{implements} the automorphisms $\phi$ and $\psi$, in the sense that \begin{equation} \label{ImplementingRelation}
\phi(a) \lact m = u^{-1}(a \lact u(m)), \qquad m \ract \psi(b) = u^{-1}(u(m) \ract b), \qquad a \in A, ~~m \in M, ~~b \in B\text{.}
\end{equation}
We will next consider the projection maps
\begin{equation} \label{SourceTargetMaps}
        p_\ell: I(M) \to \Aut(A)
        \quand
        p_r: I(M) \to \Aut(B)\text{,}
\end{equation} 
which are smooth group homomorphisms. The following definition has been introduced in \cite[Def.~3.1.3]{Kristel2022}. 

\begin{definition}[Implementing modules] \label{DefinitionImplementing}
Let $A$ and $B$ be super algebras and let $M$ be an $A$-$B$-bimodule.
Then, $M$ is called implementing if the maps $p_{\ell}$ and $p_r$ are open.
\end{definition}

We refer to \cite[\S 3.1]{Kristel2022} for some remarks concerning this definition.
Moreover, we infer from \cite{Kristel2022} the following results.

\begin{proposition}
\begin{enumerate}[(i)]

\item
\label{ExampleImplementing1} 
For any automorphisms $\phi,\psi\in\Aut(A)$, the bimodule $\pss\psi A_\phi$ is implementing.

\item
\label{PropositionInvertibleImplementing}
Every invertible bimodule is implementing.

\item
\label{RemarkHHSemisimple}
Every bimodule over semisimple algebras is implementing.

\item
\label{PropositionTensorProdImplementing}
The relative tensor product of implementing bimodules is implementing. 

\end{enumerate}
\end{proposition}

For examples of non-implementing bimodules, we refer again to \cite{Kristel2022}.
In particular, if $\phi,\psi:A\to B$ are (non-invertible) super algebra homomorphisms, it  happens that the bimodule $\pss \psi B_{\phi}$ is \emph{not} implementing.

Due to \cref{PropositionTensorProdImplementing} we are in position to define a sub-bicategory 
\begin{equation*}
\stwoVect_k\subset \sAlgbi_k
\end{equation*}
 with the same objects (super algebras) but only the \textit{implementing} bimodules as 1-morphisms (and all intertwiners between those as 2-morphisms). By \cref{PropositionInvertibleImplementing}, both bicategories have the same set of isomorphism classes of objects.
 In other words, two algebras are Morita equivalent if and only if they are isomorphic in any of these bicategories.
Moreover, by \cref{RemarkHHSemisimple}, we still have $\Endcat_{\stwoVect_k}(\unit) \cong \sVect_k$, so that $\stwoVect_k$ is a valid choice of a bicategory of super 2-vector spaces.
 One of the fundamental insights of our paper \cite{Kristel2022} was that in the context of algebra \textit{bundles}, it is the smaller bicategory $\stwoVect_k$ which succeeds as a model for 2-vector spaces in relation with 2-vector bundles; this is the topic of the next subsection.

Before we proceed, we note that the framing 
$\sAlg_k \to \sAlgbi_k$
of \cref{eq:framing2vect} does not co-restrict to the sub-bicategory $\stwoVect_k$, because the bimodule $B_{\varphi}$ is not necessarily implementing for all algebra homomorphisms $\varphi:A \to B$. One way to restore this is to restrict to the \emph{groupoid} $\sAlggrpd k$ of super algebras and super algebra \textit{iso}morphisms. Then, by \cref{ExampleImplementing1}, we obtain a new framing
\begin{equation}
\label{eq:framingimp}
\sAlggrpd k \to \stwoVect_k\text{.} 
\end{equation} 
Another option is to restrict to \emph{semisimple} super algebras, i.e., to the full subcategory $\sssAlg_k \subset \sAlg_k$ and to the full sub-bicategory $\ssstwoVect_k \subset \stwoVect_k$ over all semisimple super algebras. Then, by \cref{RemarkHHSemisimple} we obtain a framing
\begin{equation}
\label{eq:framingimp2}
\sssAlg_k \to \ssstwoVect_k\text{.} 
\end{equation}

Similarly, we note that the symmetric monoidal structure on $\sAlgbi_k$ does not restrict to $\stwoVect_k$, because the exterior tensor product of implementing bimodules need not be implementing \cite[Ex. 3.1.5 (4)]{Kristel2022}.
However, the symmetric monoidal structure restricts (by \cref{PropositionInvertibleImplementing}) to the sub-bigroupoid $\stwoVectgrpd k$ and (by \cref{RemarkHHSemisimple}) to the sub-bicategory $\ssstwoVect_k$, so that both of these bicategories are framed symmetric monoidal bicategories.

\subsection{Algebra bundles and bimodule bundles }
\label{sec:algebrabundles}

We will now set up a bundle version of our bicategory $\stwoVect_k$ of super 2-vector spaces over smooth manifolds, which will be a preliminary, yet incomplete, version of 2-vector bundles.
The larger bicategory $\sAlgbi_k$ does not admit such a bundle version as the relative tensor product of perfectly fine but non-implementing bimodule bundles does in general not admit a locally trivial vector bundle structure.

Let $X$ be a smooth manifold. The definition of super algebra bundles is standard, and the definition of bimodule bundles is taken from \cite[\S 4.1]{Kristel2022}.

\begin{definition}[Super algebra bundle]
A \emph{super algebra bundle over} $X$ is a smooth vector bundle $\pi:\mathcal{A} \to X$, with the structure of a super algebra on each fibre $\mathcal{A}_x$, $x \in X$, such that each point in $X$ has an open neighborhood $U \subset X$ for which there exists a super algebra $A$ and a diffeomorphism $\phi: \mathcal{A}|_U \to U \times A$ that preserves fibres and restricts to a super algebra isomorphism $\phi_x: \mathcal{A}_x \to A$ in each fibre over $x\in U$. 
A \emph{morphism between two super algebra bundles} is a vector bundle morphism that is a grading-preserving algebra homomorphism in each fibre. Super algebra bundles over $X$ and homomorphisms form a category $\sAlgBdl_k(X)$. 
\end{definition}

\begin{remark}
\label{re:typicalfibre}
\begin{enumerate}[(1)]

\item 
As defined, super algebra bundles do not necessarily have a single typical fibre, in the sense that the super algebras $A$ of all local trivializations could be chosen to be the same. However, a straightforward argument shows that the restriction of a super algebra bundle to a connected component of $X$ does have a single typical fibre. 

\item A \emph{trivial super algebra bundle} $\mathcal{A}$ over $X$ is given by a family $\{A_i\}_{i \in I}$ of super algebras, one for each connected component $X_i\subset X$, via $\mathcal{A}|_{X_i} := \underline{A}_i := X_i \times A_i$, the trivial algebra bundle with fibre $A_i$. 
Note that trivial super algebra bundles canonically pull back to trivial ones. 
\end{enumerate}
\end{remark} 

\begin{definition}[Bimodule bundle]
\label{def:bimodulebundle}
Let $\mathcal{A}$ and $\mathcal{B}$ be super algebra bundles over $X$. An \emph{$\mathcal{A}$-$\mathcal{B}$-bimodule bundle} is a super vector bundle $\mathcal{M}$ over $X$ with the structure of an $\mathcal{A}_x$-$\mathcal{B}_x$-bimodule in each fibre $\mathcal{M}_x$, such that each point $x \in X$ has an open neighborhood $U\subset X$ over which there exist local trivializations $\phi: \mathcal{A}|_U \to U \times A$ and $\psi:\mathcal{B}|_U \to U \times B$ (as super algebra bundles), an $A$-$B$-bimodule $M$, and a local trivialization $u: \mathcal{M}|_U \to U \times M$ (as a vector bundle) that is fibrewise an intertwiner along $\phi$ and $\psi$.
If $X$ can be covered by open sets supporting local trivializations with the same $A$, $B$, and $M$, then we say that $\mathcal{M}$ has \emph{typical fibre} the triple $(A,M,B)$.  \emph{Morphisms} between $\mathcal{A}$-$\mathcal{B}$-bimodule bundles are super vector bundle morphisms that are even intertwiners in each fibre (they will again be called intertwiners).
$\mathcal{A}$-$\mathcal{B}$-bimodule bundles and their morphisms form a category $\sBimodBdl_{\mathcal{A},\mathcal{B}}(X)$. 
\end{definition}

\begin{example} \label{ExampleTwistedModuleIso}
\begin{enumerate}[(1)]
\item 
Let $\phi:\mathcal{A} \to \mathcal{B}$ be an isomorphism of super algebra bundles over $X$. Then, there is a $\mathcal{B}$-$\mathcal{A}$-bimodule bundle $\mathcal{B}_\phi$, which over a point $x \in X$ has fibres $(\mathcal{B}_x)_{\phi_x}$. 
Its typical fibre is $(B,B_{\varphi},A)$, where $A$ and $B$ are typical fibres of $\mathcal{A}$ and $\mathcal{B}$, respectively, and $\varphi:A \to B$ is an arbitrary super algebra isomorphism, see \cite[Ex. 4.1.5]{Kristel2022}.
We remark that this does \emph{not} generalize to non-invertible super algebra bundle homomorphisms $\phi: \mathcal{A} \to \mathcal{B}$, see \cite[Ex. 4.3.3 \& 4.3.4]{Kristel2022}.

\item
If $\mathcal{A}$ and $\mathcal{B}$ are trivial super algebra bundles over $X$, defined by families $\{A_i\}_{i\in I}$ and $\{B_i\}_{i\in I}$ as in \cref{re:typicalfibre}, then a \emph{trivial $\mathcal{A}$-$\mathcal{B}$-bimodule bundle} is given by a family $\{M_i,\phi_i,\psi_i\}_{i\in I}$ with $A_i$-$B_i$-bimodules $M_i$ and smooth maps $\varphi_i : X_i \to \Aut(A_i)$ and $\psi_i:X_i \to \Aut(B_i)$, one for each connected component $X_i \subset X$, via $\mathcal{M}|_{X_i} := \lli{{\psi_i}}{\sheaf{M_i}}_{\,\phi_i}$. 
Every bimodule bundle is locally isomorphic to a trivial one. 

\end{enumerate}
\end{example}

For a deeper discussion of bimodule bundles we refer to \cite[\S 4.1]{Kristel2022}. It turns out that there is no relative tensor product of bimodule bundles as defined above (see \cref{sec:algebrabundles}, \cite[\S 4.2]{Kristel2022}). 
We proved in \cite[Thm.~4.2.6]{Kristel2022} that the following additional constraint fixes this problem.

\begin{definition}[Implementing bimodule bundle]
\label{def:implementingbimodulebundle}
Let $\mathcal{A}$ and $\mathcal{B}$ be super algebra bundles over $X$. An $\mathcal{A}$-$\mathcal{B}$-bimodule bundle $\mathcal{M}$ is called \textit{implementing} if all
fibres $\mathcal{M}_x$ are implementing. Implementing $\mathcal{A}$-$\mathcal{B}$-bimodule bundles and their morphisms form a full subcategory $\sBimodBdlimp_{\mathcal{A},\mathcal{B}}(X)$ of $\sBimodBdl_{\mathcal{A},\mathcal{B}}(X)$.
 \end{definition}

Let $\mathcal{A}$, $\mathcal{B}$, and $\mathcal{C}$ be super algebra bundles, let $\mathcal{M}$ be an implementing $\mathcal{A}$-$\mathcal{B}$-bimodule bundle and $\mathcal{N}$ an implementing $\mathcal{B}$-$\mathcal{C}$-bimodule bundle. In \cite[Prop.~4.2.3]{Kristel2022} we have constructed the relative tensor product $\mathcal{M} \otimes_{\mathcal{B}} \mathcal{N}$, which is an implementing $\mathcal{A}$-$\mathcal{C}$-bimodule bundle. Thus, we have a functor
\begin{equation} \label{Composition}
\sBimodBdl_{\mathcal{A},\mathcal{B}}^{\mathrm{imp}}(X) \times \sBimodBdl_{\mathcal{B},\mathcal{C}}^{\mathrm{imp}}(X) \to \sBimodBdl_{\mathcal{A},\mathcal{C}}^{\mathrm{imp}}(X)\text{.}
\end{equation}
It is unproblematic to construct associators and unitors for this tensor product functor, turning it into the composition of 1-morphisms in a bicategory whose objects are super algebra bundles. 

\begin{definition}[Bicategory of super algebra bundles] \label{DefinitionPrestwoVectBdl}
The \emph{bicategory  of super algebra bundles} over $X$, denoted $\sAlgBdlbi_k(X)$, has objects super algebra bundles over $X$, 1-morphisms implementing bimodule bundles over $X,$ and 2-morphisms even intertwiners.
The composition is given by the relative tensor product \eqref{Composition}.
\end{definition}

Super algebra bundles are a preliminary version of super 2-vector bundles:  they (only) form a \emph{pre-}2-stack \cite[Prop.~4.5.1]{Kristel2022}, while 
the true super 2-vector bundles that we define in \cref{sec:2stack2vec} form a 2-stack (\cref{ThmAllisStack}).

We recall from \cref{SectionImplementingModules} that the bicategory $\stwoVect_k$ has two interesting sub-bicategories, the sub-bicategory $\stwoVectgrpd k$ and the full sub-bicategory $\ssstwoVect_k$ over all semisimple super algebras. They are symmetric monoidal and come with (symmetric monoidal) framings
\begin{equation*}
\sAlggrpd k \to \stwoVectgrpd k
\quand
\sssAlg_k \to \ssstwoVect_k
\end{equation*}
discussed in \cref{eq:framingimp,eq:framingimp2}. We recall that $\stwoVect_k$ itself is not symmetric monoidal, and also not framed by $\sAlg_k$. 

This picture passes without changes from super 2-vector \emph{spaces} to the bicategory of super algebra  \emph{bundles}.
In the first case, we define the sub-bicategory 
\begin{equation*}
\sAlgBdlbigrpd kX \subset \sAlgBdlbi_k(X)
\end{equation*}
containing only the invertible bimodule bundles (note that by \cite[Lem.~4.2.8]{Kristel2022}, \quot{invertible} is equivalent to \quot{fibrewise invertible}) and the invertible intertwiners.  
Then, \cref{ExampleTwistedModuleIso} furnishes a functor
\begin{equation}
\label{eq:framing1}
\sAlgBdlgrpd kX \to \sAlgBdlbigrpd kX
\end{equation}
which is indeed a framing and symmetric monoidal \cite[Lem.~4.3.1]{Kristel2022}. For semisimple algebras, we first define the full sub-category $\sssAlgBdl_k(X)\subset \sAlgBdl_k(X)$ over all super algebra bundles with semisimple fibres, and similarly, the full sub-bicategory $\sssAlgBdlbi_k(X) \subset \sAlgBdlbi_k(X)$. If $\varphi:\mathcal{A} \to \mathcal{B}$ is a homomorphism of semisimple super algebra bundles, then by \cite[Prop.~4.3.5]{Kristel2022}, we obtain a well-defined $\mathcal{A}$-$\mathcal{B}$-bimodule bundle $\mathcal{B}_\varphi$, and hence, a functor
\begin{equation}
\label{eq:framing2}
\sssAlgBdl_k(X) \to \sssAlgBdlbi_k(X)\text{.} 
\end{equation}
By \cite[Cor. 4.3.6]{Kristel2022}, this is again a framing and symmetric monoidal. 
The following result is \cite[Cor. 4.4.4]{Kristel2022}.

\begin{proposition}
\label{prop:dualizabilityalgebrabundles}
The following table describes the dualizable, fully dualizable, and invertible objects in both symmetric monoidal categories of preliminary super 2-vector bundles: 
\begin{center}
\begin{tabular}{l|c|c|c}
 & dualizable & fully dualizable & invertible \\\hline
$\sAlgBdlbigrpd kX$ & central simple & central simple & central simple \\
$\sssAlgBdlbi_k(X)$ & all & all & central simple \\
\end{tabular}
\end{center}
\end{proposition}

\begin{remark}
\label{re:directsum}
Both symmetric monoidal bicategories $\sAlgBdlbigrpd kX$ and $\sssAlgBdlbi_k(X)$ have a second symmetric monoidal structure given by the direct sum of algebra bundles, and the exterior direct sum of bimodule bundles. The two monoidal structures are compatible with each other in the sense of distributive laws, and form a \emph{commutative rig bicategory}.
\end{remark}

We denote by
\begin{equation*}
\cssAlgBdlbi_k(X) \subseteq \sssAlgBdlbi_k(X)
\end{equation*}
the full sub-bicategory over all  super algebra bundles with central simple fibres. It will turn out to be a preliminary version of super 2-\emph{line} bundles.
Note that $\cssAlgBdlbi_k(X)$ is symmetric monoidal with the induced tensor product.
A result of Donovan-Karoubi \cite[Theorems 6 and 11]{DK70} shows the following.

\begin{theorem}
\label{th:classalg}
There is a canonical bijection
\begin{equation*}
\mathrm{h}_0(\cssAlgBdlbi_k(X)) \cong \mathrm{H}^0(X,\mathrm{BW}_k) \times {\mathrm{H}}^1(X,\Z_2) \times \mathrm{Tor}(\check{\mathrm{H}}^2(X,\underline{k}^{*}))\text{,}
\end{equation*}
where $\check{\mathrm{H}}^2$ denotes \v{C}ech cohomology, $\underline{k}^{*}$ is the sheaf of smooth $k^{*}$-valued functions and the last group is the torsion subgroup of $\check{\mathrm{H}}^2(X,\underline{k}^{*})$. Moreover, this bijection becomes an isomorphism of groups upon defining the group structure on the right hand side by 
\begin{equation} \label{GroupStructureDonovanKaroubi}
(\alpha_0,\alpha_1,\alpha_2) \cdot (\beta_0,\beta_1,\beta_2) := (\alpha_0+\beta_0,\alpha_1+\beta_1,(-1)^{\alpha_1\cup\beta_1}\alpha_2\beta_2)\text{.}
\end{equation}
\end{theorem}

\begin{remark}
\begin{enumerate}[(1)]

\item
The cup product is viewed here as a map
\begin{equation*}
{\mathrm{H}}^1(X,\Z_2) \times {\mathrm{H}}^1(X,\Z_2) \to {\mathrm{H}}^2(X,\Z_2) \to \mathrm{Tor}(\check{\mathrm{H}}^2(X,\underline{k}^{*})) 
\end{equation*} 
with the second arrow induced by the inclusion $\Z_2 \to k^{*}: x \mapsto (-1)^{x}$. 

\item 
We have $\mathrm{Tor}(\check{\mathrm{H}}^2(X,\underline{\C}^{*}))=\mathrm{Tor}(\mathrm{H}^3(X,\Z))$ and $\mathrm{Tor}(\check{\mathrm{H}}^2(X,\underline{\R}^{*}))=\mathrm{H}^2(X,\Z_2)$.

\end{enumerate}
\end{remark}
\begin{example}
Let $\mathcal{V}$ be a Riemannian vector bundle over $X$ with typical fibre a Euclidean vector space $V$, let $\Cl(\mathcal{V})$ be the associated bundle of Clifford algebras, and let $\CCl(\mathcal{V}) := \Cl(\mathcal{V}) \otimes_{\R}\C$ be its complexification. Both bundles, $\Cl(\mathcal{V})$ and $\CCl(\mathcal{V})$, are invertible by \cref{prop:dualizabilityalgebrabundles}, and their classes under the classification of \cref{th:classalg} are  $([\Cl(V)], w_1(\mathcal{V}),w_2(\mathcal{V}))$ and $([\CCl(V)], w_1(\mathcal{V}),W_3(\mathcal{V}))$ respectively, where $w_1$ and $w_2$ are the Stiefel-Whitney classes and $W_3$ is the third integral Stiefel-Whitney class \cite[Lemma 7]{DK70}.
\end{example}

\begin{remark}
Ungraded algebra bundles are again treated as a special case of super algebra bundles concentrated in even degrees. 
They form a bicategory $\AlgBdlbi_k(X)$.
The subcategories $\ssAlgBdlbi_k(X)$ and $\AlgBdlbigrpd kX$ are symmetric monoidal and framed by $\ssAlgBdl_k(X)$ and $\AlgBdlgrpd kX$, respectively. 
The invertible objects form a full sub-bicategory $\csAlgBdlbi_k(X)$. The ungraded version of \cref{th:classalg} was also proved by Donovan-Karoubi \cite[Theorems 3 and 8]{DK70}, and gives a group isomorphism
\begin{equation*}
\mathrm{h}_0(\csAlgBdlbi_k(X)) \cong \mathrm{H}^0(X,\mathrm{Br}_k)\times \mathrm{Tor}(\check{\mathrm{H}}^2(X,\underline{k}^{*}))\text{,}
\end{equation*}
with the direct product group structure on the right hand side.
Here, $\mathrm{Br}_k$ denotes the Brauer group of the field $k$.
\end{remark}

\subsection{The 2-stack of 2-vector bundles}
\label{sec:2stack2vec}

One of the most important features any notion of bundles should have, is that bundles can be glued together from locally defined pieces.
In other words, bundles satisfy \emph{descent}, or form a \emph{stack}. In a bicategorical setting, there is a corresponding notion of a \emph{2-stack}. 
Super algebra bundles and bimodule bundles as in  \cref{DefinitionPrestwoVectBdl} form a \emph{pre-2-stack}, but in general not a 2-stack. In this section, we will explain how to solve this issue.
For a general proper discussion of (pre-)2-stacks we refer to \cite{nikolaus2}.

First of all, super algebra bundles and bimodule bundles can be pulled back along smooth maps in a coherent fashion. This can be phrased by saying that the assignment 
\begin{equation*}
X\mapsto \sAlgBdlbi_k(X)
\end{equation*}
forms a \emph{presheaf of bicategories} over the category of smooth manifolds. We will denote this presheaf by $\sAlgBdlbi_k$.
It has two important sub-presheaves:
\begin{enumerate}

\item
The sub-presheaf $\sssAlgBdlbi_k$ where only the semisimple super algebra bundles are admitted. This is a presheaf of symmetric monoidal framed bicategories. 

\item
The further sub-presheaf $\cssAlgBdlbi_k$, where only the central simple super algebra bundles are admitted.

\end{enumerate}
Further, we have corresponding versions of ungraded presheaves of bicategories.

\begin{remark}
To make the notion of bicategorical descent a bit more explicit, let us look for the moment at the presheaf $\AlgBdlbi_{\C}$ of ungraded, complex algebra bundles, bimodule bundles, and intertwiners.
Consider an open cover $\{U_{\alpha} \}_{\alpha}$ of a smooth manifold $X$.
Suppose a family $\mathcal{A}_{\alpha}$ of algebra bundles over $U_{\alpha}$ is given, together with invertible $\mathcal{A}_{\beta}$-$\mathcal{A}_{\alpha}$-bimodule bundles $\mathcal{M}_{\alpha\beta}$ over $U_{\alpha}\cap U_{\beta}$, and invertible intertwiners $\mu_{\alpha\beta\gamma}: \mathcal{M}_{\beta\gamma} \otimes_{\mathcal{A}_{\beta}} \mathcal{M}_{\alpha\beta} \to \mathcal{M}_{\alpha\gamma}$ over $U_{\alpha}\cap U_{\beta}\cap U_{\gamma}$ that satisfy the obvious associativity condition over 4-fold intersections.
For the presheaf $\AlgBdlbi_{\C}$ to be a 2-stack, there must exist a globally defined algebra bundle $\mathcal{A}$ together with invertible $\mathcal{A}|_{U_\alpha}$-$\mathcal{A}_\alpha$-bimodule bundles $\mathcal{N}_\alpha$ and invertible intertwiners 
\begin{equation*}
  \varphi_{\alpha \beta} : \mathcal{N}_\beta \otimes_{\mathcal{A}_{\beta}} \mathcal{M}_{\alpha \beta} \longrightarrow \mathcal{N}_\alpha
\end{equation*}
that are compatible with the isomorphisms $\mu_{\alpha \beta \gamma}$ in an obvious way.
However, this is not necessarily the case (see \cref{re:notastack}).
Our current knowledge is the following:
\begin{itemize}

\item 
The presheaves $\cssAlgBdlbi_{\R}$ and $\csAlgBdlbi_{\R}$ are 2-stacks; this will be proved in \cref{co:csgastack}.

\item
The presheaves $\AlgBdlbi_{\C}$, $\sAlgBdlbi_{\C}$, $\csAlgBdlbi_{\C}$, and $\cssAlgBdlbi_{\C}$ are not  2-stacks; this will be proved in \cref{re:notastack}.

\item
Currently, we do not know whether or not the presheaves $\AlgBdlbi_{\R}$ and $\sAlgBdlbi_{\R}$ are 2-stacks.

\end{itemize}
To impose order on this chaos, we choose to 2-stackify \emph{all} these pre-2-stacks, even those that are 2-stacks already.
In those cases where we already have a 2-stack, 2-stackification gives an equivalent 2-stack with more objects, and it turns out that these additional objects are often useful.
\end{remark}

For the 2-stackification of presheaves $\mathscr{F}$ of bicategories we use the plus construction $\mathscr{F} \mapsto \mathscr{F}^{+}$ of Nikolaus-Schweigert \cite{nikolaus2}. 

\begin{definition}[2-vector bundle]
\label{def:applplus}
We define the following presheaves of bicategories of super 2-vector bundles:
\begin{itemize}

\item 
$\stwoVectBdl_{k} := (\sAlgBdlbi_k)^{+}$ is the presheaf of \emph{super 2-vector bundles} over $k$.

\item 
$\ssstwoVectBdl_{k} := (\sssAlgBdlbi_k)^{+}$ is the presheaf of \emph{semisimple super 2-vector bundles} over $k$.

\item 
$\stwoLineBdl_{k} := (\cssAlgBdlbi_k)^{+}$ is the presheaf of \emph{super 2-line bundles} over $k$.

\end{itemize}
Additionally, we define the following presheaves of bicategories of (ungraded) 2-vector bundles:
\begin{itemize}

\item 
$\twoVectBdl_{k} := (\AlgBdlbi_k)^{+}$ is the presheaf of \emph{2-vector bundles} over $k$.

\item 
$\sstwoVectBdl_{k} :=( \ssAlgBdlbi_k)^{+}$ is the presheaf of \emph{semisimple 2-vector bundles} over $k$.

\item 
$\twoLineBdl_{k} := (\csAlgBdlbi_k)^{+}$ is the presheaf of \emph{2-line bundles} over $k$.

\end{itemize}
\end{definition}

Below we spell out explicitly all details of the plus construction. Before that, we shall state its main purpose.

\begin{theorem} \label{ThmAllisStack}
All presheaves of bicategories defined in \cref{def:applplus} are 2-stacks.
\end{theorem}

\begin{proof}
This follows from \cite[Thm.~3.3]{nikolaus2}, whose only condition is that the presheaves are pre-2-stacks, which we proved in \cite[Prop.~4.5.1]{Kristel2022}.
\end{proof}

For the convenience of the reader (and the authors) we shall now spell out all definitions in case of the bicategory $\stwoVectBdl_k(X)$ of super 2-vector bundles over $X$, on the basis of the description of the plus construction given in \cite{nikolaus2}. The other versions of 2-vector bundles can then easily be obtained as sub-bicategories: 
\begin{itemize}

\item 
$\ssstwoVectBdl_k(X)\subset \stwoVectBdl_k(X)$ is the sub-bicategory where all super algebra bundles are bundles of semisimple super algebras.

\item 
$\stwoLineBdl_k(X)\subset \stwoVectBdl_k(X)$ is the sub-bicategory where all super algebra bundles are bundles of central simple super algebras.

\item 
The corresponding ungraded versions consist of only ungraded algebras and ungraded bimodules.

\end{itemize}
This way, our explanations below apply to all four cases by employing the corresponding restrictions.

\bigskip

\begin{description}
\item[I.) Objects.] 
A \emph{super 2-vector bundle} over $X$ is a quadruple $\mathscr{V}=(\pi,\mathcal{A},\mathcal{M},\mu)$ consisting of:
\begin{itemize}

\item 
a surjective submersion $\pi:Y \to X$,

\item
a super algebra bundle $\mathcal{A}$ over $Y$,

\item
an invertible bimodule bundle $\mathcal{M}$ over $Y^{[2]}$ whose fibre $\mathcal{M}_{y_1,y_2}$ over a point $(y_1,y_2)\in Y^{[2]}$ is an $\mathcal{A}_{y_2}$-$\mathcal{A}_{y_1}$-bimodule, and

\item
an invertible even intertwiner $\mu$ of bimodule bundles over $Y^{[3]}$, which restricts over each point $(y_1,y_2,y_3)\in Y^{[3]}$ to an $\mathcal{A}_{y_3}$-$\mathcal{A}_{y_1}$-intertwiner 
\begin{equation*}
\mu_{y_1,y_2,y_3}: \mathcal{M}_{y_2,y_3} \otimes_{\mathcal{A}_{y_2}} \mathcal{M}_{y_1,y_2} \to \mathcal{M}_{y_1,y_3}\text{.}
\end{equation*}

\end{itemize}
This structure is subject to the condition that 
  $\mu$ is \emph{associative}, \ie the diagram
\begin{equation*}
\xymatrix@C=5em{\mathcal{M}_{y_3,y_4} \otimes_{\mathcal{A}_{y_3}} \mathcal{M}_{y_2,y_3} \otimes_{\mathcal{A}_{y_2}} \mathcal{M}_{y_1,y_2} \ar[d]_{1\otimes \mu_{y_1,y_2,y_3}} \ar[r]^-{\mu_{y_2,y_3,y_4} \otimes 1} &  \mathcal{M}_{y_2,y_4}  \otimes_{\mathcal{A}_{y_2}} \mathcal{M}_{y_1,y_2} \ar[d]^{\mu_{y_1,y_2,y_3}} \\ \mathcal{M}_{y_3,y_4}  \otimes_{\mathcal{A}_{y_3}} \mathcal{M}_{y_1,y_3} \ar[r]_-{\mu_{y_1,y_3,y_4}} & \mathcal{M}_{y_1,y_4} }
\end{equation*}
is commutative for all $(y_1,y_2,y_3,y_4)\in Y^{[4]}$.

We remark the following additional facts, which are easy to deduce: first, if $\Delta:Y \to Y^{[2]}$ denotes the diagonal map, then there exists a canonical invertible intertwiner $\Delta^{*}\mathcal{M} \cong \mathcal{A}$ of $\mathcal{A}$-$\mathcal{A}$-bimodule bundles over $Y$. Second, if $s: Y^{[2]} \to Y^{[2]}$ swaps the factors, then $s^{*}\mathcal{M}$ is inverse to $\mathcal{M}$. 

The simplicial structure of a super 2-vector bundle may be depicted as the following diagram:
\begin{equation*}
\mathscr{V}= \left ( 
\begin{aligned}
\xymatrix{
\mathcal{A} \ar[d] & \mathcal{M} \ar[d] & \mu \ar@{..}[d] & \footnotesize{\text{coherence}} \ar@{..}[d] \\ 
Y \ar[d] & Y^{[2]} \ar@<0.5ex>[l]^{\pr_2}\ar@<-0.5ex>[l]_{\pr_1} & Y^{[3]} \ar@<1ex>[l]\ar@<-1ex>[l]\ar[l] & Y^{[4]} \ar@<1.5ex>[l] \ar@<-1.5ex>[l] \ar@<0.5ex>[l] \ar@<-0.5ex>[l]\\ 
X
} 
\end{aligned}
\right )
\end{equation*}

\item[II.) 1-morphisms. ]
Let $\mathscr{V}_1=(\pi_1,\mathcal{A}_1,\mathcal{M}_1,\mu_1)$ and $\mathscr{V}_2=(\pi_2,\mathcal{A}_2,\mathcal{M}_2,\mu_2)$ be super 2-vector bundles over $X$. A \emph{1-morphism}
$\mathscr{P}:\mathscr{V}_1 \to \mathscr{V}_2$
is a triple $\mathscr{P}=(\zeta,\mathcal{P},\phi)$ consisting of:
\begin{itemize}

\item 
a surjective submersion $\zeta: Z \to Y_1 \times_X Y_2$,

\item
an implementing bimodule bundle $\mathcal{P}$ over $Z$, whose fibre $\mathcal{P}_z$ over a point $z\in Z$ with $\zeta(z)=:(y_1,y_2)$ is an $(\mathcal{A}_2)_{y_2}$-$(\mathcal{A}_1)_{y_1}$-bimodule, and 

\item
an invertible even intertwiner $\phi$ of bimodule bundles over $Z^{[2]}$, which restricts over a point  $(z,z')\in Z^{[2]}$ with $\zeta(z)=:(y_1,y_2)$ and $\zeta(z')=:(y_1',y_2')$ to an $(\mathcal{A}_2)_{y_2'}$-$(\mathcal{A}_1)_{y_1}$-intertwiner
\begin{equation*}
\phi_{z,z'}:  \mathcal{P}_{z'} \otimes_{(\mathcal{A}_1)_{y_1'}} (\mathcal{M}_1)_{y_1,y_1'} \to (\mathcal{M}_2)_{y_2,y_2'} \otimes_{(\mathcal{A}_2)_{y_2}} \mathcal{P}_{z} \text{.}
\end{equation*}
%\begin{equation*}
%\phi_{z,z'}: (\mathcal{M}_2)_{y_2,y_2'} \otimes_{(\mathcal{A}_2)_{y_2}} \mathcal{P}_{z} \to \mathcal{P}_{z'} \otimes_{(\mathcal{A}_1)_{y_1'}} (\mathcal{M}_1)_{y_1,y_1'}\text{.}
%\end{equation*}

\end{itemize}
This structure is subject to the condition that the intertwiner $\phi$ is a \quot{homomorphism} with respect to the intertwiners $\mu_1$ and $\mu_2$, \ie the diagram

\begin{equation}\label{diag:1Morphisms}
\begin{gathered}
\xymatrix@C=8em{
\mathcal{P}_{z^{\prime\prime}} \otimes_{(\mathcal{A}_1)_{y_1^{\prime\prime}}}(\mathcal{M}_1)_{y_1',y_1''} \otimes_{(\mathcal{A}_1)_{y_1'}} (\mathcal{M}_1)_{y_1,y_1'}  \ar[r]^-{ \id \otimes (\mu_1)_{y_1,y_1',y_1''}} \ar[d]_{\phi_{z^\prime,z^{\prime\prime}}\otimes \id} &
\mathcal{P}_{z^{\prime\prime}} \otimes_{(\mathcal{A}_1)_{y_1}} (\mathcal{M}_1)_{y_1,y_1''}   \ar[dd]^{\phi_{z,z''}} \\ 
 (\mathcal{M}_2)_{y_2',y_2''} \otimes_{(\mathcal{A}_2)_{y_2'}} \mathcal{P}_{z'} \otimes_{(\mathcal{A}_1)_{y_1'}} (\mathcal{M}_1)_{y_2,y_2'} \ar[d]_{\id \otimes \phi_{z,z'}} & \\ 
  (\mathcal{M}_2)_{y_2',y_2''} \otimes_{(\mathcal{A}_2)_{y_2'}} (\mathcal{M}_2)_{y_2,y_2'}  \otimes_{(\mathcal{A}_2)_{y_2}} \mathcal{P}_{z} \ar[r]_-{(\mu_2)_{y_2,y_2',y_2''} \otimes \id} & 
 (\mathcal{M}_2)_{y_2,y_2''} \otimes_{(\mathcal{A}_2)_{y_2}} \mathcal{P}_{z} 
 }
\end{gathered}
\end{equation}
%\begin{equation}\label{diag:1Morphisms}
%\begin{gathered}
%\xymatrix@C=6em{(\mathcal{M}_2)_{y_2',y_2''} \otimes_{(\mathcal{A}_2)_{y_2'}} (\mathcal{M}_2)_{y_2,y_2'} \otimes_{(\mathcal{A}_2)_{y_2}} \mathcal{P}_{z} \ar[r]^-{(\mu_2)_{y_2,y_2',y_2''}\otimes \id} \ar[d]_{\id \otimes \phi_{z,z'}} & (\mathcal{M}_2)_{y_2,y_2''} \otimes_{(\mathcal{A}_2)_{y_2}} \mathcal{P}_{z} \ar[dd]^{\phi_{z,z''}} \\ (\mathcal{M}_2)_{y_2',y_2''} \otimes_{(\mathcal{A}_2)_{y_2'}} \mathcal{P}_{z'} \otimes_{(\mathcal{A}_1)_{y_1'}} (\mathcal{M}_1)_{y_1,y_1'} \ar[d]_{\phi_{z',z''} \otimes  \id} & \\ \mathcal{P}_{z''} \otimes_{(\mathcal{A}_1)_{y_1''}} (\mathcal{M}_1)_{y_1',y_1''} \otimes_{(\mathcal{A}_1)_{y_1'}} (\mathcal{M}_1)_{y_1,y_1'} \ar[r]_-{\id \otimes (\mu_1)_{y_1,y_1',y_1''}} & \mathcal{P}_{z''} \otimes_{(\mathcal{A}_1)_{y_1''}} (\mathcal{M}_1)_{y_1,y_1''}}
%\end{gathered}
%\end{equation}
is commutative for all $(z,z',z'')\in Z^{[3]}$, where $\zeta(z'')=: (y_1'',y_2'')$. 
\item[III.) Identity 1-morphisms.]
The \emph{identity 1-morphism} of a super 2-vector bundle  $\mathscr{V}=(\pi,\mathcal{A},\mathcal{M},\mu)$  is the triple $\id_{\mathscr{V}} :=(\id_{Y^{[2]}},\mathcal{M},\phi_{\mu})$, where $Z:= Y^{[2]}$ and 
\begin{equation*}
(\phi_{\mu})_{y_1,y_2,y_1',y_2'} := \mu^{-1}_{y_1,y_1',y_2'}\circ \mu_{y_1,y_2,y_2'}\text{.}
\end{equation*}

\item[IV.) Composition of 1-morphisms.]
Consider three super 2-vector bundles and two 1-morphisms 
\begin{equation*}
\xymatrix{\mathscr{V}_1 \ar[r]^{\mathscr{P}_{12}} & \mathscr{V}_2 \ar[r]^{\mathscr{P}_{23}} & \mathscr{V}_3\text{,}}
\end{equation*}
for whose structure we use the same letters as in above definitions. The \emph{composition} $\mathscr{P}_{23} \circ \mathscr{P}_{12} := (\zeta, \mathscr{P},\phi)$ is defined as follows. We set $Z:= Z_{12} \times_{Y_2} Z_{23}$, and consider a point $z:=(z_{12},z_{23})\in Z$ with  $\zeta_{12}(z_{12})=:(y_1,y_2)$ and $\zeta_{23}(z_{23})=:(y_2,y_3)$. The surjective submersion  $\zeta:Z \to Y_1 \times_X Y_3$ is then $\zeta(z) := (y_1,y_3)$. The bimodule bundle $\mathscr{P}$ is defined so that its fibre   over the point   $z\in Z$ is the $(\mathcal{A}_3)_{y_3}$-$(\mathcal{A}_1)_{y_1}$-bimodule
\begin{equation*}
\mathcal{P}_{z}:=(\mathcal{P}_{23})_{z_{23}} \otimes_{(\mathcal{A}_2)_{y_2}} (\mathcal{P}_{12})_{z_{12}}\text{.}
\end{equation*}
Finally, the intertwiner $\phi$ is over a point $((z_{12},z_{23}),(z_{12}',z_{23}'))\in Z^{[2]}$defined by 
\begin{equation*}
\xymatrix@C=4em{
\hspace{-10em}\mathcal{P}_{z_{12}',z_{23}' } \otimes_{(\mathcal{A}_1)_{y_1}} \mathcal{M}_1=(\mathcal{P}_{23})_{z'_{23}} \otimes_{(\mathcal{A}_2)_{y_2'}} (\mathcal{P}_{12})_{z_{12}'} \otimes_{(\mathcal{A}_1)_{y_1}} \mathcal{M}_1\ar[d]^-{\id\circ( \phi_{12})_{z_{12},z_{12}'}} \\
(\mathcal{P}_{23})_{z_{23}'}\otimes_{(\mathcal{A}_2)_{y_2'}} \mathcal{M}_2 \otimes_{(\mathcal{A}_2)_{y_2}} (\mathcal{P}_{12})_{z_{12}}  \ar[d]^-{(\phi_{23})_{z_{23},z_{23}'} \circ \id}
\\
\mathcal{M}_3 \otimes_{(\mathcal{A}_3)_{y_3}} (\mathcal{P}_{23})_{z_{23}} \otimes_{(\mathcal{A}_2)_{y_2}} (\mathcal{P}_{12})_{z_{12}}=\mathcal{M}_3 \otimes_{(\mathcal{A}_3)_{y_3}} \mathcal{P}_{z_{12},z_{23}}\text{.}\hspace{-10em}  }
\end{equation*}

\item[V.) 2-morphisms.]
Consider two 1-morphisms between the same super 2-vector bundles,
\begin{equation*}
\xymatrix{\mathscr{V}_1 \ar@/^1pc/[r]^{\mathscr{P}}\ar@/_1pc/[r]_{\mathscr{P}'} & \mathscr{V}_2\text{.}}
\end{equation*}

For abbreviation, we set $Y_{12} := Y_1 \times_X Y_2$. 
Then, a \emph{2-morphism} $\mathscr{P} \Rightarrow \mathscr{P}'$ is represented by pairs $(\rho,\varphi)$ consisting of the following structure:
\begin{itemize}

\item 
a surjective submersion $\rho: W \to Z \times_{Y_{12}} Z'$, and

\item
an intertwiner $\varphi$ of bimodule bundles over $W$ that restricts over a point $w\in W$ with $\rho(w)=:(z,z')$ to an intertwiner
\begin{equation*}
\varphi_w: \mathcal{P}_{z} \to \mathcal{P}'_{z'}
\end{equation*}
of $(\mathcal{A}_1)_{y_1}$-$(\mathcal{A}_2)_{y_2}$-bimodules, where $\zeta(z)=\zeta'(z')=:(y_1,y_2)$. \end{itemize}
This structure is subject to the condition that $\varphi$ commutes with the intertwiners $\phi$ and $\phi'$, \ie the diagram
\begin{equation*}
\xymatrix@C=4em{
\mathcal{P}_{\tilde z} \otimes_{(\mathcal{A}_1)_{\tilde y_1}} (\mathcal{M}_1)_{y_1,\tilde y_1} \ar[d]_{\varphi_{\tilde w} \otimes \id} \ar[r]^-{\phi_{z,\tilde z}} & (\mathcal{M}_2)_{y_2,\tilde y_2} \otimes_{(\mathcal{A}_2)_{y_2}} \mathcal{P}_{z} \ar[d]^{\id \otimes \varphi_w}
 \\ 
 \mathcal{P}'_{\tilde z'} \otimes_{(\mathcal{A}_1)_{\tilde y_1'}} (\mathcal{M}_1)_{y_1',\tilde y_1'}  \ar[r]_-{\phi'_{z',\tilde z'}} &(\mathcal{M}_2)_{y_2',\tilde y_2'} \otimes_{(\mathcal{A}_2)_{y'_2}} \mathcal{P}'_{z'}
 }
\end{equation*}
is commutative for all $(w,\tilde w) \in W \times_X W$, where $\rho(w)=(z,z')$, $\rho(\tilde w)=(\tilde z, \tilde z')$.
Two pairs $(\rho,\varphi)$ and $(\rho',\varphi')$ have to be identified if the pullbacks of $\varphi$ and $\varphi'$ coincide over $W \times_{Z \times_{Y_{12}} Z'} W'$. 
Since in that sense the choice of $\rho$ is unimportant, we usually denote the 2-morphism by just $\varphi$.

\item[VI.) Vertical composition of 2-morphisms.]

Next, consider three 1-morphisms between the same super 2-vector bundles, and two 2-morphisms:
\begin{equation*}
\xymatrix@C=5em{\mathscr{V}_1 \ar[r]|*+{\mathscr{P}'}="2" \ar@/^2.5pc/[r]^*+{\mathscr{P}}="1"\ar@/_2.5pc/[r]_*+{\mathscr{P}''}="3" & \mathscr{V}_2\text{.} \ar@{=>}"1";"2"|-{\varphi} \ar@{=>}"2";"3"|-{\varphi'}}
\end{equation*}
We suppose that the 1-morphisms $\mathscr{P}$, $\mathscr{P}'$, and $\mathscr{P}''$ come with surjective submersions $\zeta:Z \to Y_{12}$, $\zeta': Z'\to Y_{12}$, and $\zeta'': Z'' \to Y_{12}$, respectively, and that the 2-morphisms $\varphi$ and $\varphi'$ come with surjective submersions $\rho: W \to Z \times_{Y_{12}} Z'$ and $\rho': W' \to Z' \times_{Y_{12}} Z''$, respectively. 
We consider $W \times_{Z'} W'$ equipped with the surjective submersion $(w,w') \mapsto (z,z'')$, where $(z,z'):=\rho(w)$ and $(z',z'') := \rho'(w')$.  Then the vertical composition $\varphi' \bullet \varphi: \mathscr{P} \Rightarrow \mathscr{P}''$ is the intertwiner over $W \times_{Z'} W'$ defined fibrewise over $(w,w')$ by
\begin{equation*}
\xymatrix{\mathcal{P}_{z} \ar[r]^{\varphi_{w}} & \mathcal{P}'_{z'} \ar[r]^{\varphi'_{w'}} & \mathcal{P}''_{z''}}\text{.}
\end{equation*}

\item[VII.) Identity 2-morphisms.]
The identity 2-morphism $\id_{\mathscr{P}}$ of a 1-morphism $\mathscr{P}=(\zeta,\mathcal{P},\phi)$ is obtained by restricting the intertwiner $\phi$ to $W:=Z \times_{Y_{12}} Z \subset Z \times_X Z$. Over $(z_1,z_2)$ with $\zeta(z_1)=\zeta(z_2)=(y_1,y_2)$, this becomes an intertwiner, 
\begin{equation*}
\mathcal{P}_{z_1}\otimes_{(\mathcal{A}_1)_{y_1}} (\mathcal{M}_1)_{y_1,y_1} \to (\mathcal{M}_2)_{y_2,y_2} \otimes_{(\mathcal{A}_2)_{y_2}} \mathcal{P}_{z_2} \text{.}
\end{equation*}
Under the canonical invertible intertwiners $(\mathcal{A}_1)_{y_1} \cong (\mathcal{M}_1)_{y_1,y_1}$ and $(\mathcal{A}_2)_{y_2} \cong (\mathcal{M}_1)_{y_2,y_2}$, this yields  an intertwiner $\varphi$ with $\varphi_{z_1,z_2}: \mathcal{P}_{z_1} \to \mathcal{P}_{z_2}$, and the pair $(\id_W,\varphi)$ is the identity 2-morphism of $\mathscr{P}$.
\begin{comment}
The following shows that this is really an identity 2-morphism: observe that the diagram
\begin{equation*}
\xymatrix{\mathcal{P}_{z} \ar[r]^{\varphi_{z, z'}} \ar@/^2.5pc/[rr]^{\varphi_{z, z''}} & \mathcal{P}_{z'} \ar[r]^{\varphi_{z', z''}} & \mathcal{P}_{z''}}
\end{equation*}
 is commutative because it is the restriction on diagram \cref{diag:1Morphisms} to $Z \times_{Y_{12}} Z \times_{Y_{12}} Z\subset Z^{[3]}$, taking into account that the occurrences of $\mu_1$ and $\mu_2$ there disappear under the canonical intertwiners $\mathcal{A}_i\cong \Delta^{*}\mathcal{M}_i$.
\end{comment}

\item[VIII.) Horizontal composition of 2-morphisms.]

Consider the following super 2-vector bundles, 1-morphisms, and 2-morphisms,
\begin{equation*}
\xymatrix@C=6em{\mathscr{V}_1 \ar@/^1.5pc/[r]^-{\mathscr{P}_{12}}="1"\ar@/_1.5pc/[r]_{\mathscr{P}_{12}'}="2" & \mathscr{V}_2 \ar@/^1.5pc/[r]^-{\mathscr{P}_{23}}="3" \ar@/_1.5pc/[r]_-{\mathcal{P}_{23}'}="4" &\mathscr{V}_3\text{,} \ar@{=>}"1";"2"|{\varphi_{12}}\ar@{=>}"3";"4"|{\varphi_{23}}}
\end{equation*}
with all structure labelled as above. The horizontal composition of $\varphi_{12}$ and $\varphi_{23}$, denoted $\varphi_{23}\circ \varphi_{12}$, is defined by the smooth manifold $W:=W_{12} \times_{Y_2} W_{23}$ equipped with the surjective submersion $(w_{12},w_{23})\mapsto ((z_{12},z_{23}),(z_{12}',z_{23}'))$, where $(z_{12},z_{12}'):= \rho_{12}(w_{12})$ and $(z_{23},z_{23}'):=\rho_{23}(w_{23})$, and the intertwiner of bimodule bundles over $W$, given in the fibre over $(w_{12},w_{23})\in W$ by
\begin{equation*}
\xymatrix{(\varphi_{23})_{w_{23}} \otimes (\varphi_{12})_{w_{12}}:(\mathcal{P}_{23})_{z_{23}} \otimes_{(\mathcal{A}_2)_{y_2}} (\mathcal{P}_{12})_{z_{12}} \to (\mathcal{P}_{23}')_{z_{23}'}\otimes_{(\mathcal{A}_2)_{y_2}} (\mathcal{P}_{12}')_{z_{12}'}\text{.}}
\end{equation*}   

\end{description}
This completes the explicit description of super 2-vector bundles.
We close this section with four useful technical results about super 2-vector bundles, which follow directly from the plus construction \cite{nikolaus2}, and are well known, e.g.,  for bundle gerbes.

\begin{lemma}\
\label{LemmaInvertibility}
\begin{enumerate}[ (a)]

\item \label{LemmaInvertibilityA}
A 1-morphism $\mathscr{P}=(\zeta,\mathcal{P},\phi)$ is invertible if and only if its bimodule bundle $\mathcal{P}$ is invertible. 

\item \label{LemmaInvertibilityA2}
A 1-morphism  $\mathscr{P}=(\zeta,\mathcal{P},\phi)$ has a right (left) adjoint if and only if its bimodule bundle $\mathcal{P}$ has a right (left) adjoint. 

\item \label{LemmaInvertibilityB}
A 2-morphism $\varphi$ is invertible if and only if its intertwiner  $\varphi$ is  invertible. 

\end{enumerate}  
\end{lemma}

\begin{remark}
Bimodule bundles are invertible (have adjoints) if and only if they are fibrewise invertible (have fibrewise adjoints) \cite[Lem.~4.2.8]{Kristel2022}.
Thus, the conditions in \cref{LemmaInvertibility} can be checked fibrewise. 
\end{remark}

\begin{remark}
\cref{LemmaInvertibility} also has the following consequence. We may consider as in \cref{sec:algebrabundles} the 
 sub-presheaf $\sAlgBdlbigrpd k-$ where only invertible bimodule bundles and invertible intertwiners are admitted. Applying the plus construction, we see by \cref{LemmaInvertibilityB} that all resulting 2-morphisms are invertible, and we see by \cref{LemmaInvertibilityA} that all resulting 1-morphisms are invertible. Thus, we have
\begin{equation*}
\sAlgBdlgrpd k-^{+} = \stwoVectBdlgrpd k-
\quand
\AlgBdlgrpd k-^{+} = \twoVectBdlgrpd k-\text{.}
\end{equation*}
In other words, it does not matter if we truncate to 2-groupoids before or after 2-stackification. 
\end{remark}

Our second result shows that the surjective submersions of 1-morphisms and 2-morphisms can be assumed to be identities.

\begin{lemma}\
\label{lem:canonicalrefinements}
\begin{enumerate}[ (a)]

\item
\label{lem:canonicalrefinements:a}
Every 1-morphism is  2-isomorphic to one with $Z=Y_1 \times_X Y_2$ and $\zeta=\id_{Z}$.

\item 
\label{lem:canonicalrefinements:b}
Every 2-morphism can be represented by a pair $(\rho,\varphi)$ with $W= Z \times_{Y_{12}} Z'$ and $\rho=\id_W$.

\end{enumerate}
\end{lemma}

\Cref{lem:canonicalrefinements} makes use of the fact that $\sAlgBdlbi_k$ is a pre-2-stack. The reason that our definitions above allow for general $Z$ and $W$ is that all kinds of compositions result in such more general choices, and in practice it is often easier to keep those instead of performing descent. 
Our third result allows to refine the surjective submersion of a super 2-vector bundles without changing its isomorphism class.

\begin{lemma}
\label{lem:refinementofss}
If $\mathscr{V}=(\pi,\mathcal{A},\mathcal{M},\mu)$ is a super 2-vector bundle with surjective submersion $\pi:Y \to X$, and $\rho:Y' \to Y$ is a smooth map such that $\pi':= \pi\circ\rho$ is again a surjective submersion, then $\mathscr{V}^{\rho}:=(\pi',\rho^{*}\mathcal{A},(\rho^{[2]})^{*}\mathcal{M},(\rho^{[3]})^{*}\mu)$ is another super 2-vector bundle, and there exists a canonical isomorphism $\mathscr{V}^{\rho}\cong \mathscr{V}$. 
\end{lemma}

The canonical isomorphism $\mathscr{V}^{\rho} \to \mathscr{V}$ is best viewed in terms of the framing of the bicategory of super 2-vector bundles that we introduce below in \cref{sec:framing}. A consequence is the following result, referring to the triviality of super algebra bundles and bimodule bundles defined in \cref{re:typicalfibre,ExampleTwistedModuleIso}.
\begin{proposition}
\label{prop:trivialbundles}
Every super 2-vector bundle is isomorphic to one with trivial super algebra bundle  and trivial bimodule bundle.
\end{proposition}

\begin{proof}
We first show that the super algebra bundle can assumed to be trivial. 
Let $\mathscr{V}=(\pi,\mathcal{A},\mathcal{M},\mu)$ be a  super 2-vector bundle. Let $\{U_{\alpha}\}$ be an open cover of $X$ that admits local sections $\sigma_\alpha: U_{\alpha} \to Y$ such that $\sigma_{\alpha}^{\mathcal{*}}\mathcal{A}\cong U_{\alpha} \times A_{\alpha}$. 
Let $Y'$ be the disjoint union of the open sets $U_{\alpha}$, let $\pi':Y'\to X$ be the canonical projection, and let $\sigma:Y' \to Y$ be given by $\sigma|_{U_{\alpha}}:=\sigma_{\alpha}$. 
Then, we have $\pi'=\pi\circ \sigma$, and by \cref{lem:refinementofss}, $\mathscr{V}$ is isomorphic to a super 2-vector bundle $\mathscr{V}^{\sigma}$ with trivial algebra bundle. 
\\
Now let  $\mathscr{V}=(\pi,\mathcal{A},\mathcal{M},\mu)$ be a  super 2-vector bundle with trivial super algebra bundle $\mathcal{A}$. 
Let $\{U_{\alpha}\}$ be an open cover of $Y^{[2]}$ with connected sets $U_\alpha$ over which the bimodule bundle $\mathcal{M}$ trivializes, i.e., $\mathcal{M}|_{U_{\alpha}}\cong \lli{{\varphi_{\alpha}}}\sheaf{M_{\alpha}}_{\,\psi_{\alpha}}$, where $M_{\alpha}$ is an $A_{i}$-$A_{j}$-bimodule,  $\varphi_{\alpha}: U_{\alpha} \to \Aut(A_i)$, and $\psi_{\alpha}:U_{\alpha} \to \Aut(A_j)$, for $i$ labelling the connecting component of $Y$ containing $\pr_1(U_{\alpha})$, and $j$ labelling the one containing $\pr_2(U_{\alpha})$. 
On paracompact spaces such as manifolds, such a {hypercover of height 1} can be refined by an ordinary cover; \ie there exists an open cover $\{W_{i}\}$ of $X$ with smooth sections $\sigma_{i}:W_{i} \to Y$, such that for each non-trivial double intersection $W_i \cap W_j$ there exists an index $\alpha$ with $\sigma_i(W_i)\times_{X}  \sigma_j(W_j) \subset  U_{\alpha}$. 
Proceeding as above, and using \cref{lem:refinementofss}, $\mathscr{V}$ is isomorphic to a super 2-vector bundle $\mathscr{V}^{\sigma}$ with trivial super algebra bundle and trivial bimodule bundle. 
\end{proof}

\section{Properties of 2-vector bundles}

\label{sec:properties}

In this section we discuss a number of structures and features of 2-vector bundles. We will mostly stick to the 2-stack $\stwoVectBdl_k$ and only consider its sub-2-stacks $\ssstwoVectBdl_k$ and $\stwoLineBdl_k$ when necessary or interesting. 

\subsection{The Morita class of a 2-vector bundle}

\begin{definition}[Morita class]
Let $\mathscr{V}=(\pi,\mathcal{A},\mathcal{M},\mu)$ be a super 2-vector bundle over $X$, and let $A$ be a super algebra. We say that $\mathscr{V}$ \emph{is of Morita class $A$} if around every point $y\in Y$ there exists a local trivialization $\mathcal{A}|_U \cong U \times A_U$ with a super algebra $A_U$ that is Morita equivalent to $A$.  
\end{definition}

In \cite[Def.~4.2.9]{Kristel2022} we have also introduced the notion of a Morita class for super algebra bundles; in that sense, a super 2-vector bundle is of Morita class $A$ if and only if its super algebra bundle $\mathcal{A}$ is of Morita class $A$.
The following result describes the behaviour of the Morita class of a super 2-vector bundle. It exhibits the Morita class as a generalization of the rank of an ordinary vector bundle. 

\begin{lemma}
\label{lem:Moritaclass2vect}
Let $\mathscr{V}$ be a super 2-vector bundle.
\begin{enumerate}[(a)]

\item
\label{lem:Moritaclass2vect:a}
If $X$ is connected, then there exists a super algebra $A$ such that $\mathscr{V}$ is of Morita class $A$. 
In fact, $\mathscr{V}$ is of Morita class $\mathcal{A}_y$, where $\mathcal{A}_y$ is the fibre of $\mathcal{A}$ over any point $y\in Y$.

\item 
Let $A$ and $B$ be super algebras, and suppose $\mathscr{V}$ is of Morita class $A$. 
Then, $\mathscr{V}$ is  of Morita class $B$ if and only if $A$ and $B$ are Morita equivalent. 

\item
Let $\mathscr{V}_1$, $\mathscr{V}_2$ be two super 2-vector bundles with $\mathscr{V}_1\cong \mathscr{V}_1$. Then, for any super algebra $A$, $\mathscr{V}_1$ is of Morita class $A$ if and only if $\mathscr{V}_2$ is of Morita class $A$. 

\end{enumerate}
\end{lemma}

\begin{proof}
(a) For each connected component $Y_i$ of $Y$, there exists by \cite[Lem.~4.2.10 (a)]{Kristel2022} a super algebra $A_{i}$ such that $\mathcal{A}|_{Y_i}$ is of Morita class $A_i$. 
If $Y_i$ and $Y_j$ are components such that $\pi(Y_i)\cap \pi(Y_j)\neq\emptyset$, then there exists a point $(y_i,y_j)\in Y^{[2]}$ with $y_i\in Y_i$ and $y_j\in Y_j$.
 Then $\mathcal{M}_{y_i,y_j}$ is a Morita equivalence between $\mathcal{A}_{y_2}$ and $\mathcal{A}_{y_1}$, showing that $A_{i}$ and $A_{j}$ are Morita equivalent. 
 If $Y_i$ and $Y_j$ are arbitrary connected components, then, since $X$ is connected, there exists a finite sequence $Y_i=Y_{a_1},..., Y_{a_n}=Y_j$ of connected components of $Y$ such that $\pi(Y_{a_k})$ and $\pi(Y_{a_{k+1}})$ intersect. 
 This shows that $\mathcal{A}$ is of Morita class $A_{i}$, for any connected component $Y_j$ of $Y$. 
 \\
 (b) is trivial. 
For (c), suppose that $\mathscr{V}_1$ and $\mathscr{V}_2$ are presented in terms of surjective submersions $Y_1$, respectively $Y_2$.
Let $\mathscr{P}=(\zeta,\mathcal{P},\phi)$ be an isomorphism between $\mathscr{V}_1$ and $\mathscr{V}_2$. 
Then for $z \in Z$, $\mathcal{P}_z$ is a $(\mathcal{A}_1)_{y_1}$-$(\mathcal{A}_2)_{y_2}$-bimodule, where $\zeta(z) = (y_1, y_2)$. 
As $\mathscr{P}$ is an isomorphism, $\mathcal{P}_z$ is an invertible bimodule, by \cref{LemmaInvertibilityA}, \ie a Morita equivalence.
Hence if $\mathscr{V}_1$ is of Morita class $A$ meaning that $(\mathcal{A}_1)_{y_1}$ is Morita equivalent to $A$, then so is $\mathscr{V}_2$.
\end{proof}

In case  of super 2-\emph{line} bundles over $k=\R,\C$, the classification of central simple super algebras implies that the Morita class of a \emph{complex} super 2-line bundle can be either $\CCl_0$ or $\CCl_1$, and the Morita class of a \emph{real} super 2-line bundle can be one of $\Cl_0,...,\Cl_7$.
The following lemma shows that the converse is also true.

\begin{lemma}
\label{lem:2linebundlesandmorita}
Let $\mathscr{V}$ be a super 2-vector bundle of Morita class $\CCl_n$ (for $n=0,1$ and $k=\C$) or of Morita class $\Cl_n$ (for $n=0,...,7$ and $k=\R$). Then, $\mathscr{V}$ is a super 2-line bundle.
\end{lemma}  

\begin{proof}
We discuss the complex case, the real case is analogous. By definition of Morita class of super 2-vector bundles, the super algebra bundle $\mathcal{A}$ of $\mathscr{V}$ is of Morita class $\CCl_n$. Thus, each fibre of $\mathcal{A}$ is a central simple super algebra, and thus, by \cite[Prop.~4.4.3]{Kristel2022}, $\mathcal{A}$ is invertible. This proves that $\mathscr{V}$ is an object in $\stwoLineBdl_k(X)$. \end{proof}

\begin{remark}
The Morita class of an \emph{ungraded} 2-vector bundle is of course an ungraded algebra (\ie a super algebra concentrated in degree zero). 
Conversely, however, if the Morita class of a super 2-vector bundle $\mathscr{V}$ happens to be an ungraded algebra, it is \emph{not} necessarily true that $\mathscr{V}$ lies in the sub-bicategory $\twoVectBdl_k(X)\subset \stwoVectBdl_k(X)$, as it may still have a non-trivially graded bimodule bundle $\mathcal{M}$ over $Y^{[2]}$.
One example where this happens are super line bundle gerbes, considered as super 2-line bundles, see \cref{sec:bundlegerbes}.
\end{remark}

Sometimes it will be convenient to consider only super 2-vector bundles of a fixed Morita class. 

\begin{definition}[2-vector bundles with fixed Morita class]
\label{def:2vboffixedMoritaclass}
Let $A$ be a super algebra over $k$. Then, the presheaf $A\text{-}\stwoVectBdl$ of \emph{super 2-vector bundles of Morita class $A$} is defined to be the full sub-presheaf of $\stwoVectBdl_k$ over all super 2-vector bundles of Morita class $A$.
\end{definition}

\begin{remark}
\label{re:moritaclass} 
\begin{enumerate}[(1)]

\item 
\label{re:moritaclass:a} 
The presheaf $A\text{-}\stwoVectBdl$ is again a 2-stack, since descent preserves the typical fibres of algebra bundles. In fact, we may consider the presheaf of bicategories $A\text{-}\sAlgBdlbi$ with all super algebra bundles whose fibres are Morita equivalent to $A$; then, $A\text{-}\stwoVectBdl = (A\text{-}\sAlgBdlbi)^{+}$.

\item
If $A$ is an ungraded algebra, then a 2-stack $A\text{-}\twoVectBdl$ is obtained in the same way. We remark that, for $A$ ungraded, there is a  functor $A\text{-}\twoVectBdl \to A\text{-}\stwoVectBdl$ that is in general not essentially surjective, as ungraded algebras may have invertible bimodules that are not concentrated in even degrees. 
Graded bundle gerbes provide examples, as follows from the classification result of \cref{prop:classcssa}.

\item
\label{re:moritaclass:c} 
For a smooth manifold $X$, the bicategory $\C\text{-}\twoVectBdl(X)$ is equivalent to the bicategory of Morita bundle gerbes of Ershov \cite{Ershov2016}. Concerning the objects, Ershov only allows ungraded matrix algebra bundles $\mathcal{A}$ over $Y$, which are precisely the algebra bundles of Morita class $\C$. One difference is that Ershov allows only surjective submersions coming from open covers; this, however, is unproblematic in view of \cref{lem:refinementofss}.  In fact, Ershov considers all Morita bundle gerbes, 1-morphisms, and 2-morphisms  with respect to the \emph{same} fixed open cover, which gives an equivalent bicategory to ours when this cover is \emph{good}.

\end{enumerate}
\end{remark}

\subsection{Inclusion of bundle gerbes}

\label{sec:bundlegerbes}

Let $\sVectBdl_k(X)$ be the symmetric monoidal category of super vector bundles over $X$. We denote by $\mathscr{B}\sVectBdl_k(X)$ the corresponding bicategory with a single object. 
Then, $X \mapsto \mathscr{B}\sVectBdl_k(X)$ is a pre-2-stack; this is just a reformulation of the fact that vector bundles form a monoidal stack. Its stackification is, by definition \cite{nikolaus2}, the 2-stack of \emph{super line bundle gerbes},
\begin{equation*}
\sGrb_k := (\mathscr{B}\sVectBdl_k)^{+}\text{.}
\end{equation*}
Super line bundle gerbes can be identified with the twistings of complex K-theory defined by Freed-Hopkins-Teleman \cite{Freed2011a}; the following statement has been proved by Mertsch \cite{Mertsch2020}.

\begin{proposition}
\label{prop:twistings}
The homotopy 1-category $\mathrm{h}_1(\sGrb_\C(X))$ of the bicategory of complex super line bundle gerbes is canonically equivalent to the category of twistings of complex K-theory  defined by Freed-Hopkins-Teleman.
\end{proposition}

Put differently, the presheaf of categories defined by Freed-Hopkins-Teleman  extends to a presheaf of bicategories, and that presheaf is in fact a 2-stack.

Next, we describe the relation between super line bundle gerbes and super 2-vector bundles. We consider the obvious inclusion
\begin{equation}
\label{eq:inclusionBsVect}
\mathscr{B}\sVectBdl_k \to \cssAlgBdlbi_k
\end{equation}
of presheaves of bicategories, taking the single object over a manifold $X$  to the trivially graded super algebra bundle $\underline{k}$ over $X$, and considering super vector bundles as $\underline{k}$-$\underline{k}$-bimodule bundles (they are implementing due to \cref{RemarkHHSemisimple}). This is fully faithful over each smooth manifold $X$. 
By functoriality of the plus construction, we obtain the following result.

\begin{proposition}
\label{prop:inclusionsuperlinebundlegerbes}
The inclusion in \cref{eq:inclusionBsVect} induces under the plus construction a fully faithful morphism
\begin{equation*}
\sGrb_k \to \stwoLineBdl_k
\end{equation*}
of 2-stacks. In other words, super line bundle gerbes form a full sub-bicategory of super 2-line bundles. 
\end{proposition}

Explicitly, the 2-stack morphism of \cref{prop:inclusionsuperlinebundlegerbes} simply adds to the structure of a given super line bundle gerbe the trivial algebra bundle $\underline{k}$ over the domain $Y$ of its surjective submersion.

\begin{example}
\label{ex:trivbundlegerbe1}
The \emph{trivial bundle gerbe} $\mathscr{I}$ (consisting of the trivial surjective submersion $\id_X$, the trivial super line bundle $\mathcal{M}:=\underline{k}$ over $X^{[2]}=X$, and the line bundle isomorphism $\mu:\mathcal{M} \otimes \mathcal{M} \to \mathcal{M}$ induced by multiplication in $k$) corresponds under the 2-functor $\sGrb_k(X) \to \stwoLineBdl_k(X)$ to the \emph{trivial  2-vector bundle} (consisting of the trivial surjective submersion $\id_X$, the trivial super algebra bundle $\mathcal{A}:=\underline{k}$, and the same $\mathcal{M}$ and $\mu$ as before). Since there is no need to distinguish between the trivial bundle gerbe and the trivial 2-vector bundle, we will henceforth denote both by $\mathscr{I}$. \end{example}

\begin{proposition}
\label{lem:bundlegerbesand2vectorbundles}
Let $\mathscr{V}$ be a super-2-vector bundle. Then, $\mathscr{V}$ is isomorphic to a super line bundle gerbe if and only if $\mathscr{V}$ is of Morita class $k$. 
\end{proposition}

\begin{proof}
The \quot{only if}-part is clear. 
Suppose $\mathscr{V}$ is of Morita class $k$. 
Then by \cref{prop:trivialbundles}, $\mathscr{V}$ is isomorphic to a super 2-vector bundle $\mathscr{V}'$ whose super algebra bundle is the trivial bundle $\underline{k}$ over $Y$. 
Its bimodule bundle is then a bundle of invertible $k$-$k$-bimodules, \ie a super line bundle. 
This shows that $\mathscr{V}'$ is a super line bundle gerbe. 
\end{proof}

\Cref{lem:bundlegerbesand2vectorbundles} shows that the fully faithful morphism of \cref{prop:inclusionsuperlinebundlegerbes} is not an isomorphism of 2-stacks, since the Morita class of a general super 2-line bundle may be any central simple super algebra. The relation between super line bundle gerbes and super 2-line bundles is further clarified in \cref{sec:classificationof2linebundles}.

In the contexts of twisted K-theory and 2-dimensional sigma models, one considers 1-morphisms $\mathscr{E}:\mathscr{G} \to \mathscr{I}$ between a super line bundle gerbe $\mathscr{G}$ and the trivial bundle gerbe $\mathscr{I}$ \cite{Bouwknegt2002,gawedzki1}. These are often called {\emph{bundle gerbe modules}} or \emph{$\mathscr{G}$-twisted (super) vector bundles}, see \cite{waldorf1}. An immediate consequence of \cref{prop:inclusionsuperlinebundlegerbes} is a reformulation in terms of morphisms between super 2-line bundles.

\begin{corollary}
\label{co:twistedsupervectorbundles}
Let $\mathscr{G}$ be a super line bundle gerbes over $X$. Then, the category $\Homcat_{\sGrb_{k}(X)}(\mathscr{G},\mathcal{I})$ of $\mathscr{G}$-twisted super vector bundles is canonically isomorphic to the category $\Homcat_{\stwoLineBdl_{k}(X)}(\mathscr{G},\mathcal{I})$ of super 2-line bundle morphisms from $\mathscr{G}$ to $\mathscr{I}$.
\end{corollary}

For later use, let us spell out explicitly what a $\mathscr{G}$-twisted super vector bundle is, consulting above definition of a 1-morphism (and using \cref{lem:canonicalrefinements:a}). Suppose $\mathscr{G}=(\pi,\mathcal{M},\mu)$. Then, a $\mathscr{G}$-twisted super vector bundle $\mathscr{E}$ is a pair $(\mathcal{E},\varepsilon)$ consisting of a super vector bundle $\mathcal{E}$ over $Y$ and a super vector bundle isomorphism $\varepsilon:  \pr_2^{*}\mathcal{E} \otimes \mathcal{M} \to \pr_1^{*}\mathcal{E}$ over $Y^{[2]}$ such that
\begin{equation}
\label{eq:gerbemodulecond}
\begin{aligned}
%\pr_{13}^{*}\varepsilon = (\id \otimes \mu)\circ  (\pr_{23}^{*}\varepsilon \otimes \id) \circ \pr_{12}^{*}\varepsilon\text{.}
\xymatrix@C=6em{
\mathcal{E}_{y_3} \otimes \mathcal{M}_{y_2, y_3} \otimes \mathcal{M}_{y_1, y_2} \ar[d]_{\id \otimes \mu_{y_1, y_2, y_3}}  \ar[r]^-{\epsilon_{y_2, y_3} \otimes \id} & \mathcal{E}_{y_2} \otimes \mathcal{M}_{y_1, y_2}\ar[d]^{\epsilon_{y_1, y_2}}  
 \\
\mathcal{E}_{y_3} \otimes \mathcal{M}_{y_1, y_3}   \ar[r]_-{\epsilon_{y_1, y_3}} &  \mathcal{E}_{y_1}
 }
\end{aligned}
\end{equation}
commutes.
Likewise, a morphism of $\mathscr{G}$-twisted super vector bundles $(\mathcal{E}_1,\varepsilon_1)$ and $(\mathcal{E}_2,\varepsilon_2)$ is a super vector bundle morphism $\varphi: \mathcal{E} \to \mathcal{E}^\prime$ over $Y$ such that 
\begin{equation*}
\xymatrix@C=6em{
 \mathcal{E}_{y_2} \otimes \mathcal{M}_{y_1, y_2}  \ar[r]^-{\epsilon_{y_1, y_2}} \ar[d]_{\varphi_{y_1} \otimes \id} &  \mathcal{E}_{y_1} \ar[d]^{\varphi_{y_1}} \\
 \mathcal{E}^\prime_{y_2} \otimes \mathcal{M}_{y_1, y_2} \ar[r]_-{\epsilon^\prime_{y_1, y_2}} &   \mathcal{E}^\prime_{y_1} 
  }
\end{equation*}
commutes.
%$(\pr_2^{*}\varphi \otimes \id_{\mathcal{M}}) \circ \varepsilon_1 = \varepsilon_2 \circ \pr_1^{*}\varphi$.  

\begin{remark}
If $\mathscr{L}$ is a general super 2-line bundle, then a 1-morphism $\mathscr{E}:\mathscr{L} \to \mathscr{I}$ generalizes in a natural way the notion of a twisted vector bundle, now admitting more general twistings. In twisted K-theory, these more general twistings add to the ordinary twistings considered by Freed-Hopkins-Teleman the \emph{grading twist}. This is also explained in the lecture notes \cite{Freed2012} and  in \cite{Mertsch2020}. 
\end{remark}

\begin{remark}
In the ungraded case, the analogous definition
\begin{equation*}
\Grb_k := (\mathscr{B}\VectBdl_k)^{+}
\end{equation*}
results in the usual definition of line bundle gerbes precisely as originally defined by Murray (for $k=\C$) \cite{Murray1996}. It induces a fully faithful morphism
\begin{equation}
\label{eq:ungbgii2vb}
\Grb_k \to \twoVectBdl_k
\end{equation}
of 2-stacks, 
so that every line bundle gerbe is an example of a 2-vector bundle.
The analog of \cref{lem:bundlegerbesand2vectorbundles} holds: an ungraded 2-vector bundle $\mathscr{V}$ is a bundle gerbe if and only if it is of Morita class $k$.
\end{remark}

\subsection{Inclusion of algebra bundles}

\label{sec:inclusionofalgebrabundles}

Since the plus construction is 2-stackification, it comes equipped with a fully faithful functor \cite[Thm.~3.3]{nikolaus2} from the pre-2-stack to the 2-stack. In our case, we obtain a fully faithful functor
\begin{equation} \label{InclusionFunctorAlgebraBundles}
\sAlgBdlbi_k(X) \to \stwoVectBdl_k(X)
\end{equation}
including our \quot{preliminary} super 2-vector bundles into the \emph{true} super 2-vector bundles. We recall that the objects of $\sAlgBdlbi_k(X)$ are super algebra bundles over $X$. Hence, super algebra bundles are examples of 2-vector bundles.  

In detail, a super algebra bundle $\mathcal{A}$ over $X$ is sent to the super 2-vector bundle $(\id_X, \mathcal{A},\mathcal{A},\mu)$, where we identify the fibre products $X^{[k]}$ over $X$ with $X$, and $\mu$ is the canonical invertible intertwiner $\mathcal{A} \times_{\mathcal{A}} \mathcal{A} \cong \mathcal{A}$ induced by the multiplication in $\mathcal{A}$. 
We denote this 2-vector bundle again by $\mathcal{A}$.
We also recall that the 1-morphisms $\mathcal{A}\to \mathcal{B}$ in $\sAlgBdlbi_k(X)$ are implementing $\mathcal{B}$-$\mathcal{A}$-bimodule bundles $\mathcal{M}$ over $X$. Such a bimodule bundle is sent to the 1-morphism of 2-vector bundles given by $(\id_X,\mathcal{M},\phi)$, where $Z := X \times_X X$ and $Z^{[k]}$ are again identified with $X$, and $\phi$ is the canonical invertible intertwiner
$\mathcal{B} \otimes_{\mathcal{B}} \mathcal{M} \cong \mathcal{M} \otimes_{\mathcal{A}} \mathcal{A}$. Finally, the 2-morphisms in $\sAlgBdlbi_k(X)$ directly yield 2-morphisms in $\stwoVectBdl_k(X)$.

The statement that the functor \eqref{InclusionFunctorAlgebraBundles} is fully faithful means that it induces an equivalence of categories
\begin{equation*}
\sBimodBdl_{\mathcal{B},\mathcal{A}}^{\mathrm{imp}}(X)  \cong \Homcat_{\stwoVectBdl_k(X)}({\mathcal{A}},{\mathcal{B}})\text{.}
\end{equation*}
In particular, two super algebra bundles are Morita equivalent if and only if they are isomorphic as 2-vector bundles.

\begin{example}
The trivial super algebra bundle $\underline{k}$ over $X$ coincides under the inclusion of \cref{InclusionFunctorAlgebraBundles} with the trivial 2-vector bundle $\mathscr{I}$ of \cref{ex:trivbundlegerbe1}.
\end{example}

\begin{remark}
If a super algebra bundle $\mathcal{A}$ has a typical fibre $A$, then the corresponding super 2-vector bundle is of Morita class $A$. This shows that the Morita class cannot distinguish between general  super 2-vector bundles and those coming from  super algebra bundles. Using the classification we develop in \cref{sec:class2vect} we will obtain a result that allows one to determine for a super 2-vector bundle whose Morita class is a central simple super algebra, whether or not it comes from a super algebra bundle, see \cref{co:line2bundlesandalgebrabundles}. 
\end{remark}

Now that we are able to consider bundle gerbes \emph{and} super algebra bundles as 2-vector bundles, we may discuss their relation. A nice structure that relates bundle gerbes and algebra bundles is the following extension of the notion of a twisted super vector bundle (see \cref{sec:bundlegerbes}).  
\begin{definition}[Twisted module bundle]
\label{def:twistedmodulebundle}
Let $\mathscr{G}$ be a super line bundle gerbe and let $\mathcal{A}$ be a super algebra bundle over $X$. 
A \emph{$\mathscr{G}$-twisted $\mathcal{A}$-module bundle} is 
a $\mathscr{G}$-twisted super vector bundle $\mathscr{E}=(\mathcal{E},\varepsilon)$  together with a left $\pi^{*}\mathcal{A}$-module bundle structure on $\mathcal{E}$, such that $\varepsilon$ is $\mathcal{A}$-linear. A morphism between $\mathscr{G}$-twisted $\mathcal{A}$-module bundles is an $\mathcal{A}$-linear morphism of $\mathscr{G}$-twisted super vector bundles.
\end{definition}

This structure appears in an infinite-dimensional setting in \cite[Def.~2.3.9]{kristel2020smooth}.  The following result generalizes \cref{co:twistedsupervectorbundles} and shows that the notion of a twisted module bundle can now be absorbed in the bicategory of super 2-vector bundles. Let $\mathscr{E}=(\mathcal{E},\varepsilon)$ be a $\mathscr{G}$-twisted $\mathcal{A}$-module bundle.  In order to obtain from $\mathscr{E}$ a 1-morphism $\mathscr{G} \to {\mathcal{A}}$ we identify $Y$ with the common refinement $Y \times_X X$ of the coverings of $\mathscr{G}$ and ${\mathcal{A}}$, set $Z:= Y$ and $\zeta=\id_Z$. Since $\mathcal{E}_y$ is a left $\mathcal{A}_x$-module, where $x=\pi(y)$, it is a $\mathcal{A}_x$-$k$-bimodule. Since $\varepsilon$ is linear and $\pi^{*}\mathcal{A}$-linear, we may consider it as an $\mathcal{A}_x$-$k$-intertwiner
\begin{equation*}
\epsilon_{y_1,y_2}: \mathcal{E}_{y_2} \otimes_{k} \mathcal{L}_{y_1,y_2} \to \mathcal{A}_{x} \otimes_{\mathcal{A}_{x}} \mathcal{E}_{y_1} \text{.}
\end{equation*}
Thus, $\mathscr{P}_{\mathscr{E}}:=(\zeta,\mathcal{E},\varepsilon)$ is a 1-morphism $\mathscr{G} \to {\mathcal{A}}$. 
 
\begin{lemma}
\label{lem:twistedmodulebundles}
The assignment $\mathscr{E}\mapsto \mathscr{P}_{\mathscr{E}}$ establishes an isomorphism between the category of $\mathscr{G}$-twisted $\mathcal{A}$-module bundles and the category $\Homcat_{\stwoVectBdl_k(X)}(\mathscr{G},{\mathcal{A}})$ of super 2-vector bundle morphisms from $\mathscr{G}$ to ${\mathcal{A}}$. Moreover, $\mathscr{P}_{\mathscr{E}}: \mathscr{G} \to {\mathcal{A}}$ is invertible  if and only if the fibres $\mathcal{E}_y$ are Morita equivalences between $k$ and $\mathcal{A}_{\pi(y)}$, for all $y\in Y$.  
\end{lemma}

\begin{proof}
It is straightforward to extend above construction to a functor, and to show that it is an equivalence. The invertibility statement follows from \cref{LemmaInvertibility} and \cite[Lem.~4.2.8 (c)]{Kristel2022}.   
\end{proof}

\begin{remark}
\label{re:untwisting}
$\mathscr{G}$-twisted $\mathcal{A}$-module bundles can be {untwisted} by a trivialization of $\mathscr{G}$. 
Indeed, if $\mathscr{T}$ is such trivialization, i.e., a 1-isomorphism $\mathscr{T}:\mathscr{G} \to \mathscr{I}$, and $\mathscr{P}_{\mathscr{E}}: \mathscr{G} \to {\mathcal{A}}$ is the 1-morphism that corresponds to a $\mathscr{G}$-twisted $\mathcal{A}$-module bundle $\mathscr{E}$ under the isomorphism of \cref{lem:twistedmodulebundles}, then $\mathscr{E}^{\mathscr{T}}:=\mathscr{P}_{\mathscr{E}} \circ \mathscr{T}^{-1}: \mathscr{I} \to {\mathcal{A}}$ is a 1-morphism between super 2-vector bundles in the image of the inclusion of algebra bundles. 
Since this inclusion functor is fully faithful, $\mathscr{E}^{\mathscr{T}}$ corresponds canonically to a 1-morphism $\underline{k} \to \mathcal{A}$ in $\sAlgBdlbi_k(X)$, i.e., to a right $\mathcal{A}$-module bundle.
Summarizing, this procedure turns a $\mathscr{G}$-twisted $\mathcal{A}$-module bundle $\mathscr{E}$ into an $\mathcal{A}$-module bundle $\mathscr{E}^{\mathscr{T}}$ over $X$, using the trivialization $\mathscr{T}$.      
\end{remark}

\begin{remark}
Given a $\mathscr{G}$-twisted $\mathcal{A}$-module bundle $\mathscr{E}$, it is possible to forget the $\mathcal{A}$-module structure and just keep a $\mathscr{G}$-twisted vector bundle.
Under the identification of $\mathscr{E}$ with a 1-morphism $\mathscr{E}:\mathscr{G} \to {\mathcal{A}}$, this corresponds to the composition with the canonical, but non-invertible 1-morphism ${\mathcal{A}} \to \mathscr{I}$ obtained as the image of the  $\underline{k}$-$\mathcal{A}$-bimodule bundle $\mathcal{A}$ under the inclusion of \cref{InclusionFunctorAlgebraBundles}. 
\end{remark}

Given a general super line bundle gerbe $\mathscr{G}$, one may try to construct an algebra bundle $\mathcal{A}$ with a 1-morphism $\mathscr{G} \to {\mathcal{A}}$. 
One method, which applies to lifting gerbes and uses representation theory, is described in \cref{sec:liftinggerbes}. 
Another method is the following. 
Suppose a non-zero $\mathscr{G}$-twisted super vector bundle $\mathscr{E}=(\mathcal{E},\varepsilon)$ is given. 
Then, the endomorphism bundle $\underline{\mathrm{End}}(\mathcal{E})=\mathcal{E} \otimes \mathcal{E}^{*}$ descends using $\varepsilon$ to a super algebra bundle over $X$, which we denote by $\End(\mathscr{E})$. 
 By construction, $\mathscr{E}$ becomes then a $\mathscr{G}$-twisted $\End(\mathscr{E})$-module bundle, and certainly, $\mathcal{E}$ is fibrewise a Morita equivalence.  Hence, we have the following consequence of \cref{lem:twistedmodulebundles}.

\begin{corollary}
Let $\mathscr{G}$ be a super line bundle gerbe over $X$, and suppose $\mathscr{E}$ is a non-zero $\mathscr{G}$-twisted super vector bundle. Then, $\mathscr{E}$ induces a 1-isomorphism $\mathscr{G} \cong \End(\mathscr{E})$ in $\stwoVectBdl_k(X)$.
\end{corollary}

A nice argument \cite[Prop.~4.1]{Bouwknegt2002} shows that a (super) bundle gerbe $\mathscr{G}$ admits a non-zero $\mathscr{G}$-twisted vector bundle if and only if its Dixmier-Douady class in $\mathrm{H}^3(X,\Z)$ is torsion. 
\begin{comment}
That argument is nice indeed: If $\mathscr{G}$-twisted vector bundle $\mathcal{E}$ has rank $r$, then $\Lambda^r \mathcal{E}$ is a $\mathscr{G}^{\otimes r}$-twisted vector bundle. 
However, $\Lambda^r \mathcal{E}$ has rank one, hence it is a trivialization. 
This implies that the Dixmier-Douadi class of $\mathscr{G}^{\otimes r}$ is zero, which is $r$ times the Dixmier-Douadi class of $\mathscr{G}$.
\end{comment}
Thus, any torsion bundle gerbe is isomorphic to an algebra bundle.  
The converse is also true. Both statements will be proved in a different way later, see \cref{co:bundlegerbesandalgebrabundles}.

\subsection{The fibres of a 2-vector bundle}

In the introduction we claimed that a (super) 2-vector bundle is a structure whose fibres are (super) 2-vector spaces.
Suppose $\mathscr{V}$ is a super 2-vector bundle over $X$, and $x\in X$. Then, the fibre of $\mathscr{V}$ at $x$ is defined to be the pullback of $\mathscr{V}$ along the map $x:\ast \to X$. We shall thus analyze what a super 2-vector bundle over the point is. 

A super 2-vector bundle over a point $X=\{\ast\}$ is a tuple $(Y,\mathcal{A},\mathcal{M},\mu)$ consisting of a smooth manifold $Y$, a super algebra bundle $\mathcal{A}$ over $Y$, an invertible $\pr_1^{*}\!\mathcal{A}$-$\pr_2^{*}\!\mathcal{A}$-bimodule bundle $\mathcal{M}$ over $Y^{2}=Y\times Y$, and an invertible intertwiner $\mu: \pr_{23}^{*}\mathcal{M} \otimes_{\pr_2^{*}\mathcal{A}} \pr_{12}^{*}\mathcal{M} \to \pr_{13}^{*}\mathcal{M}$ over $Y^{3}$ that is associative over $Y^{4}$. Similarly, we obtain notions of 1-morphisms and 2-morphisms of super 2-vector bundles over a point, and we may consider the bicategory $\stwoVectBdl_k(\ast)$. 

We consider the functor \cref{InclusionFunctorAlgebraBundles} that includes super algebra bundles into super 2-vector bundles over a point $X=\{\ast\}$, obtaining a functor
\begin{equation}
\label{eq:fibres}
\sAlgBdlbi_k(\ast) \to \stwoVectBdl_k(\ast)\text{.}
\end{equation}
We note that the bicategory of super algebra bundles over a point coincides on the nose with the bicategory $\stwoVect_k$ of super 2-vector spaces. 
Moreover, we have the following result.

\begin{lemma}
\label{lem:fibres}
The functor of \cref{eq:fibres} establishes an equivalence of categories,
\begin{equation*}
\stwoVect_k \cong \stwoVectBdl_k(\ast)\text{.}
\end{equation*}
\end{lemma}

\begin{proof}
We know already from \cref{sec:inclusionofalgebrabundles} that the functor is fully faithful; hence it remains to show that it is essentially surjective. Indeed, suppose $\mathscr{V}=(Y,\mathcal{A},\mathcal{M},\mu)$ is an object in $\stwoVectBdl_k(\ast)$. Choose a point $y_0\in Y$, and let $A := \mathcal{A}_{y_0}$. We show that $A$ is an essential preimage for $\mathscr{V}$. We consider $Z:= Y \cong \ast \times_{\ast} Y$, equipped with the surjective submersion $\zeta:=\id_Y$. Over $Z$ we consider the invertible bimodule bundle $\mathcal{P} := \Delta_1^{*}\mathcal{M}$, where $\Delta_l: Y \to Y^l$ is defined by $\Delta_l(y_1,...,y_l)=(y_0,y_1,...,y_l)$.
\begin{comment}
Thus, $\mathcal{P}_y = \mathcal{M}_{y_0,y}$ is an $\mathcal{A}_{y}$-$A$-bimodule.
\end{comment}
Finally, we consider over $Z^{[2]}=Y^2$ the intertwiner $\phi:=\Delta_2^{*}\mu$, which we may view fibrewise as an intertwiner 
\begin{equation*}
\phi_{y,y'}: \mathcal{P}_{y'} \otimes_{A} A \to \mathcal{M}_{y,y'} \otimes_{\mathcal{A}_{y}} \mathcal{P}_{y} \text{.}
\end{equation*}
It is straightforward to see that $(\zeta,\mathcal{P},\phi)$ is a 1-morphism $A \to \mathscr{V}$,  and it follows from \cref{LemmaInvertibilityA} that it is a 1-\emph{iso}morphism.      
\end{proof}

In view of the equivalence of \cref{lem:fibres}, the fibres of a super 2-vector bundle are super 2-vector spaces, as we claimed in the introduction. We remark that, by \cref{lem:Moritaclass2vect:a}, all fibres of a super 2-vector bundle $\mathscr{V}$ (over a connected base manifold) are isomorphic as super 2-vector spaces, and that they are all isomorphic to the Morita class of $\mathscr{V}$. In this sense, the Morita class may also be viewed as the typical fibre of $\mathscr{V}$.

\subsection{Framing by refinements}

\label{sec:framing}

As discussed in  \cref{sec:algebrabundles} (see \cref{eq:framing1,eq:framing2}), the bicategory $\sAlgBdlbi_k(X)$ of super algebra bundles is framed under the groupoid $\sAlgBdlgrpd kX$ of super algebra bundles over $X$ and bundle \emph{iso}morphisms,
while the bicategory $\sssAlgBdlbi_k(X)$ of \emph{semisimple} super algebra bundles is framed under the category $\sssAlgBdl_k(X)$ of semisimple super  algebra bundles and \textit{all} bundle homomorphisms.
These framings are obviously compatible with pullbacks, and hence morphisms of pre-2-stacks 
\begin{equation}
\label{eq:framingpre2vect}
\sAlgBdlgrpd k-\to \sAlgBdlbi_k
\quand
 \sssAlgBdl_k \to \sssAlgBdlbi_k\text{.}
\end{equation}
We will use the obvious terminology to  say that a \emph{framing for a presheaf of bicategories} $\mathscr{F}$ is a presheaf of categories $\mathscr{E}$ together with a morphism $\mathscr{E} \to \mathscr{F}$ of presheaves of bicategories such that over every smooth manifold $X$, the functor $\mathscr{E}(X) \to \mathscr{F}(X)$ is  framing. In this situation, we will also call $\mathscr{F}$ a \emph{framed presheaf of bicategories}. In this sense, we see that $\sAlgBdlbi_k$ and $\sssAlgBdlbi_k$ are \emph{framed pre-2-stacks}.
Observe that the second pre-2-stack is smaller but has a larger framing on the level of morphisms.

Framings are important because they provide a convenient and simple way to construct 1-morphisms, in situations where the structure of a general 1-morphism is fairly complex, as in the case of super 2-vector bundles. 
The plus construction automatically sends framed pre-2-stacks to framed 2-stacks. More explicitly, if $\mathscr{E} \to \mathscr{F}$ is a framed pre-2-stack, then there exists a general procedure to associate to $\mathscr{E}$ another presheaf of categories $\mathscr{E}^{\mathscr{F}}$ together with a morphism $\mathscr{E}^{\mathscr{F}} \to \mathscr{F}^{+}$ turning $\mathscr{F}^{+}$ into a framed 2-stack. 
To avoid confusion, we remark that this procedure, $\mathscr{E} \mapsto \mathscr{E}^{\mathscr{F}}$, is \emph{not} the plus construction or another method of stackification. In fact, 
in many cases,  $\mathscr{E}$ is already a stack -- as in \cref{eq:framingpre2vect}.  Instead, the construction of $\mathscr{E}^{\mathscr{F}}$ depends on the framing $\mathscr{E} \to \mathscr{F}$; after all, $\mathscr{E}^{\mathscr{F}}$ needs to have the same objects as $\mathscr{F}^{+}$ in order to be eligible for a framing. We will describe the general definition of $\mathscr{E}^{\mathscr{F}}$ elsewhere; below we will only spell it out in the present cases of the framed pre-2-stacks of \cref{eq:framingpre2vect}.

We start with the following basic construction. Given a smooth manifold $X$, we define a category $\stwoVectBdlref_k(X)$ as follows.
The objects are  all super 2-vector bundles $\mathscr{V}$ over $X$. 
The morphisms will be called \emph{refinements}, defined as follows.

\begin{definition}[Refinement]
Let $\mathscr{V}_1=(\pi_1,\mathcal{A}_1,\mathcal{M}_1,\mu_1)$ and $\mathscr{V}_2=(\pi_2,\mathcal{A}_2,\mathcal{M}_2,\mu_2)$ be super 2-vector bundles over $X$. 
A \emph{refinement} $\mathscr{R}:\mathscr{V}_1\to \mathscr{V}_2$ is a
 triple $\mathscr{R}=(\rho,\phi,u)$ consisting of a smooth map $\rho:Y_1 \to Y_2$ such that $\pi_2\circ \rho=\pi_1$, of a homomorphism $\phi:\mathcal{A}_1\to \rho^{*}\mathcal{A}_2$ of super algebra bundles over $Y_1$ and 
of an invertible bundle morphism $u:\mathcal{M}_1 \to \rho^* \mathcal{M}_2$  over $Y_1^{[2]}$
%  of an invertible intertwiner $u$ of bimodule bundles over $Y_1^{[2]}$,
that over a point $(y,y')\in Y_1^{[2]}$ restricts to an intertwiner 
\begin{equation*}
u_{y,y'}: (\mathcal{M}_1)_{y,y'} \to (\mathcal{M}_2)_{\rho(y),\rho(y')}
\end{equation*}
along the algebra homomorphisms $\phi_{y'}:(\mathcal{A}_1)_{y'} \to (\mathcal{A}_2)_{\rho(y')}$ and $\phi_{y}:(\mathcal{A}_1)_y \to (\mathcal{A}_2)_{\rho(y)}$, and renders the diagram
\begin{equation}
\label{eq:refinementdiag}
\begin{aligned}
\xymatrix@C=6em{(\mathcal{M}_1)_{y',y''} \otimes_{(\mathcal{A}_1)_{y'}} (\mathcal{M}_1)_{y,y'}\ar[d]_{u_{y',y''} \otimes u_{y,y'}} \ar[r]^-{\id \circ \mu_1} & (\mathcal{M}_1)_{y,y''} \ar[d]^{u_{y,y''}}\\   (\mathcal{M}_2)_{\rho(y'),\rho(y'')} \otimes_{(\mathcal{A}_2)_{\rho(y')}} (\mathcal{M}_2)_{\rho(y),\rho(y')}  \ar[r]_-{\mu_2\circ \id} & (\mathcal{M}_2)_{\rho(y),\rho(y'')}}
\end{aligned}
\end{equation}
commutative for all $(y,y',y'')\in Y_1^{[3]}$.
\end{definition}

Given two refinements $\mathscr{R}_{12}=(\rho_{12},\phi_{12},u_{12}):\mathscr{V}_1 \to \mathscr{V}_2$ and $\mathscr{R}_{23}=(\rho_{23},\phi_{23},u_{23}):\mathscr{V}_2 \to \mathscr{V}_3$, their composition is given by 
\begin{equation*}
\mathscr{R}_{23} \circ \mathscr{R}_{12} := (\rho_{23} \circ \rho_{12},\rho_{12}^{*}\phi_{23} \circ \phi_{12},(\rho_{12}^{[2]})^{*}u_{23} \circ u_{12})\text{,}
\end{equation*}
and the identity morphism of $\mathscr{V}$ is $(\id_Y,\id_{\mathcal{A}},\id_{\mathcal{M}})$. 
This defines the category $\stwoVectBdlref_k(X)$. It is clear that everything is compatible with pullbacks, and so $\stwoVectBdlref_k$ is a presheaf of categories. 

For each manifold $X$, let $\stwoVectBdlrefinv_k(X)$ be the subcategory of $\stwoVectBdlref_k(X)$ with all super 2-vector bundles and only those refinements whose algebra bundle homomorphism $\phi$ is invertible. Further, we let $\ssstwoVectBdlref_k(X)$ be the full subcategory of $\stwoVectBdlref_k(X)$ over all semisimple super 2-vector bundles.  They assemble to sub-presheaves $\stwoVectBdlrefinv_k$ and $\ssstwoVectBdlref_k$ of $\stwoVectBdlref_k$, and these  are the presheaves  $\mathscr{E}^{\mathscr{F}}$ in the above general notation, explicitly, we have
\begin{align*}
\stwoVectBdlrefinv_k &= \sAlgBdlgrpd k-^{\sAlgBdlbi_k}
\\
 \ssstwoVectBdlref_k &= \sssAlgBdl_k^{\sssAlgBdlbi_k}
\end{align*}
The functors $\mathscr{E}^{\mathscr{F}} \to \mathscr{F}^{+}$, explicitly,
\begin{equation} \label{Framing2Vect}
  \stwoVectBdlrefinv_k \to \stwoVectBdl_k
  \quand
  \ssstwoVectBdlref_k \to \ssstwoVectBdl_k,
\end{equation}
are defined as follows.
Working over a manifold $X$, they are, of course, the identity on the level of objects.

%The framing
%\begin{equation*}
%\stwoVectBdlref_k=\sAlgBdl_k^{\sAlgBdlbi_k} \to (\sAlgBdlbi_k)^{+} = \stwoVectBdl_k\text{,}
%\end{equation*} 
%of which we claim that it is induced from the original framing \cref{eq:framingpre2vect} under the plus construction, is the following. 
%Again working over $X$, it is, of course, the identity on the level of objects. 
On the level of morphisms, they associate to a refinement $\mathscr{R}=(\rho,\phi,u):\mathscr{V}_1 \to \mathscr{V}_2$ the following 1-morphism $(\zeta,\mathcal{P},\phi')$. We define $Z:= Y_1 \times_X Y_2$ and $\zeta=\id_Z$. Consider the smooth map $\tilde \rho: Z \to Y_2^{[2]}$ with $\tilde\rho(y_1,y_2):=(\rho(y_1),y_2)$. 
We define $\mathcal{P} := (\tilde\rho^{*}\mathcal{M}_2)_{\pr_1^{*}\phi}$ over $Z$.
If $\phi$ is invertible, this is an implementing bimodule bundle by \cref{ExampleTwistedModuleIso}, while if $\mathcal{A}_1$ and $\mathcal{A}_2$ have semisimple fibres, it is an implementing bimodule bundle by \cref{RemarkHHSemisimple}.
\begin{comment}
Recall that $\mathcal{M}_2$ is an $\pr_2^{*}\mathcal{A}_2$-$\pr_1^{*}\mathcal{A}_2$-bimodule bundle over $Y^{[2]}_2$. 
Note that $\tilde\rho^{*}\mathcal{M}_2$ is then a $\pr_2^{*}\mathcal{A}_2$-$(\rho \circ \pr_1)^{*}\mathcal{A}_2$ bimodule bundle. We have the algebra bundle homomorphism $\phi: \mathcal{A}_1 \to \rho^{*}\mathcal{A}_2$ over $Y_1$. Its pullback along $\pr_1: Z \to Y_1$ is $\pr_1^{*}\phi: \pr_1^{*}\mathcal{A}_1 \to (\rho\circ \pr_1)^{*}\mathcal{A}_2$.  
\end{comment}
Over a point $(y_1,y_2)\in Z$, its fibre is $\mathcal{P}_{y_1,y_2}=((\mathcal{M}_2)_{\rho(y_1),y_2})_{\phi_{y_1}}$, which is indeed an $(\mathcal{A}_2)_{y_2}$-$(\mathcal{A}_1)_{y_1}$-bimodule. 
Finally, we define the intertwiner $\phi'$ fibrewise over a point $((y_1,y_2),(y_1',y_2'))\in Z^{[2]}$ by
\begin{equation*}
\xymatrix{
\hspace{-13em}  \mathcal{P}_{y_1',y_2'} \otimes_{(\mathcal{A}_1)_{y_1'}} (\mathcal{M}_1)_{y_1,y_1'}=((\mathcal{M}_2)_{\rho(y_1'),y_2'})_{\phi_{y_1'}} \otimes_{(\mathcal{A}_1)_{y_1'}} (\mathcal{M}_1)_{y_1,y_1'} 
 \ar[d]^{\id \otimes u_{y_1,y_1'}} 
\\ 
(\mathcal{M}_2)_{\rho(y_1^\prime), y_2^\prime} \otimes_{(\mathcal{A}_2)_{\rho(y_1^\prime)}} ((\mathcal{M}_2)_{\rho(y_1),\rho(y_1^\prime)})_{\phi_{y_1}} 
 \ar[d]^{(\mu_2)_{\rho(y_1),\rho(y_1'),y_2'}}
\\ 
((\mathcal{M}_2)_{\rho(y_1),y_2'})_{\phi_{y_1}} \ar[d]^{(\mu_2)_{\rho(y_1),y_2,y_2'}^{-1}} 
\\
(\mathcal{M}_2)_{y_2,y_2'} \otimes_{(\mathcal{A}_2)_{y_2}} ((\mathcal{M}_2)_{\rho(y_1,)y_2})_{\phi_{y_1}} 
= (\mathcal{M}_2)_{y_2,y_2'} \otimes_{(\mathcal{A}_2)_{y_2}} \mathcal{P}_{y_1,y_2}\text{.}  \hspace{-13em} }
\end{equation*}
%\begin{equation*}
%\xymatrix{\hspace{-10em}(\mathcal{M}_2)_{y_2,y_2'} \otimes_{(\mathcal{A}_2)_{y_2}} \mathcal{P}_{y_1,y_2} = (\mathcal{M}_2)_{y_2,y_2'} \otimes_{(\mathcal{A}_2)_{y_2}} ((\mathcal{M}_2)_{\rho(y_1,)y_2})_{\phi_{y_1}} \ar[d]^{(\mu_2)_{\rho(y_1),y_2,y_2'}} \\ ((\mathcal{M}_2)_{\rho(y_1),y_2'})_{\phi_{y_1}} \ar[d]^{(\mu_2^{-1})_{\rho(y_1),\rho(y_1'),y_2'}}  \\ 
%(\mathcal{M}_2)_{\rho(y_1^\prime), y_2^\prime} \otimes_{(\mathcal{A}_2)_{\rho(y_1^\prime)}} ((\mathcal{M}_2)_{\rho(y_1),\rho(y_1^\prime)})_{\phi_{y_1}} \ar[d]^{\id \otimes u^{-1}_{y_1,y_1'}}
%\\
% ((\mathcal{M}_2)_{\rho(y_1'),y_2'})_{\phi_{y_1'}} \otimes_{(\mathcal{A}_1)_{y_1'}} (\mathcal{M}_1)_{y_1,y_1'} = \mathcal{P}_{y_1',y_2'} \otimes_{(\mathcal{A}_1)_{y_1'}} (\mathcal{M}_1)_{y_1,y_1'}\text{.}\hspace{-10em}}
%\end{equation*}
It is tedious though absolutely straightforward to check that $(\zeta,\mathcal{P},\phi')$ defined like this is a 1-morphism $\mathscr{V}_1 \to \mathscr{V}_2$.
Even more, it turns out that this assignment indeed defines a functor, and moreover, using \cref{LemmaInvertibilityA2}, that every 1-morphism obtained from a refinement has a right adjoint. 
%Thus, for each smooth manifold $X$, we have a framed bicategory $\stwoVectBdlref_k(X) \to \stwoVectBdl_k(X)$, and this shows that we have constructed the claimed framed 2-stack
%\begin{equation*}
%\stwoVectBdlref_k \to \stwoVectBdl_k\text{.}
%\end{equation*}   
In total, we obtain the following statement.

\begin{proposition}
The functors of \cref{Framing2Vect} are  framings.
\end{proposition}

\begin{remark}
We recall from \cref{sec:bundlegerbes} that super bundle gerbes are obtained by applying the plus construction to the pre-2-stack $\mathscr{B}\sVectBdl_k$. That pre-2-stack  is trivially framed, i.e., the trivial functor
\begin{equation*}
\ast \to \mathscr{B}\sVectBdl_k(X)
\end{equation*}
from the category $\ast$ with a single object and a single morphism is a framing. Nonetheless, the construction $\mathscr{E}\mapsto \mathscr{E}^{\mathscr{F}}$ generates from it a non-trivial presheaf of categories 
\begin{equation*}
\sGrbref_k := \ast^{\mathscr{B}\sVectBdl_k} \text{,}
\end{equation*}  
together with a functor
\begin{equation*}
\sGrbref_k \to \sGrb_k\text{.}
\end{equation*}
This is the \quot{usual} framing of the 2-stack of bundle gerbes, i.e., the one whose morphisms are refinements of bundle gerbes, or \quot{non-stable morphisms}.
The (trivially) commutative diagram
\begin{equation*}
\xymatrix{\ast \ar[r] \ar[d] & \mathscr{B}\sVectBdl_k \ar[d] \\ \sssAlgBdl_k\ar[r] & \sssAlgBdlbi_k}
\end{equation*}
expresses that the vertical arrows form a \emph{morphism of framed pre-2-stacks}. Going from $\mathscr{E} \to \mathscr{F}$ to $\mathscr{E}^{\mathscr{F}} \to \mathscr{F}^{+}$ is functorial; hence, the diagram
\begin{equation*}
\xymatrix{\sGrbref_k \ar[r]\ar[d] & \sGrb_k \ar[d] \\ \ssstwoVectBdlref_k \ar[r] & \ssstwoVectBdl_k}
\end{equation*}
is commutative, too. Again, we may say that the vertical arrows, the passage from super line bundle gerbes to super 2-vector bundles, form a \emph{morphism of framed 2-stacks}.
\end{remark}

\begin{remark}
We recall from \cref{sec:inclusionofalgebrabundles} that super algebra bundles (a.k.a.\ preliminary super 2-vector bundles) are examples of super 2-vector bundles, and that we have a pre-2-stack morphism
\begin{equation*}
\sAlgBdlbi_k \to \stwoVectBdl_k\text{.}
\end{equation*} 
The bicategories $\sAlgBdlbi_k$ and $\sssAlgBdlbi_k$ of preliminary super 2-vector bundles are themselves framed under the categories $\sAlgBdlgrpd k-$ and $\sssAlgBdl_k$, respectively.
These framings are compatible with the ones given by refinements, in the sense of functors
\begin{equation*}
\sAlgBdlgrpd k-\to \stwoVectBdlrefinv_k\quand \sssAlgBdl_k \to \ssstwoVectBdlref_k\text{.}
\end{equation*} 
On the level of objects, these functors regard a super algebra bundle $\mathcal{A}$ as a super 2-vector bundle like in \eqref{InclusionFunctorAlgebraBundles}. On the level of morphisms, they send a super algebra bundle homomorphism $\phi: \mathcal{A} \to \mathcal{B}$ to the refinement $\mathscr{R}_{\phi}:=(\id_X,\phi,\phi)$. The diagrams
\begin{equation*}
\xymatrix{ \sAlgBdlgrpd k-\ar[r] \ar[d]_{} & \sAlgBdlbi_k \ar[d] \\ \stwoVectBdlrefinv_k \ar[r] &  \stwoVectBdl_k }
\qquad
\xymatrix{ \sssAlgBdl_k \ar[r] \ar[d]_{} & \sssAlgBdlbi_k \ar[d] \\ \ssstwoVectBdlref_k \ar[r] &  \ssstwoVectBdl_k }
\end{equation*}
obviously commute; here, the vertical arrows go from super algebra bundles to super 2-vector bundles, and the horizontal arrows are the framings. In other words, the passage  from super algebra bundles to super 2-vector bundles is a morphism between framed pre-2-stacks. 
\end{remark}

\subsection{Symmetric monoidal structures}
\label{SectionTensorProduct}

The plus construction automatically extends (symmetric) monoidal structures from pre-2-stacks to 2-stacks. Unfortunately, this has not been discussed in \cite{nikolaus2}; we will describe this in full generality elsewhere. 
The rationale is to go to common refinements of surjective submersions, and then to use the given symmetric monoidal structure of the pre-2-stack.

While the pre-2-stack $\sAlgBdlbi_k$ is \emph{not} symmetric monoidal (recall that the heart of the problem is that the exterior tensor product of implementable bimodules need not be implementable again), we described in \cref{sec:algebrabundles} the two sub-pre-2-stacks $\sAlgBdlbigrpd k-$ and $\sssAlgBdlbi_k$ that \emph{are} symmetric monoidal.
Applying the plus construction equips the 2-stacks
\begin{equation*}
\stwoVectBdlgrpd k- \qquad \text{and} \qquad \ssstwoVectBdl_k
\end{equation*}
with symmetric monoidal structures.
We remark that, if $\mathscr{V}$ and $\mathscr{W}$ are two super 2-vector bundles, their tensor product $\mathscr{V} \otimes \mathscr{W}$ is always  defined (in $\stwoVectBdlgrpd k-$), and it coincides with their tensor product in $\ssstwoVectBdl_k$ whenever $\mathscr{V}$ and $\mathscr{W}$ are semisimple. In other words, the problems only arise from the tensor product of 1-morphisms.

To describe the tensor product of super 2-vector bundles explicitly, let $\mathscr{V}_1=(\pi_1,\mathcal{A}_1,\mathcal{M}_1,\mu_1)$ and $\mathscr{V}_2=(\pi_2,\mathcal{A}_2,\mathcal{M}_2,\mu_2)$ be super 2-vector bundles over $X$. Their tensor product $\mathscr{V}_1 \otimes \mathscr{V}_2$ has the surjective submersion $Y_{12} := Y_1 \times_X Y_2 \to X$, the super algebra bundle $\mathcal{A}_1 \otimes \mathcal{A}_2$ over $Y_{12}$ (here, according to our conventions, the pullbacks are suppressed), the bimodule bundle $\mathcal{M}_1 \otimes \mathcal{M}_2$ over $Y_{12}^{[2]}$, and the intertwiner $\mu_1 \otimes \mu_2$ over $Y_{13}^{[3]}$. The trivial 2-vector bundle $\mathscr{I}$ is the tensor unit.

The following statement is  trivial, but worthwhile to state explicitly.

\begin{proposition}
\label{prop:Moritaclasstensor}
Let $\mathscr{V},\mathscr{W}$ be super 2-vector bundle over $X$, let $\mathscr{V}$ be of Morita class $A$ and $\mathscr{W}$ be of Morita class $B$. Then $\mathscr{V} \otimes \mathscr{W}$ is of Morita class $A \otimes B$.
\end{proposition}

The next statement gives, in particular, another justification for the terminology \quot{2-line bundles}.

\begin{proposition}
\label{prop:dualizability}
Let $X$ be a smooth manifold.
\begin{enumerate}[(a)]

\item 
Every semisimple super 2-vector bundle is fully dualizable in $\ssstwoVectBdl_k(X)$.

\item
A semisimple super 2-vector bundle is invertible in $\ssstwoVectBdl_k(X)$ if and only if it is a super 2-line bundle. 
\item
The following are equivalent for a super 2-vector bundle $\mathscr{V}$:
\begin{enumerate}[1.]

\item 
$\mathscr{V}$ is dualizable in  $\stwoVectBdlgrpd kX$.

\item 
$\mathscr{V}$ is fully dualizable in  $\stwoVectBdlgrpd kX$.
\item 
$\mathscr{V}$ is invertible in  $\stwoVectBdlgrpd kX$.
\item 
$\mathscr{V}$ is a super 2-line bundle.

\end{enumerate}

\end{enumerate}
\end{proposition}

\begin{proof}
One can construct from any super 2-vector bundle $\mathscr{V}$ a dual super 2-vector bundle $\mathscr{V}^{*}$ in a completely natural way: if $\mathcal{A}$ is the super algebra bundle of $\mathscr{V}$, then the opposite algebra bundle $\mathcal{A}^{\opp}$ is the algebra bundle of $\mathscr{V}^{*}$.
For a detailed construction of $\mathscr{V}^{*}$ we refer to \cite[Rem. 2.1.14]{Mertsch2020}.
Similarly, it is unproblematic to construct evaluation and coevaluation 1-morphisms, whose bimodule bundles have the underlying vector bundle $\mathcal{A}$, considered as an $(\mathcal{A} \otimes \mathcal{A}^{\opp})$-$\underline{k}$-bimodule bundle and a $\underline{k}$-$(\mathcal{A}^{\opp}\otimes \mathcal{A})$-bimodule bundle, respectively. Since $\mathcal{A}$ is semisimple, these bimodule bundles have adjoints by \cref{prop:dualizabilityalgebrabundles}, and hence by \cref{LemmaInvertibilityA2}, the corresponding 1-morphisms in $\ssstwoVect_k(X)$ have adjoints, too.  
This shows (a). When $\mathcal{A}$ is central simple, then by \cref{prop:dualizabilityalgebrabundles,LemmaInvertibilityA} we see that evaluation and coevaluation are invertible; this shows \quot{if} of  (b). The \quot{only if} in (b) follows by restricting to connected components and looking at the Morita class: if $\mathscr{V}$ is invertible, then by \cref{prop:Moritaclasstensor,lem:Moritaclass2vect} its Morita class $A$ must be invertible, too; hence $A$ is central simple and $\mathscr{V}$ is a line 2-bundle. 
\\
In (c),  1. -- 3. are equivalent because $\stwoVectBdlgrpd kX$  is a 2-groupoid. The equivalence of 3. and 4. is seen as in (b).
\end{proof}

\begin{remark}
The inclusion $\sGrb_k \to \ssstwoVectBdl_k$ of super line bundle gerbes, a well as the inclusion  $\sssAlgBdlbi_k \to \ssstwoVectBdl_k$ of semisimple super algebra bundles, are symmetric monoidal.  
\end{remark}

 The symmetric monoidal pre-2-stacks $\sAlgBdlbigrpd k-$ and $\sssAlgBdlbi_k$ have the direct sum as a second symmetric monoidal structure (\cref{re:directsum}), which  induces analogously symmetric monoidal structures on $\stwoVectBdlgrpd k-$ and $\ssstwoVectBdl_k$.
 The monoidal unit is the zero super 2-vector bundle $\mathscr{O}$, which is the image of the zero super algebra bundle under the inclusion of algebra bundles. 
 Thus, it has the trivial cover $\id_X$, the zero algebra bundle, the zero bimodule bundle (note that this is invertible as a bimodule bundle between zero algebras), and the zero map. 
Note that, unlike the tensor unit $\mathscr{I}$, the zero 2-vector bundle $\mathscr{O}$ is not a bundle gerbe.

The following statement is again obvious.

\begin{lemma}
Let $\mathscr{V},\mathscr{W}$ be super 2-vector bundle over $X$, let $\mathscr{V}$ be of Morita class $A$ and $\mathscr{W}$ be of Morita class $B$. Then, $\mathscr{V} \oplus \mathscr{W}$ is of Morita class $A \oplus B$.  
\end{lemma}

A perhaps stunning application of the direct sum is that -- within 2-vector bundles -- one may now take the direct sum of two bundle gerbes. Recall that bundle gerbes are 2-line bundles and have Morita class $k$. The direct sum of two bundle gerbes is of Morita class $k \oplus k$ and is hence not a 2-\emph{line} bundle anymore. 

\subsection{Endomorphisms and automorphisms}

If $\mathscr{V}$ is a super 2-vector bundle, then we write $\Endcat(\mathscr{V}) := \Homcat(\mathscr{V},\mathscr{V})$ for the category of endomorphisms, which is monoidal under the composition. We first want to compute the endomorphisms of the trivial 2-vector bundle $\mathscr{I}$.

We recall from \cref{sec:bundlegerbes} that we have  fully faithful morphisms
\begin{equation*}
\mathscr{B}\sVectBdl_k \to \sGrb_k \to \stwoLineBdl_k \to \stwoVectBdl_k
\end{equation*}
of presheaves of bicategories, under which the trivial 2-vector bundle $\mathscr{I}$ is the image of the single object $\ast$ in $\mathscr{B}\sVectBdl_k$.  In particular, we have an equivalence of monoidal categories
\begin{equation*}
\sVectBdl_k(X) = \Endcat_{\mathscr{B}\sVectBdl_k}(\ast) \cong \Endcat_{\stwoVectBdl_k}(\mathscr{I})\text{.}
\end{equation*}
We recall from \cref{sec:inclusionofalgebrabundles} that we also have a fully faithful morphism
\begin{equation*}
\sAlgBdlbi_k \to \stwoVectBdl_k
\end{equation*}
of presheaves of bicategories, under which the trivial 2-vector bundle $\mathscr{I}$ over a manifold $X$ is the image of the trivial super algebra bundle $\underline{k}$ over $X$. 
In particular, we have again an equivalence of monoidal categories
\begin{equation*}
\sVectBdl_k(X) \cong \Endcat_{\sAlgBdlbi_k}(\underline{k}) \cong \Endcat_{\stwoVectBdl_k}(\mathscr{I})\text{.}
\end{equation*}  
It is straightforward to see from the various definitions that both equivalences coincide, and both give the same functor
$H:\sVectBdl_k(X) \to \Endcat_{\stwoVectBdl_k}(\mathscr{I})$: if $\mathcal{V}$ is a super vector bundle over $X$, then the corresponding 1-morphism  $H_{\mathcal{V}}:\mathscr{I} \to \mathscr{I}$ has the identity surjective submersion on $Z:= X \times_X X \cong X$, it has the bimodule bundle $\mathcal{V}$, on which the super algebra bundle $\underline{k}$ of $\mathscr{I}$ acts fibrewise from both sides by scalar multiplication. Moreover, a morphism $\varphi: \mathcal{V} \to \mathcal{W}$ of super vector bundles induces in a straightforward way a 2-morphism $H_{\mathcal{V}} \to H_{\mathcal{W}}$. We summarize above considerations in the following lemma.

\begin{lemma}
\label{lem:HomII}
The functor $H$ establishes an equivalence of monoidal categories
\begin{equation*}
\sVectBdl_k(X) \cong \Endcat(\mathscr{I})\text{.}
\end{equation*}
\end{lemma}

Now let $\mathscr{V}$ be any super 2-vector bundle over $X$, which we take to be semisimple  at first.
Suppose $\mathcal{E}$ is a super vector bundle over $X$. 
Using the functor $H$ and the tensor product in $\ssstwoVectBdl_k$, we define a 1-morphism $H_{\mathscr{V}}(\mathcal{E}):\mathscr{V} \to \mathscr{V}$ by
\begin{equation*}
\xymatrix@C=6em{\mathscr{V} \cong \mathscr{V} \otimes \mathscr{I} \ar[r]^{\id_{\mathscr{V}} \otimes H(\mathcal{E})} & \mathscr{V} \otimes \mathscr{I} \cong \mathscr{V}\text{.}}
\end{equation*}
Likewise, if $\varphi: \mathcal{E}_1 \to \mathcal{E}_2$ is a morphism of super vector bundles, then we define 2-morphism $H_{\mathscr{V}}(\varphi):H_{\mathscr{V}}(\mathcal{E}_1) \Rightarrow H_{\mathscr{V}}(\mathcal{E}_2)$ in the obvious way using $\id_{\id_{\mathscr{V}}} \otimes H(\varphi)$. This defines a monoidal functor
\begin{equation*}
H_{\mathscr{V}}:\sVectBdl_k(X) \to \Endcat(\mathscr{V})\text{,}
\end{equation*}
which coincides in the case of $\mathscr{V}=\mathscr{I}$ with the functor $H$. 
\\
Explicitly, if $\mathscr{V}=(\pi,\mathcal{A},\mathcal{M},\mu)$ is the super 2-vector bundle, then the 1-morphism $H_{\mathscr{V}}(\mathcal{E})$ has the covering space $Z := Y^{[2]}$, the identity surjective submersion $\zeta:=\id_Z$, and the bimodule bundle over $Z$ is $\mathcal{M} \otimes_k \mathcal{E}$, where $\mathcal{E}$ is understood to be pulled back along the projection $Z \to X$.
The 2-morphism $H_{\mathscr{V}}(\varphi): H_{\mathscr{V}}(\mathcal{E}_1) \Rightarrow H_{\mathscr{V}}(\mathcal{E}_2)$ is given by the intertwiner $\id_{\mathcal{M}} \otimes \varphi: \mathcal{M} \otimes_k \mathcal{E}_1 \to \mathcal{M} \otimes_k \mathcal{E}_2$. 
\\
From this explicit description, it is clear that such a functor $H_{\mathscr{V}}$ exists for \emph{any} super vector bundle $\mathscr{V}$ (not just semisimple ones), as the typical fibre $M \otimes_k E$ of $\mathcal{M} \otimes_k \mathcal{E}$ is implementing if $\mathcal{M}$ is; see \cite[Ex. 3.1.5 (2)]{Kristel2022}.

\begin{lemma}
\label{lem:HV}
The functor $H_{\mathscr{V}}$ is  faithful.  
\end{lemma}    

\begin{proof}
Suppose $\mathcal{E}_1$ and $\mathcal{E}_2$ are super vector bundles, and $\varphi,\varphi': \mathcal{E}_1\to \mathcal{E}_2$ are super vector bundle homomorphisms. An equality $H_{\mathscr{V}}(\varphi)=H_{\mathscr{V}}(\varphi')$ between 2-morphisms implies an equality $\id_{\mathcal{M}} \otimes \varphi_1=\id_{\mathcal{M}} \otimes \varphi_2$ between linear maps $\mathcal{M} \otimes_k \mathcal{E}_1 \to \mathcal{M} \otimes_k \mathcal{E}_2$, and since these are finite-dimensional vector spaces, this implies $\varphi_1=\varphi_2$.
\begin{comment}
In bases, the matrix of $\id_{\mathcal{M}} \otimes \varphi$ is the Kronecker product of the unit matrix with the matrix of $\varphi$. Taking Kronecker products with a fixed non-zero matrix is injective. 
\end{comment}  
This shows that $H_{\mathscr{V}}$ is faithful.   
\end{proof}

Let us now specialize  to \emph{automorphism 2-groups}. 
We recall that a (weak) 2-group is a monoidal groupoid in which every object is invertible with respect to the tensor product.
One example of a 2-group is $\sLineBdlgrpd kX$, the groupoidification  of the monoidal category of super line bundles over $X$. The 2-group $\sLineBdlgrpd kX$ is additionally \emph{symmetric} monoidal, i.e., a Picard groupoid.   

 If $\mathscr{C}$ is a bicategory, and $c$ is an object in $\mathscr{C}$, then its automorphism 2-group 
\begin{equation*}
\AUT(c) := \mathrm{Grpd}(\Endcat(c)^{\times})
\end{equation*}
is obtained from the monoidal category $\Endcat(c)$ by discarding all non-invertible objects (i.e., all non-invertible 1-morphisms $c \to c$), and by discarding all non-invertible morphisms (i.e., all 2-morphisms in $\mathscr{C}$ that are not invertible under vertical composition). 
Note that $\AUT(c)$ is in general not \emph{symmetric} monoidal. 
\\
By this construction, every super 2-vector bundle $\mathscr{V}$ over $X$ has an automorphism 2-group $\AUT(\mathscr{V})$. For a general super 2-vector bundle, \cref{lem:HV} implies that the functor $H_{\mathscr{V}}$ induces a  faithful monoidal functor 
\begin{equation*}
\sLineBdlgrpd kX \to \AUT(\mathscr{V})\text{.}
\end{equation*}
By \cref{lem:HomII}, this functor is an equivalence in case of the trivial 2-vector bundle,
\begin{equation*}
\sLineBdlgrpd kX \cong \AUT(\mathscr{I})\text{,}
\end{equation*}
\ie the automorphisms of $\mathscr{I}$ are precisely the super line bundles.

Furthermore, a result from the theory of bundle gerbes shows in fact that the automorphism 2-group of \emph{every}  line bundle gerbe $\mathscr{G}$ is $\LineBdl_k(X)^{\grp}$, see \cite[Thm.~2.5.4]{waldorf4}. The proof given there generalizes in a straightforward way to the super case. Since the functor $\sGrb_k \to \stwoVectBdl_k$ is fully faithful, this shows that we have another equivalence of 2-groups, 
\begin{equation*}
\sLineBdlgrpd kX \cong \AUT(\mathscr{G})\text{,}
\end{equation*}
for every super line bundle gerbe $\mathscr{G}$.

\subsection{Equivariant 2-vector bundles}

As explained in  \cite[Prop.~2.8]{nikolaus2}, presheaves (of bicategories) on the category of smooth manifolds extend canonically to presheaves on the category of Lie groupoids (with smooth functors). This holds, in particular, for super 2-vector bundles, so that we automatically have a notion of super 2-vector bundles \emph{over Lie groupoids}. In particular, this applies to action groupoids and hence leads automatically to the correct notion of an \emph{equivariant} super 2-vector bundle.

Let $G$ be a Lie group and let $\rho:G \times X \to X$ be a smooth action of $G$ on a smooth manifold $X$. The corresponding action  groupoid $X/\!/ G$ is a Lie groupoid with objects $X$, morphisms $G \times X$, source map $\pr_X:G \times X \to X$, target map $\rho$, and composition $(g_2,g_1x)\circ(g_1,x) := (g_2g_1,x)$.

\begin{definition}[Equivariant 2-vector bundle]
A \emph{$G$-equivariant super 2-vector bundle} over $X$ is a super 2-vector bundle over the Lie groupoid $X/\!/G$.
\end{definition}

Spelling out the details on the basis of  \cite[Def.~2.5]{nikolaus2},
 a $G$-equivariant super 2-vector bundle over $X$ is a triple $(\mathscr{V},\mathscr{P},\phi)$ consisting of a super 2-vector bundle $\mathscr{V}$ over $X$, a 1-isomorphism $\mathscr{P}: \pr_X^{*}\mathscr{V} \to \rho^{*}\mathscr{V}$ of super 2-vector bundles over $G \times X$, and a 2-isomorphism 
\begin{equation*}
\phi: (\id \times \rho)^{*}\mathscr{P} \circ \pr_{23}^{*}\rho^{*}\mathscr{P} \Rightarrow (m \times \id)^{*}\mathscr{P}
\end{equation*}
of 2-vector bundles over $G^2  \times X$, where $m:G^2 \to G$ denotes the product of $G$, such that $\phi$ satisfies a coherence condition over $G^3 \times X$. This becomes more instructive when restricted to single group elements: if $g\in G$, then pulling back $\mathscr{P}$ along $X \to G \times X: x \mapsto (g,x)$ yields a 1-isomorphism $\mathscr{P}_g :\mathscr{V} \to g^{*}\mathscr{V}$ over $X$. Similarly, for two group elements $g_1,g_2\in G$ we obtain from $\phi$ a 2-isomorphism $\phi_{g_1,g_2}:  g_1^{*}\mathscr{P}_{g_2} \circ \mathscr{P}_{g_1}  \Rightarrow \mathscr{P}_{g_1g_2}$ over $X$, and the coherence condition becomes 
\begin{equation*}
\phi_{g_1,g_2g_3} \bullet (g_1^{*}\phi_{g_2,g_3} \circ \id) = \phi_{g_1g_2,g_3} \bullet (\id \circ \phi_{g_1,g_2})\text{.}
\end{equation*}
As explained  in \cite[Def.~2.6]{nikolaus2}, a whole bicategory $\stwoVectBdl_k^{G}(X)$ of $G$-equivariant super 2-vector bundles can be constructed in a canonical way.   

\begin{remark}
The 1-isomorphism $\mathscr{P}$ in the structure of a $G$-equivariant super 2-vector bundle may of course be (induced by) a refinement $\mathscr{R}$, see \cref{sec:framing}. In this case, the 2-isomorphism $\phi$ may be an equality, as refinements form a 1-category. Let us briefly spell out what such a \emph{strict} $G$-equivariant structure is; for this purpose we  denote the involved structure by $\mathscr{V}=(\pi,\mathcal{A},\mathcal{M},\mu)$ and $\mathscr{R}=(\tilde\rho,\varphi,\phi)$. First, $\tilde\rho$ induces a lift of the $G$-action along $\pi:Y \to X$, and hence,  $G$-actions on all fibre products $Y^{[k]}$. Second, $\varphi$ induces a $G$-equivariant structure on the super algebra bundle $\mathcal{A}$ over $Y$. Third, $\phi$ induces a {compatible} $G$-equivariant structure on the bimodule bundle $\mathcal{M}$ over $Y^{[2]}$ in such a way that the intertwiner $\mu$ is $G$-equivariant.
Summarizing, a  $G$-equivariant structure on a super 2-vector bundle $\mathscr{V}$ may in simple cases  consist of lifts of the $G$-action to all of its structure.  
\end{remark}

Maybe the most useful statement about equivariant structure is that it descends to quotients whenever well-behaved quotients exist. In the present case, we have the following statement.

\begin{theorem}
Suppose a Lie group $G$ acts freely and properly on a smooth manifold $X$.
Then, pullback along the projection $X \to X/G$ induces an equivalence of bicategories between super 2-vector bundles on $X/G$ and $G$-equivariant super 2-vector bundles on $X$,
\begin{equation*}
\stwoVectBdl_k(G/X) \cong \stwoVectBdl_k^{G}(X)\text{.}
\end{equation*}
\end{theorem}  

\begin{proof}
We derive this from the abstract theory developed in \cite{nikolaus2} and the fact that by our construction, super 2-vector bundles form a 2-stack. Suppose $\mathscr{F}$ is a presheaf of bicategories on the category of smooth manifolds; we will denote its canonical extension to Lie groupoids by the same symbol. This is justified by the fact that, if $X$ is a smooth manifold and $X_{dis}$ denotes the {discrete} Lie groupoid with objects $X$ and only identity morphisms, then $\mathscr{F}(X_{dis})\cong\mathscr{F}(X)$. Any smooth functor $F: \mathscr{X} \to \mathscr{Y}$ between Lie groupoids induces a functor $F^{*}:\mathscr{F}(\mathscr{Y}) \to \mathscr{F}(\mathscr{X})$ between the corresponding bicategories. If $\mathscr{F}$ is a 2-stack, and $F$ is a weak equivalence, then $F^{*}$ is an equivalence $\mathscr{F}(X) \cong \mathscr{F}(Y)$ \cite[Thm.~2.16]{nikolaus2}. Now, in the present situation we consider the evident smooth functor $X/\!/ G \to (X/G)_{dis}$, which is a weak equivalence as the projection $X \to X/G$ is a surjective submersion. Combining and evaluating these facts for the 2-stack $\stwoVectBdl_k$ of super 2-vector bundles yields the claim. 
\end{proof}

\section{Classification of 2-vector bundles}

\label{sec:class2vect}

In this section, we classify super 2-vector bundles of a fixed Morita class $A$, see \cref{def:2vboffixedMoritaclass}. Our classification is based on an idea of Pennig \cite{Pennig2011} and uses the automorphism 2-group of a super algebra $A$. In the first two subsections, we neglect the monoidal structure and classify 2-vector bundles up to isomorphism  as a set. The monoidal structure is then added in \cref{sec:monoidalstructure}.
\Cref{sec:classificationof2linebundles} treats 2-line bundles, for which the classification simplifies.

\subsection{Non-abelian cohomology for algebras}

\label{sec:nonabcohfa}

Let $A$ be a Picard-surjective super algebra (see \cref{Section2VectorSpaces}). In \cite[\S 2.3]{Kristel2022} we have shown that the automorphism 2-group of $A$ as an object in the bicategory $\stwoVect_k$ can be represented by a  crossed module of Lie groups, denoted by $\AUT(A)$.
This crossed module will be central for the classification of super 2-vector bundles. We refer to \cref{sec:crossedmodules} for a quick recollection of  crossed modules.
The crossed module $\AUT(A)$ consists of the Lie group $A^{\times}_0$ of even invertible elements of $A$ and of the Lie group $\Aut(A)$ of even automorphisms of $A$, together with the Lie group homomorphism $i: A_0^{\times} \to \Aut(A)$ that associates to an element $a\in A_0^{\times}$ the conjugation $i(a)$ by $a$, and the action of $\Aut(A)$ on $A_0^{\times}$ by evaluation. 
\begin{comment}
That this structure forms a crossed module is guaranteed by the two evident conditions
\begin{equation*}
i({\varphi(a)})=\varphi \circ i(a) \circ \varphi^{-1}
\quad\text{ and }\quad
i(a)(b)=aba^{-1}\text{.}
\end{equation*}
\end{comment}

The \v Cech cohomology of $X$ with values in the crossed module $\AUT(A)$ is (see \cref{DefNonAbChechCohom}),
\begin{equation} \label{DefNonAbChechCohomAUTA}
\check{\mathrm{H}}^1(X,\AUT(A)):=\mathrm{h}_0(\underline{\mathscr{B}\AUT}(A)^{+}(X))\text{,}
\end{equation}
\ie we consider the Lie 2-groupoid $\mathscr{B}\AUT(A)$ with a single object that is associated to the crossed module $\AUT(A)$, the presheaf of bicategories $\sheaf{\mathscr{B}\AUT}(A)$ represented by $\mathscr{B}\AUT(A)$,  apply the plus-construction, evaluate on $X$, and then take the set of isomorphism classes of objects.
Spelling this out, see \cref{sec:crossedmodules}, an element is represented with respect to an open cover $\{U_{\alpha}\}_{\alpha\in A}$ by a pair $(\varphi,a)$ where $\varphi$ is a collection of smooth maps $\varphi_{\alpha\beta}:U_{\alpha} \cap U_{\beta} \to \Aut(A)$ and $a$ is a collection of smooth maps $a_{\alpha\beta\gamma}:U_{\alpha} \cap U_{\beta}\cap U_{\gamma} \to A_0^{\times}$, such that the cocycle conditions 
\begin{equation} \label{CocycleConditionsNonAb}
i(a_{\alpha\beta\gamma})\circ \varphi_{\beta\gamma} \circ \varphi_{\alpha\beta}=\varphi_{\alpha\gamma}
\quad\text{ and }\quad
a_{\alpha\gamma\delta}\cdot \varphi_{\gamma\delta}(a_{\alpha\beta\gamma})=a_{\alpha\beta\delta}\cdot a_{\beta\gamma\delta}
\end{equation}  
are satisfied. Two cocycles $(\varphi,a)$  and $(\varphi',a')$ are equivalent, if, after passing to a common refinement of the open covers, there exist smooth maps $\varepsilon_{\alpha}:U_{\alpha} \to \Aut(A)$ and $e_{\alpha\beta}:U_{\alpha}\cap U_{\beta}\to A^{\times}$ satisfying
\begin{equation*}
i(e_{\alpha\beta})\circ \varepsilon_{\beta}\circ \varphi_{\alpha\beta}=\varphi'_{\alpha\beta}\circ \varepsilon_{\alpha}
\quad\text{ and }\quad
a'_{\alpha\beta\gamma}\cdot \varphi'_{\beta\gamma}(e_{\alpha\beta})\cdot e_{\beta\gamma}= e_{\alpha\gamma}\cdot \varepsilon_{\gamma}(a_{\alpha\beta\gamma})\text{.}
\end{equation*}

We note the following result about the \v Cech cohomology of Picard-surjective super algebras.

\begin{proposition}
\label{prop:moritainvariance2group}
Suppose $A$ and $B$ are Picard-surjective super algebras. Then, every invertible $A$-$B$-bimodule induces a  bijection $\check{\mathrm{H}}^1(X,\AUT(A)) \cong \check{\mathrm{H}}^1(X,\AUT(B))$. In particular, these  sets coincide whenever $A$ and $B$ are Morita equivalent Picard surjective super algebras.
\end{proposition}

\begin{proof}
By \cite[Prop 2.3.3]{Kristel2022}, any invertible $A$-$B$-bimodule $M$ determines an invertible butterfly  between the crossed modules $\AUT(A)$ and $\AUT(B)$, see \cref{sec:butterflies}. 
Weak equivalences induce bijections in cohomology, see \cref{prop:butterflyco}.
\end{proof}

\subsection{The classification}

Let  $A$ be a super algebra. We consider  the presheaf of bicategories $A\text{-}\sAlgBdlbi$ introduced in \cref{re:moritaclass:a}, consisting of super algebra bundles whose fibres are Morita equivalent to $A$, all implementing bimodule bundles and all even intertwiners. On the other hand, we consider the presheaf of bicategories $\underline{\mathscr{B}\AUT}(A)$ mentioned in \cref{sec:nonabcohfa}.
For a smooth manifold $X$, the bicategory  $\underline{\mathscr{B}\AUT}(A)(X)$ has a single object, the 1-morphisms are smooth maps $X \to \AUT(A)$, and the 2-morphisms are smooth maps $X \to A_0^{\times} \times \Aut(A)$. Here, a pair $(a,\varphi)$ with $a: X \to A_0^{\times}$ and $\varphi:X \to \Aut(A)$, is a 2-morphism from $x \mapsto \varphi(x)$ to $x \mapsto i(a(x)) \circ \varphi(x)$. The vertical composition of 2-morphisms is $(a_2,\varphi_2) \circ (a_1,\varphi_1) := (a_2a_1,\varphi_1)$. The horizontal composition of 1-morphisms is given by the group structure on $\Aut(A)$, and the one of 2-morphisms is given by the semi-direct product $(a_2,\varphi_2) \cdot (a_1,\varphi_1) := (a_2\varphi_2(a)),\varphi_2\circ \varphi_1$; see \cref{sec:crossedmodules} for the general description.

In the following, we  describe a  morphism
\begin{equation}
\label{eq:functorF}
F: \underline{\mathscr{B}\AUT}(A) \to A\text{-}\sAlgBdlbi
\end{equation}
of presheaves of bicategories. 
Over a smooth manifold $X$, 
it sends the unique object
of $\underline{\mathscr{B}\AUT}(A)(X)$ to the trivial algebra bundle $F_X(\ast) := \underline{A}$.
It sends a 1-morphism, \ie a smooth map $\varphi:X \to \Aut(A)$ to the (implementing) bimodule bundle $F_X(\varphi) :=\mathcal{M}_{\varphi} := \underline{A}_{\varphi}$ (see \cref{ExampleTwistedModuleIso}). 
Finally, it sends a 2-morphism, \ie a pair $(a,\varphi)$ of a smooth map $a:X \to A^{\times}$ and a smooth map $\varphi: X \to \Aut(A)$, 
to the intertwiner $F_X(a,\varphi) :=\phi_{a,\varphi}: \mathcal{M}_{\varphi} \to \mathcal{M}_{i(a)\circ \varphi}$ given in each fibre by the canonical isomorphism 
\begin{equation*}
A_{\varphi_x} \cong A \otimes_A A_{\varphi_x} \cong A_{i(a(x))} \otimes_A A_{\varphi_x} \cong A_{i(a(x))\circ \varphi_x}\text{;}
\end{equation*}
explicitly,  $\phi_{a,\varphi}(x,b):=(x,ba^{-1})$.

\begin{lemma}
Above definitions yield a 2-functor $F_X: \underline{\mathscr{B}\AUT}(A)(X) \to A\text{-}\sAlgBdlbi(X)$.
\end{lemma}

\begin{proof} 
$F_X$ preserves the composition of 1-morphisms under the compositor 
\begin{equation*}
\mathcal{M}_{\varphi_2} \circ \mathcal{M}_{\varphi_1} \cong \mathcal{M}_{\varphi_2\circ \varphi_1},
\end{equation*}
given fibrewise by the map $A_{\varphi_2} \otimes_A A_{\varphi_1} \cong A_{\varphi_2 \circ\varphi_1}:a\otimes b\mapsto a\varphi_2(b)$. Moreover, 
$F_X$ preserves the vertical composition of 2-morphisms because of the equality $\phi_{a_2,i(a_1)\circ \varphi} \circ \phi_{a_1,\varphi} = \phi_{a_2a_1,\varphi}$ of intertwiners, which follows immediately from the definition of $\phi_{a,\varphi}$. 
\\
It is less obvious that $F_X$ preserves the horizontal composition of 2-morphisms. Consider two 2-morphisms $(a_1,\varphi_1): \varphi_1 \Rightarrow i(a_1)\circ \varphi_1$ and $(a_2,\varphi_2): \varphi_2 \Rightarrow i(a_2)\circ \varphi_2$, whose horizontal composition is, cf. \cref{eq:semidirect}: 
\begin{equation*}
(a_2\varphi_2(a_1),\varphi_2\circ \varphi_1): \varphi_2\circ \varphi_1 \Rightarrow i(a_2\varphi_2(a_1)) \circ \varphi_2 \circ \varphi_1\text{.}
\end{equation*}
On the other side, we consider the images under $F_X$, i.e., the intertwiners $\phi_{a_1,\varphi_1}: \mathcal{M}_{\varphi_1} \to \mathcal{M}_{i(a_1) \circ \varphi_1}$ and $\phi_{a_2,\varphi_2}: \mathcal{M}_{\varphi_2} \to \mathcal{M}_{i(a_2) \circ \varphi_2}$ and their tensor product 
\begin{equation*}
\phi_{a_2,\varphi_2} \otimes \phi_{a_1,\varphi_1}: \mathcal{M}_{\varphi_2} \otimes_{\mathcal{A}} \mathcal{M}_{\varphi_2} \to \mathcal{M}_{i(a_2) \circ \varphi_2} \otimes_{\mathcal{A}}  \mathcal{M}_{i(a_1) \circ \varphi_1}\text{.}
\end{equation*}
The condition that this corresponds, under the compositors, to $\phi_{a_2\varphi_2(a_1),\varphi_2\circ\varphi_1}$ is the commutativity of the diagram
\begin{equation*}
\xymatrix{A_{\varphi_2} \otimes_A A_{\varphi_1} \ar[r] \ar[d] & A_{i(a_2)\circ \varphi_2} \otimes_A A_{i(a_1) \circ \varphi_1} \ar[d] \\ A_{\varphi_2\circ\varphi_1} \ar[r] & A_{i(a_2\varphi_2(a_1)) \circ \varphi_2 \circ \varphi_1}}
\quad\quad
\xymatrix{b\otimes c\ar@{|->}[r] \ar@{|->}[d] & b a_2^{-1} \otimes c a_1^{-1} \ar@{|->}[d] \\ b\varphi_2(c) \ar@{|->}[r] & b\varphi_2(c)\varphi_2(a_1)^{-1}a_2^{-1}\text{.}}
\end{equation*}
This is indeed commutative; this finishes the proof that $F_X$ is a 2-functor.
\end{proof}

It is obvious that $F_X$ is compatible with the pullback along smooth maps between contractible manifolds; hence $X \mapsto F_X$ is indeed a morphism $F$ of presheaves of bicategories.  

Let $\mathrm{Grpd}(A\text{-}\sAlgBdlbi)$ denote  the 2-groupoidification of the presheaf of super algebra bundles of Morita class $A$, obtained by discarding all non-invertible 1-morphisms and all non-invertible intertwiners. 
By \cite[Lemmas 2.1.3 and 4.2.8]{Kristel2022}, the  morphism $F$ factors through the inclusion $\mathrm{Grpd}(A\text{-}\sAlgBdlbi)\subset A\text{-}\sAlgBdlbi$. We denote by $\contrMfd\subset \Mfd$ the full subcategory of all smooth manifolds with all connected components contractible, and infer the following result.

\begin{lemma}
\label{lem:presheavesoncontractible}
Suppose $A$ is a Picard-surjective super algebra.
Then, the morphism $F$ of \cref{eq:functorF} induces an isomorphism
\begin{equation*}
\underline{\mathscr{B}\AUT}(A)|_{\contrMfd} \cong \mathrm{Grpd}(A\text{-}\sAlgBdlbi)|_{\contrMfd}
\end{equation*}
between the restrictions of both presheaves to the  subcategory $\contrMfd \subset \Mfd$.
\end{lemma}

\begin{proof}
It suffices to show that the 2-functor $F_X$ is an equivalence of bicategories whenever $X$ is contractible. 
That $F_X$ is essentially surjective comes from the fact that algebra bundles over contractible manifolds are trivializable. This can be proved analogously to the corresponding fact for vector bundles, see \cite[Prop.~1.7 and Cor. 1.8]{Hatcher2017}. Next we prove that the Hom-functor 
\begin{equation*}
F_X(\ast,\ast): \underline{\AUT}(A)(X) \to \underline{A}\text{-}\underline{A}\text{-}\sBimodBdl(X)^{\times}
\end{equation*}
is an equivalence of categories. Here, $(..)^{\times}$ denotes the subcategory of invertible bimodule bundles and invertible intertwiners, which appears here because of the 2-groupoidification. 
First of all, by the same argument as before, every  $\underline{A}$-$\underline{A}$-bimodule bundle is trivializable as a vector bundle over the contractible manifold $X$, and hence, by \cref{def:bimodulebundle}, of the form $\pss{\phi}{}\underline{M}_\psi$, where $M$ is an (invertible) $A$-$A$-bimodule and $\phi, \psi: X \to \Aut(A)$ are smooth maps. 
As $A$ is Picard-surjective, we have $M \cong A_{\varphi_0}$ for some $\varphi_0 \in \Aut(A)$.
We obtain an isomorphism of module bundles
\begin{equation*}
\pss{\phi}{}\underline{M}_\psi \cong \pss{\phi}{}\underline{A}_{\varphi_0 \circ \psi} \cong \underline{A}_{\phi^{-1} \circ \varphi_0 \circ \psi}\text{,}
\end{equation*}
where  the last isomorphism is obtained by  applying the isomorphism of bimodules $\pss{\varphi}{}A \mapsto A_{\varphi^{-1}}$, $a \mapsto \varphi^{-1}(a)$ fibrewise.
This shows that $F_X(\ast,\ast)$ is essentially surjective. Finally, on the level of 2-morphisms, that $F_X$ is fully faithful follows from \cite[Lemma 2.1.3 (d)]{Kristel2022}.
\end{proof}

If two presheaves of bicategories become isomorphic when restricted to the subcategory $\contrMfd$ of manifolds with contractible connected components, then the plus construction will associate to them \emph{isomorphic} presheaves. The reason for this is the existence of good open covers on manifolds, in combination with \cref{lem:refinementofss}. More precisely, let us denote by $\mathscr{F}^{c+}$ a variant of the plus construction in which the domains $Y$, $Z$, $W$ of all surjective submersions that appear in the description given in \cref{sec:2stack2vec} are objects of $\contrMfd$. On one side, we obtain an evident inclusion $\mathscr{F}^{c+} \to \mathscr{F}^{+}$ of presheaves, and we claim that this is an isomorphism.  For instance, essential surjectivity follows from \cref{lem:refinementofss} by choosing a good open cover of $X$ with local sections into the surjective submersion $\pi:Y \to X$ of a given super 2-vector bundle. Analogous considerations show essential surjectivity on the level of 1-morphisms and surjectivity on the level of 2-morphisms. On the other side, the presheaf $\mathscr{F}^{c+}$ evaluates the given presheaf $\mathscr{F}$ only on objects of $\contrMfd$. This proves above claim, and thus, \cref{lem:presheavesoncontractible} implies the following result. 

\begin{proposition}
\label{prop:classpicsur}
Let $A$ be a Picard-surjective super algebra. Then, the morphism $F$ of \cref{eq:functorF} induces an isomorphism of 2-stacks
\begin{equation*}
\underline{\mathscr{B}\AUT}(A)^{+} \cong \mathrm{Grpd}(A\text{-}\sAlgBdlbi)^{+}\text{.}
\end{equation*}
\end{proposition}

Now we are in position to present our classification theorem.

\begin{theorem}
\label{th:class2vect}
For any Picard-surjective super algebra $A$, there is a canonical bijection
\begin{equation*}
\mathrm{h}_0(A\text{-}\stwoVectBdl(X)) \cong \check{\mathrm{H}}^1(X,\AUT(A)).
\end{equation*}
In other words, super 2-vector bundles over $X$ of Morita class $A$ are classified by the \v Cech cohomology of $X$ with values in the crossed module $\AUT(A)$.
\end{theorem}

\begin{proof}
We note that
\begin{align*}
\mathrm{h}_0(A\text{-}\stwoVectBdl(X)) &= \mathrm{h}_0((A\text{-}\sAlgBdlbi)^{+}(X))
\\&= \mathrm{h}_0(\mathrm{Grpd}((A\text{-}\sAlgBdlbi)^{+}(X)))
\\&= \mathrm{h}_0(\mathrm{Grpd}(A\text{-}\sAlgBdlbi)^{+}(X))\text{,}
\end{align*}
where the first equality holds by definition of super 2-vector bundles, and the second holds because 2-groupoidification preserves the set of isomorphism classes of objects, and the third holds because 2-groupoidification commutes with the plus construction (this follows from \cref{LemmaInvertibility}). 
The claim follows then from \cref{prop:classpicsur} and  \cref{DefNonAbChechCohomAUTA}.
\end{proof}

\begin{remark}
The condition of being Picard-surjective can be achieved at any time by passing to a Morita equivalent super algebra: this is possible because every super algebra is Morita equivalent to a Picard-surjective one (\cite[Prop.~A.2]{Kristel2022})  and $\mathrm{h}_0(A\text{-}\stwoVectBdl(X))$ is Morita invariant by \cref{lem:Moritaclass2vect}.
\end{remark}

A result of Baez and Stevenson \cite{baez8} shows that the geometric realization $|\Gamma|$ of  (the Lie 2-group associated to) a crossed module $\Gamma$ is a topological group, whose classifying space $\mathrm{B}|\Gamma|$  represents the cohomology with values in $\Gamma$, 
\begin{equation*}
  \check{\mathrm{H}}^1(X,\Gamma) \cong [X,\mathrm{B}|\Gamma|].
\end{equation*}
Combining this with \cref{th:class2vect}, we obtain the following.  

\begin{corollary}
For any Picard-surjective super algebra $A$, $\mathrm{B}|\AUT(A)|$ is a classifying space for super 2-vector bundles of Morita class $A$.
\end{corollary}

For the convenience of the reader, we shall spell out explicitly  procedures that realize the bijection of \cref{th:class2vect}, obtained by passing through all intermediate steps described above.
Given a cocycle $(\varphi,a)$ in $\check{\mathrm{H}}^1(X,\AUT(A))$, we construct the following super 2-vector bundle over $X$. 
Its surjective submersion $Y$ is the disjoint union of the open sets $U_{\alpha}$ of the cover on which $(\varphi,a)$ is defined. The algebra bundle is the trivial algebra bundle $\underline{A}$ over $Y$. 
Over $Y^{[2]}$, which is the disjoint union of double intersections $U_{\alpha}\cap U_{\beta}$, we have the map $\varphi:Y^{[2]} \to \Aut(A)$, to which we associate the bimodule bundle  $\underline{A}_{\varphi}$, the fibre of which is $A_{\varphi_x}$. 
Over $Y^{[3]}$, we have the three automorphisms $\varphi_{12}$, $\varphi_{23}$ and $\varphi_{13}$ defined by $\varphi_{ij}:= \pr_{ij}^{*}\varphi$, and we have $i(a) \circ \varphi_{23}\circ \varphi_{12} = \varphi_{13}$ due to the cocycle condition, where $a: Y^{[3]} \to A^{\times}$. 
At each point, $a$ defines an invertible even intertwiner of $A$-$A$-bimodules  
\begin{equation*}
\underline{A}_{\varphi_{23}} \otimes_{\underline{A}} \underline{A}_{\varphi_{12}}\cong \underline{A}_{\varphi_{23} \circ \varphi_{12}}\cong \underline{A} \otimes_{\underline{A}} \underline{A}_{\varphi_{23}\circ \varphi_{12}}\cong \underline{A}_{i(a)} \otimes_{\underline{A}}  \underline{A}_{\varphi_{23}\circ\varphi_{12}}\cong \underline{A}_{i(a) \circ \varphi_{23} \circ \varphi_{12}}\cong \underline{A}_{\varphi_{13}}\text{.}
\end{equation*}
These form the required associative intertwiner $\mu: \underline{A}_{\varphi_{23}} \otimes_{\mathcal{A}} \underline{A}_{\varphi_{12}} \to \underline{A}_{\varphi_{13}}$ of bimodule bundles over $Y^{[3]}$.

Conversely, given a super 2-vector bundle $\mathscr{V} = (\pi, \mathcal{A}, \mathcal{M}, \mu)$ in $A\text{-}\stwoVectBdl(X)$, we 
may first pass via \cref{prop:trivialbundles} to an isomorphic one where $\mathcal{A}$ and $\mathcal{M}$ are trivial in the sense of \cref{re:typicalfibre,ExampleTwistedModuleIso}.
On each connected component $X_i$, the trivial super algebra bundle has a typical fibre $A_i$ which is Morita equivalent to $A$. 
Choosing invertible $A$-$A_i$-bimodules $M_i$, one can construct a 1-isomorphism that takes us to a super 2-vector bundle whose super algebra bundle is the trivial bundle $\underline{A}$ over $Y$ and whose bimodule bundle  is still trivial. 
Because of Picard-surjectivity of $A$, we may then identify the  bimodule bundle with $\underline{A}_{\varphi}$, where $\varphi$ is a map $\varphi:Y^{[2]}\to \Aut(A)$ (see the argument in the proof of \cref{lem:presheavesoncontractible}). We write again $\varphi_{ij}:= \pr_{ij}^{*}\varphi$.  Over $Y^{[3]}$, we obtain an invertible even intertwiner
\begin{equation*}
\mu: \underline{A}_{\varphi_{23}} \otimes_{\underline{A}} \underline{A}_{\varphi_{12}}\to \underline{A}_{\varphi_{13}}
\end{equation*}
of $\underline{A}$-$\underline{A}$-bimodule bundles. 
Under the isomorphism $\underline{A}_{\varphi_{23}} \otimes_{\underline{A}} \underline{A}_{\varphi_{12}}\cong \underline{A}_{\varphi_{23} \circ \varphi_{12}}$, this becomes an even intertwiner $\underline{A}_{\varphi_{23} \circ \varphi_{12}}\cong \underline{A}_{\varphi_{13}}$ which by \cite[Lemma 2.1.3 (d)]{Kristel2022} corresponds to a unique smooth map $a:Y^{[3]} \to A^{\times}$ such that $i(a)\circ \varphi_{23}\circ \varphi_{12}= \varphi_{13}$. 
Finally, using \cref{lem:refinementofss} one can now achieve that $Y$ is the disjoint union of open sets; then, $(\varphi,a)$ is a cocycle representing the class of $\mathscr{V}$ in $\check{\mathrm{H}}^1(X,\AUT(A))$.

\begin{remark}
\label{rem:ungradedps}
Ungraded algebras are rarely Picard surjective when regarded as super algebras concentrated in even degrees. For that reason, one cannot simply apply the results of this section to ungraded algebras. However, one can proceed in close analogy. An ungraded algebra is called \emph{ungraded Picard-surjective} in the sense that every \emph{ungraded} invertible $A$-$A$-bimodule is induced from an automorphism of $A$. The presheaf morphism $F$ becomes a morphism
\begin{equation*}
F: \underline{\mathscr{B}\AUT}(A) \to A\text{-}\AlgBdlbi
\end{equation*}
to \emph{ungraded} algebra bundles. The ungraded analogue of
\Cref{lem:presheavesoncontractible} holds, and hence \cref{prop:classpicsur} has as an ungraded analogue an isomorphism   
\begin{equation*}
\underline{\mathscr{B}\AUT}(A)^{+} \cong \mathrm{Grpd}(A\text{-}\AlgBdlbi)^{+}\text{.}
\end{equation*}
We obtain a classification result analogous to \cref{th:class2vect}, 
\begin{equation*}
\mathrm{h}_0(A\text{-}\twoVectBdl(X)) \cong \check{\mathrm{H}}^1(X,\AUT(A))\text{.}
\end{equation*} 
Likewise, $\mathrm{B}|\AUT(A)|$ is a classifying space for ungraded 2-vector bundles. 
\end{remark}

To close this section, we  discuss the  canonical inclusion $\mathscr{B}Z(A)_0^{\times} \to \AUT(A)$, its induced map in cohomology, and its geometric interpretation by super 2-vector bundles. 
Here, $\mathscr{B}Z(A)_0^{\times}$ denotes the crossed module $Z(A)_0^{\times} \to \ast$, with the trivial action. 
 The inclusion $\mathscr{B}Z(A)_0^{\times} \to \AUT(A)$ is the strict homomorphism of crossed modules given by the inclusion $Z(A)_0^{\times}\subset A_0^{\times}$. Note that the cohomology with values in the crossed module  $\mathscr{B}Z(A)_0^{\times}$ is the ordinary degree two \v Cech cohomology with values in the sheaf of smooth $Z(A)_0^{\times}$-valued functions (see \cref{RemarkCheckCohomologyAppendix}). 
 Thus, the map induced by $\mathscr{B}Z(A)_0^{\times} \to \AUT(A)$ is a map 
\begin{equation*}
\check{\mathrm{H}}^2(X,\sheaf{Z(A)}^{\times}_0)\to \check{\mathrm{H}}^1(X,\AUT(A))\text{,}
\end{equation*}
and it sends a \v Cech cocycle $a_{\alpha\beta\gamma}:U_{\alpha}\cap U_{\beta}\cap U_{\gamma} \to Z(A)_0^{\times}$ to the $\AUT(A)$-cocycle $(1,a)$.

For any abelian Lie group $K$, the cohomology group $\check{\mathrm{H}}^2(X,K)$ classifies  $K$-principal bundle gerbes over $X$. These are principal bundle versions of the bundle gerbes discussed in \cref{sec:bundlegerbes}, and defined as
\begin{equation*}
\Grb_{K} := \mathscr{B}(\Bdl {K})^{+}\text{,}
\end{equation*} 
where $\Bdl K$ is the monoidal stack of principal $K$-bundles. In our case, $K=Z(A)_0^{\times}$ comes with a monoidal functor
\begin{equation*}
\Bdl {Z(A)_0^{\times}}(X) \to A\text{-}A\text{-}\sBimodBdl(X)\text{,}
\end{equation*}
obtained by associating to a principal $Z(A)_0^{\times}$-bundle $\mathcal{Z}$ the vector bundle $\mathcal{Z} \times_{Z(A)_0^{\times}} A$, which becomes a bundle of $A$-$A$-bimodules in the obvious way. In turn, we obtain
a morphism of presheaves $\mathscr{B}(\Bdl{Z(A)_0^{\times}}) \to A\text{-}\sAlgBdlbi$, which under the plus construction yields a morphism
\begin{equation}
\label{eq:bundlegerbesto2vectbdl}
\Grb_{Z(A)_0^{\times}} \to A\text{-}\stwoVectBdl\text{,}
\end{equation}
for any super algebra $A$. In particular, if $A$ is central (i.e., $Z(A)_0^{\times}=k^{\times}$) this 2-functor coincides with the 2-functor \cref{eq:ungbgii2vb}.
The following result follows directly from either the cocycle description or the abstract stackification procedure.

\begin{proposition}
Let $A$ be a Picard-surjective super algebra. 
The map in cohomology induced by the inclusion $\mathscr{B}Z(A)_0^{\times} \to \AUT(A)$ corresponds to the 2-functor \eqref{eq:bundlegerbesto2vectbdl} that sends principal $Z(A)_0^{\times}$-bundle gerbes to super 2-vector bundles. In other words, the diagram
\begin{equation*}
\xymatrix{\mathrm{h_0}(\Grb_{Z(A)_0^{\times}} (X)) \ar[d]_{\cong} \ar[r] & \mathrm{h_0}(A\text{-}\stwoVectBdl_k(X)) \ar[d]^{\cong} \\  \check{\mathrm{H}}^2(X,\sheaf{Z(A)}^{\times}_0) \ar[r] & \check{\mathrm{H}}^1(X,\AUT(A))}
\end{equation*}  
is commutative.
\end{proposition}

\subsection{Monoidal structure}

\label{sec:monoidalstructure}

Let $A$ and $B$ be super algebras. There is a strict homomorphism
\begin{equation}
\label{eq:stricthom}
m:\AUT(A) \times \AUT(B) \to \AUT(A \otimes B)
\end{equation}
of crossed modules, given by the maps
\begin{align*}
A_0^{\times} \times B_0^{\times} &\to (A\otimes B)_0^{\times};& (a,b) &\mapsto a \otimes b
\\
\mathrm{Aut}(A) \times \mathrm{Aut}(B) &\to \mathrm{Aut}(A\otimes B);& (\varphi_1,\varphi_2) &\mapsto \varphi_1\otimes\varphi_2 
\end{align*}
\begin{comment}
It is to check that $i({a}) \otimes i({b})= i({a \otimes b})$, which is true because $a$ and $b$ are even.
For the actions, the condition is only that $\varphi(a) \otimes \psi(b) = (\varphi \otimes \psi)(a \otimes b)$.
\end{comment}
It induces a map in cohomology,
\begin{equation}
\label{eq:productnonabco}
\check{\mathrm{H}}^1(X,\AUT(A)) \times \check{\mathrm{H}}^1(X,\AUT(B)) \to \check{\mathrm{H}}^1(X,\AUT(A \otimes B))\text{.}
\end{equation}

We want to show  that this map corresponds to the tensor product of super 2-vector bundles. 
\begin{comment}
Here, we work in the symmetric monoidal 2-stack $\stwoVectBdlgrpd k-$ -- since our classification only affects isomorphism classes of objects, the restriction to invertible bimodules is irrelevant.  
\end{comment}
We claim that the presheaf morphism 
\begin{equation*}
F: \underline{\mathscr{B}\AUT}(A) \to A\text{-}\sAlgBdlbi
\end{equation*}
of \eqref{eq:functorF} is compatible with the tensor product of super algebras, in the sense that the diagram
\begin{equation} 
\label{DiagramProduct}
\begin{aligned}
\xymatrix@C=4em{
\underline{\mathscr{B}\AUT}(A) \times \underline{\mathscr{B}\AUT}(B) \ar[d]_{F \times F} \ar[r]^-{\mathscr{B}M} & \underline{\mathscr{B}\AUT}(A \otimes B) \ar[d]^{F} \\ 
A\text{-}\sAlgBdlbi \times B\text{-}\sAlgBdlbi \ar[r]_-{\otimes} & (A\otimes B)\text{-}\sAlgBdlbi
}
\end{aligned}
\end{equation}
of presheaves of bicategories is (strictly!) commutative.
On the level of objects, this is clear. 
On the level of 1-morphisms, the diagram commutes since $\underline{A}_{\varphi_1} \otimes_k \underline{B}_{\varphi_2} = (\underline{A}\otimes_k \underline{B})_{\varphi_1\otimes \varphi_2}$ (the identity on $\underline{A} \otimes \underline{B}$ is an intertwiner). 
On the level of 2-morphisms, the diagram commutes because the intertwiner $\phi_{a,\varphi}$ in the definition of $F$ on 2-morphisms satisfies (by inspection) the identity
\begin{equation*}
\phi_{a,\varphi_A} \otimes \phi_{b,\varphi_B} = \phi_{a\otimes b,\varphi_A \otimes \varphi_B}\text{;}
\end{equation*}
this is precisely the coincidence between clockwise and counter-clockwise directions.

\begin{proposition}
\label{prop:classmult}
For any pair of super algebras $A$ and $B$, the diagram
\begin{equation*}
\xymatrix{
\check{\mathrm{H}}^1(X,\AUT(A)) \times \check{\mathrm{H}}^1(X,\AUT(B)) \ar[d]_{F \times F} \ar[r] & \check{\mathrm{H}}^1(X,\AUT(A\otimes B)) \ar[d]^{F} \\ 
\mathrm{h}_0(A\text{-}\stwoVectBdl(X)) \times \mathrm{h}_0(B\text{-}\stwoVectBdl(X)) \ar[r] & \mathrm{h}_0((A\otimes B)\text{-}\stwoVectBdl(X))
}
\end{equation*}
is commutative.

\end{proposition}

\begin{proof}
In view of \cref{DefNonAbChechCohom}, the commutativity of the diagram then follows from applying the (functorial) plus-construction to the diagram \eqref{DiagramProduct} and passing to isomorphism classes.
\end{proof}

\subsection{Classification of super 2-line bundles}

\label{sec:classificationof2linebundles}

A central simple super algebra has by definition $Z(A)^{\times}_0=k^{\times}$, and $\mathrm{Pic}(A)=\Z_2$, with representatives given by $A$ and $\Pi A$, see \cite[Remark 2.2.5 (1)]{Kristel2022}. 
We consider the trivial crossed module $k^{\times} \to \Z_2$, which is given by the zero map and the trivial action of $\Z_2$ on $k^{\times}$. 
It can be written as a direct product crossed module $\mathscr{B}k^{\times} \times (\Z_2)_{dis}$.

\begin{lemma}
\label{prop:classcssa}
If $A$ is a Picard-surjective central simple super algebra over $k$, then there is a bijection
\begin{equation} \label{BijectionH1AutAH1Z2H2ktimes}
\check{\mathrm{H}}^1(X,\AUT(A)) \cong \mathrm{H}^1(X,\Z_2) \times \check{\mathrm{H}}^2(X,\underline{k}^{\times})\text{.}
\end{equation}
\end{lemma}

\begin{proof}
By \cite[Prop.~2.3.4]{Kristel2022} there is a canonical weak equivalence $\AUT(A) \cong \mathscr{B}k^{\times} \times (\Z_2)_{dis}$, whose definition we recall briefly.  
         It is established by a butterfly (see \cref{sec:butterflies})
\begin{equation} 
\label{Butterflyclasscssa}
\begin{aligned}
        \xymatrix{k^{\times} \ar[dd]_{0} \ar[dr]^{i_1} && A_0^{\times} \ar[dl]_{i_2}  \ar[dd]^{c} \\ & K \ar[dr]_{p_2} \ar[dl]^{p_1} & \\ \Z_2  && \Aut(A)  .}
\end{aligned}
\end{equation}
Here, $K$ consists of triples $(\varepsilon,u,\phi)\in\Z_2 \times \underline{\mathrm{GL}}(A) \times \Aut(A)$ where $u: \epsilon A \to A_{\phi}$ is an even invertible intertwiner of $A$-$A$-bimodules, with
\begin{equation*}
\epsilon A = \begin{cases}A & \epsilon=0 \\
\Pi A & \epsilon=1\text{.}\ \\
\end{cases}
\end{equation*}
The group homomorphisms $i_1$ and $i_2$ are defined by
$i_1(\lambda) := (0,\lambda,\id)$ and $i_2(a) := (0,r_{a^{-1}},i(a))$, 
where $r_a: A \to A$ is right multiplication by $a$. The group homomorphisms $p_1$ and $p_2$ are the projections to $\varepsilon$ and $\phi$, respectively.
Given this butterfly, \cref{prop:butterflyco} shows the claim.  
 \end{proof}

\begin{remark}
\label{re:classcssa}
On the level of cocycles, the bijection
\begin{equation}
\label{eq:bijcoh}
\check{\mathrm{H}}^1(X,\AUT(A)) \cong \mathrm{H}^1(X,\Z_2) \times \check{\mathrm{H}}^2(X,\underline{k}^{\times})
\end{equation}
is obtained as follows, see \cref{sec:cmc}.
Recall that with respect to an open cover $\{U_{\alpha}\}_{\alpha\in I}$, an element in $\check{\mathrm{H}}^1(X,\AUT(A))$ is represented by a pair $(\varphi,a)$ where $\varphi$ is a collection of smooth maps $\varphi_{\alpha\beta}:U_{\alpha} \cap U_{\beta} \to \Aut(A)$ and $a$ is a collection of smooth maps $a_{\alpha\beta\gamma}:U_{\alpha} \cap U_{\beta}\cap U_{\gamma} \to A_0^{\times}$, such that the cocycle conditions \eqref{CocycleConditionsNonAb} are satisfied. 
By passing to a smaller open cover, we may assume that there are smooth maps $\epsilon_{\alpha\beta}:U_{\alpha}\cap U_{\beta} \to \Z_2$ and $u_{\alpha\beta}:U_{\alpha}\cap U_{\beta} \to \underline{\mathrm{GL}}(A)$, such that $u_{\alpha\beta}(x):\epsilon_{\alpha\beta}(x)A \to A_{\varphi_{\alpha\beta}(x)}$ intertwines the $A$-$A$-bimodule action at each $x\in U_{\alpha}\cap U_{\beta}$. Then $\epsilon_{\alpha\beta}$ is a 2-cocycle and gives the class in $\mathrm{H}^1(X,\Z_2)$, and the linear map $u_{\alpha\gamma} \circ u_{\alpha\beta}^{-1}\circ u_{\beta\gamma}^{-1} \circ r_{\alpha_{\alpha\beta\gamma}} :A \to A$ is scalar multiplication by an element $\lambda_{\alpha\beta\gamma}: U_{\alpha}\cap U_{\beta}\cap U_{\gamma} \to k^{\times}$, which  is a 3-cocycle and gives the class in $\mathrm{H}^2(X,\underline{k}^{\times})$. 
\end{remark}

Next, we investigate how the bijection of \cref{prop:classcssa} is compatible with the tensor product of super algebras. We recall that the tensor product of central simple algebras is again central simple, and note that the tensor product of Picard-surjective central simple super algebras is again Picard-surjective (\cite[Lemma A.1]{Kristel2022}).  

\begin{lemma}
\label{lem:classcohmult}
Let $A$ and $B$ Picard-surjective central simple super algebras. Then, the diagram
\begin{equation*}
\xymatrix{\check{\mathrm{H}}^1(X,\AUT(A)) \times \check{\mathrm{H}}^1(X,\AUT(B)) \ar[d]\ar[r] & \check{\mathrm{H}}^1(X,\AUT(A\otimes B)) \ar[d] \\ (\mathrm{H}^1(X,\Z_2) \times \check{\mathrm{H}}^2(X,\underline{k}^{\times})) \times (\mathrm{H}^1(X,\Z_2) \times \check{\mathrm{H}}^2(X,\underline{k}^{\times})) \ar[r] & \mathrm{H}^1(X,\Z_2) \times \check{\mathrm{H}}^2(X,\underline{k}^{\times})}
\end{equation*}
is commutative, where the vertical arrows are the bijections of \cref{prop:classcssa}, the arrow on the top is the multiplication map of \cref{eq:productnonabco}, and the map on the bottom is defined, similarly as in \cref{th:classalg}, by
\begin{equation} \label{GroupStructureDonovanKaroubiPartial}
((\alpha_1,\alpha_2) , (\beta_1,\beta_2)) \mapsto (\alpha_1+\beta_1,(-1)^{\alpha_1\cup\beta_1}\alpha_2\beta_2)\text{.}
\end{equation}  
\end{lemma}

\begin{proof}
We prove this on the level of cocycles. Suppose $(\varphi,a)$ and $(\varphi',a')$ represent elements in $\check{\mathrm{H}}^1(X,\AUT(A))$ and $\check{\mathrm{H}}^1(X,\AUT(B))$, respectively, with respect to the same open cover. We first perform the counter-clockwise calculation.
\\
We work in the notation of the proof of \cref{prop:classcssa,re:classcssa}.
We consider $\epsilon_{\alpha\beta},\epsilon'_{\alpha\beta}:U_{\alpha}\cap U_{\beta} \to \Z_2$ as well as  $u_{\alpha\beta}:U_{\alpha}\cap U_{\beta} \to \underline{\mathrm{GL}}(A)$ and $u'_{\alpha\beta}:U_{\alpha}\cap U_{\beta} \to \underline{\mathrm{GL}}(B)$ such that $(\epsilon_{\alpha\beta},u_{\alpha\beta}, \varphi_{\alpha\beta})$ lifts $\varphi_{\alpha\beta}$ to $K_A$ and $(\epsilon'_{\alpha\beta},u'_{\alpha\beta},\varphi_{\alpha\beta})$ lifts $\varphi'_{\alpha\beta}$ to $K_B$. 
Here, $K_A$ and $K_B$ are the Lie groups in the middle of the butterflies \eqref{Butterflyclasscssa} belonging to $A$ and $B$, respectively. 
We recall from \cref{re:classcssa} that the homomorphism $u_{\alpha\gamma} \circ u_{\alpha\beta}^{-1} \circ u_{\beta\gamma}^{-1}\circ r_{a_{\alpha\beta\gamma}}:A \to A$ is scalar multiplication by a unique smooth map element $\lambda_{\alpha\beta\gamma}:U_{\alpha}\cap U_{\beta}\cap U_{\gamma} \to k^{\times}$, and that $(\epsilon, \lambda)$ is the image of $(\varphi,a)$ under the identification \eqref{BijectionH1AutAH1Z2H2ktimes}. 
The same holds for the primed quantities. 
The product of $(\epsilon, \lambda)$ with $(\epsilon',\lambda')$ is $(\epsilon+\epsilon',(-1)^{\epsilon\cup\epsilon'}\lambda\lambda')$, this is the result of the counter-clockwise calculation. 
We recall from the definition of the cup product in \v Cech cohomology  that $((-1)^{\epsilon\cup\epsilon'}\lambda\lambda')_{\alpha\beta\gamma}=(-1)^{ \epsilon_{\alpha\beta}\epsilon'_{\beta\gamma}}\lambda_{\alpha\beta\gamma}\lambda'_{\alpha\beta\gamma} $.
\\
Clockwise, we consider the product $(\varphi\otimes \varphi',a\otimes a')$ under the strict homomorphism $m$ in \eqref{eq:stricthom}. 
We have a map
\begin{equation*}
K_A \times K_B \to K_{A\otimes B}
\end{equation*}
defined by 
\begin{equation*}
((\varepsilon_A,u_A, \varphi_A),(\varepsilon_B,u_B,\varphi_B)) \mapsto (\varepsilon_A+\varepsilon_B,  u_A \otimes (\eta_{\varepsilon_A} \circ u_B), \varphi_A \otimes \varphi_B)\text{,}
\end{equation*}
where $\eta_{\varepsilon_A}:B \to B$ is defined by $b \mapsto (-1)^{\varepsilon_A |b|}b$. 
The image of our previously chosen lifts $(\varphi_{\alpha\beta},\varepsilon_{\alpha\beta},u_{\alpha\beta})$ and $(\varphi_{\alpha\beta},\varepsilon_{\alpha\beta},u_{\alpha\beta})$ under this map has $\varphi_{\alpha\beta} \otimes \varphi'_{\alpha\beta}$ in the first component, which shows that it provides a correct lift. 
In the second component, it has $\varepsilon_{\alpha\beta} + \varepsilon'_{\alpha\beta}$, which shows that our diagram is commutative in its $\mathrm{H}^1(X,\Z_2)$-factor. In the third component, it has the homomorphism $v_{\alpha\beta} := u_{\alpha\beta} \otimes (\eta_{\varepsilon_{\alpha\beta}}\circ u'_{\alpha\beta})$. We need to compute the homomorphism
\begin{equation*}v_{\alpha\gamma} \circ v_{\alpha\beta}^{-1} \circ v_{\beta\gamma}^{-1}\circ r_{a_{\alpha\beta\gamma} \otimes a'_{\alpha\beta\gamma}}:A \otimes B \to A \otimes B\text{.}
\end{equation*}
In the first tensor factor, this is just $u_{\alpha\gamma} \circ u_{\alpha\beta}^{-1} \circ u_{\beta\gamma}^{-1}\circ r_{a_{\alpha\beta\gamma}}$ and hence scalar multiplication with $\lambda_{\alpha\beta\gamma}$. 
In the second factor we compute
\begin{align*}
&\mqquad(\eta_{\varepsilon_{\alpha\gamma}} \circ u'_{\alpha\gamma} \circ u'^{-1}_{\alpha\beta} \circ \eta_{\varepsilon_{\alpha\beta}}^{-1} \circ u'^{-1}_{\beta\gamma}\circ \eta_{\varepsilon_{\beta\gamma}}^{-1})(b)
\\&=(-1)^{\varepsilon_{\beta\gamma}|b|} (\eta_{\varepsilon_{\alpha\gamma}} \circ u'_{\alpha\gamma} \circ u'^{-1}_{\alpha\beta} \circ \eta_{\varepsilon_{\alpha\beta}}^{-1} )( u'^{-1}_{\beta\gamma}(b))
\\&=(-1)^{\varepsilon_{\beta\gamma}|b|+\varepsilon_{\alpha\beta}(|b|+\varepsilon_{\beta\gamma}')} \eta_{\varepsilon_{\alpha\gamma}} (( u'_{\alpha\gamma} \circ  u'^{-1}_{\alpha\beta} \circ u'^{-1}_{\beta\gamma})(b))
\\&=(-1)^{\varepsilon_{\beta\gamma}|b|+\varepsilon_{\alpha\beta}(|b|+\varepsilon_{\beta\gamma}') + \varepsilon_{\alpha\gamma}(|b| + \varepsilon'_{\beta\gamma} + \varepsilon'_{\alpha\beta} + \varepsilon'_{\alpha\gamma})} ( u'_{\alpha\gamma} \circ  u'^{-1}_{\alpha\beta} \circ u'^{-1}_{\beta\gamma})(b)
\\&=(-1)^{\varepsilon_{\alpha\beta}\varepsilon_{\beta\gamma}' } ( u'_{\alpha\gamma} \circ  u'^{-1}_{\alpha\beta} \circ u'^{-1}_{\beta\gamma})(b)\text{,}
\end{align*} 
where the last step uses the cocycle conditions for $\varepsilon$ and $\varepsilon'$. This shows that we get $(-1)^{\varepsilon_{\alpha\beta}\varepsilon_{\beta\gamma}' }\lambda'_{\alpha\beta\gamma}$ in the second factor. This completes the proof that the diagram is commutative.
\end{proof}

The above results can be used to classify all super 2-line bundles, and the result is the following.

\begin{theorem}
\label{th:classline2bundles}
For any manifold $X$, the set $\mathrm{h}_0(\stwoLineBdl_k(X))$ of isomorphism classes of super 2-line bundles over $X$ forms an abelian group, and there is a canonical isomorphism of groups
\begin{equation}
\label{th:eqclassline2bundles}
\mathrm{h}_0(\stwoLineBdl_k(X)) \cong\mathrm{H}^0(X,\mathrm{BW}_k) \times \mathrm{H}^1(X,\Z_2) \times \check{\mathrm{H}}^2(X,\underline{k}^{\times})
\end{equation} 
with respect to the group structure \eqref{GroupStructureDonovanKaroubi}. 
Moreover, this group isomorphism extends the Donovan-Karoubi classification of central simple super algebra bundles. In other words, the diagram
\begin{equation*}
\xymatrix{\mathrm{h_0}(\cssAlgBdlbi(X)) \ar[d]_{\text{\normalfont Donovan-Karoubi, \cref{th:classalg}}} \ar[r]^{\cref{InclusionFunctorAlgebraBundles}} & \mathrm{h_0}(\stwoLineBdl_k(X)) \ar[d]^{\eqref{th:eqclassline2bundles}} \\ \mathrm{H}^0(X,\mathrm{BW}_k) \times \check{\mathrm{H}}^1(X,\Z_2) \times \mathrm{Tor}(\check{\mathrm{H}}^2(X,\underline{k}^{\times})) \ar@{^(->}[r] & \mathrm{H}^0(X,\mathrm{BW}_k) \times\check{\mathrm{H}}^1(X,\Z_2) \times \check{\mathrm{H}}^2(X,\underline{k}^{\times})\text{.}}
\end{equation*}  
of groups and group homomorphisms is commutative. 
\end{theorem}

\cref{th:classline2bundles} is also a result of Mertsch's PhD thesis \cite[Thm.~2.2.6]{Mertsch2020},  obtained there by  explicitly extracting cocycles from 2-line bundles (called {\quot{algebra bundle gerbes} there}) and a reconstruction procedure. Here we have presented it as a consequence of our more general classification result \cref{th:class2vect} and our computations of \v Cech cohomology groups. 

\begin{proof}
 If $A$ is a Picard-surjective central simple super algebra, combining the classification result of \cref{th:class2vect} with \cref{prop:classcssa}, we obtain a bijection
\begin{equation*}
\mathrm{h}_0(A\text{-}\stwoVectBdl(X)) \cong \check{\mathrm{H}}^1(X,\Z_2) \times \check{\mathrm{H}}^2(X,\underline{k}^{\times})\text{.} 
\end{equation*}
Since every super 2-line bundle over a connected manifold has a unique Morita class (\cref{lem:Moritaclass2vect}) and  each Morita equivalence class has a Picard-surjective representative (\cite[Prop.~A.2]{Kristel2022}), we have for a connected smooth manifold $X$
\begin{equation*}
\mathrm{h}_0(\stwoLineBdl_k(X)) = \coprod_{[A]\in \mathrm{BW}_k}\mathrm{h}_0(A\text{-}\stwoVectBdl_k(X))\text{.}
\end{equation*} 
Combining these two results, we obtain a bijection
\begin{equation*}
\mathrm{h}_0(\stwoLineBdl_k(X)) \cong \mathrm{BW}_k \times \mathrm{H}^1(X,\Z_2) \times \check{\mathrm{H}}^2(X,\underline{k}^{\times})\text{.} 
\end{equation*}
Over a general, not necessarily connected, manifold $X$ this gives a bijection
\begin{equation*}
\mathrm{h}_0(\stwoLineBdl_k(X)) \cong \mathrm{H}^0(X,\mathrm{BW}_k) \times \mathrm{H}^1(X,\Z_2) \times \check{\mathrm{H}}^2(X,\underline{k}^{\times})\text{.}
\end{equation*}
\cref{lem:classcohmult} and \cref{prop:classmult} imply then that this bijection becomes a homomorphism of monoids, upon declaring the monoid structure on the right hand side to be given by the formula \eqref{GroupStructureDonovanKaroubi}, extending \eqref{GroupStructureDonovanKaroubiPartial}.
 Indeed, one can check by an explicit calculation that the inclusion of super algebra bundles into super 2-line bundles via the functor \eqref{InclusionFunctorAlgebraBundles} of \cref{sec:inclusionofalgebrabundles} corresponds precisely to this inclusion of groups. 
\end{proof}

In the remainder of this subsection we derive some consequences of \cref{th:classline2bundles}. First of all, we remark the following fact about elements of a group.

\begin{corollary}
\label{co:line2bundleinvertible}
Every super 2-line bundle is invertible with respect to the tensor product.
\end{corollary}

Next we look at the relation with super algebra bundles. 
In case of $k=\R$, we see that all elements of $\check{\mathrm{H}}^2(X,\underline{k}^{\times}) \cong \h^2(X, \Z_2)$ are torsion, so that the map at the bottom of the diagram in  \cref{th:classline2bundles} is the identity. This means that the plus construction, by which we passed from central simple algebra bundles to super 2-line bundles, has, up to isomorphism, not added any new objects. Hence we have the following.

\begin{corollary}
\label{co:csgastack}
The pre-2-stack
$\cssAlgBdlbi_{\R}$  of real central simple super algebra bundles is a 2-stack, and the canonical inclusion \begin{equation*}
\cssAlgBdlbi_{\R}\to \stwoLineBdl_{\R}
\end{equation*}
is an isomorphism of 2-stacks.
\end{corollary}

In spite of \cref{co:csgastack}, it still makes sense to use the bigger bicategory $\stwoLineBdl_{\R}$ compared to $\cssAlgBdlbi_{\R}$. For example, $\R$-bundle gerbes are elements in $\stwoLineBdl_{\R}$ but do not determine \emph{canonical} algebra bundles.

\begin{remark}
\label{re:prestwolinebdlnota2stack}
\label{re:notastack}
In case of $k=\C$, we see that the classification of complex super 2-line bundles does not coincide with the classification of the bicategory of algebra bundles: super 2-line bundles may represent non-torsion elements in $\check{\mathrm{H}}^2(X,\underline{\C}^{\times}) \cong\mathrm{H}^3(X,\Z)$. This shows that the plus construction has added new objects, and it shows again that $\cssAlgBdlbi_{\C}$ is not a 2-stack. 
\end{remark}

The next consequences of \cref{th:classline2bundles} concern the relation between line 2-bundles, bundle gerbes, and algebra bundles, which are all objects  in the bicategory of super 2-line bundles. First of all, we have the following obvious statement.

\begin{corollary}
\label{co:line2bundlesandalgebrabundles}
Let $\mathscr{L}$ be a super 2-line bundle. Then, $\mathscr{L}$ is an ordinary super algebra bundle (i.e., $\mathscr{L}\cong \mathcal{A}$  for a central simple super algebra bundle $\mathcal{A}$) if and only if the class of $\mathscr{L}$ in $\check {\mathrm{H}}^2(X,\underline{k}^{\times})$ is torsion.
\end{corollary}

Since super bundle gerbes are in particular super line 2-bundles, we have as a special case of \cref{co:line2bundlesandalgebrabundles} the following result.

\begin{corollary}
\label{co:bundlegerbesandalgebrabundles}
Let $\mathscr{G}$ be a super bundle gerbe. Then, $\mathscr{G}$ is isomorphic to a super algebra bundle (as super 2-vector bundles) if and only if the Dixmier-Douady class of $\mathscr{G}$ is torsion. 
\end{corollary}

Finally, using the group structure in \cref{th:classline2bundles}, the following becomes true. 

\begin{corollary}
Let $\mathscr{L}$ be a super 2-line bundle. Then, $\mathscr{L}\cong \mathcal{A} \otimes \mathscr{G}$ for a central simple super algebra bundle $\mathcal{A}$ and a super line bundle gerbe $\mathscr{G}$.
\end{corollary}

\setsecnumdepth{1}

\section{Algebra bundles and lifting gerbes}

\label{sec:liftinggerbes}

In this section, we give several examples for 2-vector bundles and corresponding morphisms which fit in the following abstract setup. Let $G$ be a $\Z_2$-graded Lie group.
This means that $G$ comes with a homomorphism to $\Z_2$. Equivalently, $G$ is a disjoint union $G = G^0 \cup G^1$ with $G^0$ a normal subgroup.
An action of $G$ on a super vector space $V$ is \emph{even} if $G^0$ and $G^{1}$ act by even, respectively odd automorphisms.

Let $Z$ be an abelian Lie group and let $Z \rightarrow \hat{G} \stackrel{\rho}{\rightarrow} G$ be a central extension (observe here that $\hat{G}$ picks up a $\Z_2$ grading by pre-composition with $\rho$). 
Let $\pi: P\rightarrow X$ be a principal $G$-bundle over a manifold $X$.
A natural question is then whether the structure group of $P$ can be lifted to $\hat{G}$; in other words, we ask for a $\hat{G}$-principal bundle $\hat{P} \rightarrow X$ together with a fibre-preserving map ${\rho}_P: \hat{P} \rightarrow P$ that intertwines the group actions (along $\rho : \hat{G} \rightarrow G$). 
To answer this question, one considers the associated \emph{lifting gerbe} $\mathscr{G}_P$, which is the principal $Z$-bundle gerbe schematically depicted as follows.
\begin{equation}
 \label{DepictionSpinLifting}
 \mathscr{G}_P= \left ( 
\begin{aligned}
\xymatrix{ & \delta^{*}\hat G \ar[r] \ar[d] & \hat G \ar[d] \\ P \ar[d]^{\pi} & P^{[2]} \ar[r]^-{\delta} \ar@<0.5ex>[l]^{\pr_2}\ar@<-0.5ex>[l]_{\pr_1} & G \\ X} 
\end{aligned}
\right )
\end{equation}
Here, $\delta: P^{[2]} \rightarrow G$ is the map determined by 
\begin{equation*}
p_2 \cdot \delta(p_1, p_2) = p_1.
%p_2 \cdot \delta(p_2, p_1) = p_1,
\end{equation*}
Explicitly, the fibres $(\delta^* \hat{G})_{p_1, p_2}$ consist of all $\hat{g} \in \hat{G}$ with $\rho(\hat{g}) = \delta(p_1, p_2)$.
The bundle gerbe product 
\begin{equation*}
  \mu : (\delta^* \hat{G})_{p_2, p_3} \times_Z (\delta^* \hat{G})_{p_1, p_2} \to (\delta^* \hat{G})_{p_1, p_3}
\end{equation*}
 over $P^{[3]}$, which is not depicted in \eqref{DepictionSpinLifting}, is just given by group multiplication in $\hat{G}$, observing that if $\hat{g}_{12} \in (\delta^*\hat{G})_{p_1, p_2}$ and $\hat{g}_{23} \in (\delta^*\hat{G})_{p_2, p_3}$, then $\hat{g}_{23} \cdot \hat{g}_{12}  \in (\delta^*\hat{G})_{p_1, p_3}$.
%
\begin{comment}
$\hat{g}_{12} \in \hat{G}$  is an elements of $(\delta^*\hat{G})_{p_1, p_2}$ are such that $\rho(\hat{g}_{12}) = \delta(p_1, p_2)$. To see that $ \hat{g}_{23} \cdot \hat{g}_{12} \in (\delta^*\hat{G})_{p_1, p_3}$, we have to check that $\rho(\hat{g}_{23} \cdot \hat{g}_{12}) = \delta(p_1, p_3)$. Since 
\begin{equation*}
\rho(\hat{g}_{23} \cdot \hat{g}_{2123}) = \rho(\hat{g}_{23})\rho(\hat{g}_{12}) = \delta(p_2, p_3)\delta(p_1, p_2), 
\end{equation*}
this is just the equation
\begin{equation*}
  \delta(p_2, p_3) \cdot \delta(p_1, p_2) = \delta(p_1, p_3), 
\end{equation*}
which follows from observing that both sides map $p_1$ to $p_3$. 
Notice that this forces to define $\delta$ above and not the other way around.
\end{comment}
%
\,\\
The lifting theory of Murray, see \cite[Section 4]{Murray1996}, tells us that a lift of the structure group exists if and only if $\mathscr{G}_P$ is trivializable, and the category of lifts is equivalent to the category of trivializations of $\mathscr{G}_P$, see \cite[Thm.~2.1]{waldorf13}.
\\
In the following we suppose that $Z \subset k^\times$ for  $k=\R$ or $\C$. Then, the lifting gerbe $\mathscr{G}_P$ can be identified with a super line bundle gerbe, whose super line bundle over $P^{[2]}$ is the associated line bundle $\mathcal{L}_{\hat G} := \delta^{*}\hat G \times_Z k$, which becomes a \emph{super} line bundle using the map $P^{[2]} \stackrel{\delta}{\to} G \to \Z_2$. 
We remark that -- under the correspondence between super line bundle gerbes and twisting of K-theory, see \cref{prop:twistings} --  $\mathscr{G}_P$ corresponds to the twisting described in \cite[Ex. 2.27]{Freed2011a} associated to the graded central extension $Z \to \hat G \to G$ and the principal bundle $P$. 

The super line bundle gerbe $\mathscr{G}_P$,
in turn,  defines a super  2-line bundle over $X$ (also denoted by $\mathscr{G}_{P}$), via the functor from \cref{sec:bundlegerbes}. It is then natural to ask if it is isomorphic to a super algebra bundle, i.e., if the lifting obstruction can also be described by a super algebra bundle over $X$. By \cref{co:line2bundlesandalgebrabundles} we know that this is the case if and only if the Dixmier-Douady class of $\mathscr{G}_P$ is torsion. A recent paper of Roberts \cite{Roberts2021} indeed shows that most lifting gerbes are torsion, for instance, whenever $X$ is connected and  $\pi_1(X)$ is finite.

In this section, we consider the problem of constructing a super algebra bundle $\mathcal{A}$ that is isomorphic to the lifting gerbe $\mathscr{G}_P$ in the bicategory of super 2-line bundles, together with an isomorphism $\mathscr{G}_P\cong \mathcal{A}$.    
An important ingredient to our solution of this problem is the following notion, which already turned out to be useful in an infinite-dimensional setting, see \cite[Def.~2.2.12, \S 2.5]{kristel2020smooth}.

\begin{definition}[Equivariant module]
\label{DefinitionRepCentralExtension}
Suppose $A$ is a super algebra on which a graded Lie group $G$ acts by super algebra automorphisms. Suppose $Z \to \hat G \to G$ is a central extension with $Z \subset k^{\times}$.
An  $A$-module $F$ is called \emph{$\hat G$-equivariant} if it is equipped with an even linear action of $\hat G$ such that $Z \subset \hat G$ acts by scalar multiplication, and such that the condition
\begin{equation}
\label{eq:IntertwinerCondition}
{g} \cdot (a \lact v) = g(a) \lact ({g} \cdot v)
\end{equation}
is satisfied for all $v \in F$, $a \in A$ and $g \in G$.
\end{definition}

\begin{example}\label{ex:InnerAuto}
        Let $A$ be a central super algebra (for example $M_n(k)$).
        Let $\Inn(A)$ be the group of inner automorphisms of $A$ and assume that $\Inn(A)$ is equipped with a grading that induces the usual grading on the group $A^{\times}$ of invertible elements in $A$, which is a central extension
\begin{equation*}
 k^{\times} \to A^{\times} \to \mathrm{Inn}(A) \text{.}
\end{equation*}
        Then, any $A$-module $F$ is automatically  $A^{\times}$-equivariant  in the sense of \cref{DefinitionRepCentralExtension}.
\end{example} 

More examples for \cref{DefinitionRepCentralExtension} will be given in the applications below. We shall first explain how \cref{DefinitionRepCentralExtension} is used to construct a super algebra bundle $\mathcal{A}$ and an isomorphism $\mathscr{G}_P\cong \mathcal{A}$.

Given a $\hat G$-equivariant $A$-module $F$ as in \cref{DefinitionRepCentralExtension}, we first construct a  $\mathscr{G}_P$-twisted vector bundle denoted  $\mathscr{F}=(\mathcal{F},\phi)$, in the sense explained in \cref{sec:bundlegerbes}.
Its super vector bundle over $P$ is the trivial bundle $\mathcal{F}:=\underline{F}$. 
The bundle morphism over $P^{[2]}$,
\begin{equation*}
\phi: \pr_2^{*}  \mathcal{F} \otimes \mathcal{L}_{\hat G} \to \pr_1^{*}\mathcal{F}\text{,} 
\end{equation*}
is defined over a point $(p_1,p_2)\in P^{[2]}$ by
\begin{equation*}
\phi_{p_1,p_2}: \mathcal{F}_{p_2} \otimes (\mathcal{L}_{\hat G})_{p_1, p_2} \to \mathcal{F}_{p_1}, \qquad (p_2, v) \otimes [\hat{g}, \lambda] \mapsto (p_1, \lambda \hat{g} v) \text{.}
\end{equation*}
It is straightforward to show that this is a well-defined, smooth bundle morphism and fits into the required commutative diagram \cref{eq:gerbemodulecond}.
Note, in particular, that well-definedness of $\phi$ requires the condition that $Z \subset k^{\times}$ acts by scalar multiplication. 
%
\begin{comment}
For the well-definedness, let $\hat g'$ be another choice; then, $\hat g'=\hat g z$ for $z\in Z\subset k$. Since $Z$ acts by scalar multiplication on $F$, we have
\begin{equation*}
(p_1, \hat g'^{-1} v) \otimes \hat g' = (p_1, z^{-1}\hat g^{-1} v) \otimes \hat g z=  (p_1, \hat g^{-1} v) \otimes \hat g \text{.}
\end{equation*}
The identity is over a point $(p_3,p_2,p_1)$:
\begin{equation*}
\phi_{p_3,p_1} = (\id \otimes \mu_{p_3,p_2,p_1})\circ  (\phi_{p_2,p_1} \otimes \id) \circ \phi_{p_3,p_2}\text{.}
\end{equation*}
We may choose $\hat g_{12}$ over $\delta(p_2,p_1)$ and $\hat g_{23}$ over $\delta_{p_3,p_2}$. Note that $\hat g_{13}:=\hat g_{23}\hat g_{12}$ is then over $\delta(p_3,p_2)\delta(p_2,p_1)=\delta(p_3,p_1)$.
To check the identity, we have
\begin{align*}
\phi_{p_3,p_1}(p_3,v) &= (p_1, \hat g_{13}^{-1} v) \otimes \hat g_{13}
\\&= (p_1, \hat g_{12}^{-1}  \hat g_{23}^{-1} v) \otimes  \mu_{p_3,p_2,p_1}(\hat g_{12} \otimes \hat g_{23})
\\&= (\id \otimes \mu_{p_3,p_2,p_1})  ((p_1, \hat g_{12}^{-1}  \hat g_{23}^{-1} v) \otimes \hat g_{12} \otimes \hat g_{23})
\\&= (\id \otimes \mu_{p_3,p_2,p_1})  (\phi_{p_2,p_1} (p_2, \hat g_{23}^{-1} v) \otimes \hat g_{23})
\\&= ((\id \otimes \mu_{p_3,p_2,p_1})\circ  (\phi_{p_2,p_1} \otimes \id) \circ \phi_{p_3,p_2})(p_3,v)\text{.}
\end{align*}
\end{comment}
%
\, \\
Next, we define a {super}algebra bundle over $X$,  namely, the associated bundle
\begin{equation*}
\mathcal{A} := P \times_G A\text{.}
\end{equation*}
We upgrade $\mathscr{F}=(\mathcal{F},\phi)$ to a $\mathscr{G}_P$-twisted $\mathcal{A}$-module bundle (\cref{def:twistedmodulebundle}), using  condition \cref{eq:IntertwinerCondition} on our representations. For this purpose, we define on $\mathcal{F}$ the left $\pi^{*}\!\mathcal{A}$-module bundle structure defined fibrewise at $p\in P$ over $x\in X$ by $\mathcal{A}_x \otimes \mathcal{F}_p \to \mathcal{F}_p : ([p,a],(p,v))\mapsto (p,a\lact v)$. 
%
\begin{comment}
Since $A_P$ is an associated bundle, its elements are equivalence classes $[p, a]$, $p \in P$, $a \in A$, subject to the relation $[p \cdot g, a] = [p, g(a)]$. 
\end{comment}
%
Again, it is straightforward to show that this gives a well-defined smooth bundle morphism, and that $\phi$ is $\mathcal{A}$-linear; this completes the construction of a $\mathscr{G}_P$-twisted $\mathcal{A}$-module bundle $\mathscr{F}$. By \cref{lem:twistedmodulebundles} the category of $\mathscr{G}_P$-twisted $\mathcal{A}$-module bundles and the category of 1-morphisms $\mathscr{G}_P \to \mathcal{A}$ are canonically isomorphic; this allows us to see $\mathscr{F}$ as a 1-morphism $\mathscr{F}: \mathscr{G}_P \to \mathcal{A}$. Let us summarize this and state some properties.

\begin{theorem}
\label{Thm:CanonicalIsomorphism:Representation}
Let $Z \to \hat G \to G$ be a central extension of a graded Lie group $G$, with $Z \subset k^{\times}$. Let $P$ be a principal $G$-bundle over $X$, let $A$ be a superalgebra on which $G$ acts by super algebra automorphisms, and let $F$ be a $\hat G$-equivariant $A$-module. Let $\mathscr{G}_P$ be the lifting gerbe and $\mathcal{A}:= P \times_G A$ be the associated super algebra bundle.      Then, the 1-morphism
\begin{equation*}
\mathscr{F}: \mathscr{G}_P \to \mathcal{A}
\end{equation*}
of super 2-vector bundles over $X$ is an isomorphism, $\mathscr{G}_P\cong \mathcal{A}$, if and only if  $F$ is a Morita equivalence between $A$ and $k$. In this case, $\mathscr{F}$ induces  a canonical equivalence between the category of lifts of $P$ and the category of invertible left $\mathcal{A}$-module bundles.
\end{theorem}

\begin{proof}
 \cref{lem:twistedmodulebundles}  shows that $\mathscr{F}$ is invertible if and only if the fibres of $\mathcal{F}$ are Morita equivalences; here, these fibres are all $F$. Finally, composition with the isomorphism $\mathscr{F}: \mathscr{G}_P \to \mathcal{A}$ yields an equivalence of categories of 1-isomorphisms
\begin{equation*}
\Isocat_{\stwoVectBdl_k(X)}(\mathcal{A},\mathscr{I}) \cong \Isocat_{\stwoVectBdl_k(X)}(\mathscr{G}_P,\mathscr{I})\text{.}
\end{equation*}
By the fully faithful inclusion of super algebra bundles (\cref{sec:inclusionofalgebrabundles}), the left hand side is equivalent to the category of invertible left $\mathcal{A}$-module bundles. By the fully faithful inclusion of super bundle gerbes, the right hand side is equivalent to the category of trivializations of $\mathscr{G}_P$, which in turn is equivalent to the category of lifts of $P$. 
\end{proof}

We now discuss several applications of this theorem. The first examples come from classical spin geometry.
First, let $G = \mathrm{O}_d$ act on the real Clifford algebra $A = \Cl_d \otimes \Cl_{-d} \cong \Cl_d \otimes \Cl_d^{\opp}$ via its standard action \emph{on the first factor only}. We may consider $F=\Cl_d$ as an $A$-module in the obvious way (\ie $\Cl_{d}$ and $\Cl_{-d}$ act by left- and right multiplication respectively), and note that it is a Morita equivalence to the ground field $k=\R$.
The group $\hat{G} = \mathrm{Pin}^-_d \subset \Cl_d$ is a central extension
\begin{equation*}
\Z_2 \to \mathrm{Pin}^-_d \to \O_d
\end{equation*}
 and $\mathrm{Pin}^-_d$ acts on $F$ by left multiplication.
It is straightforward to check that this turns $F$ into a $\Pin_d^{-}$-equivariant $A$-module. We obtain from \cref{Thm:CanonicalIsomorphism:Representation} the following result.
\begin{corollary}
\label{co:pin-}
  Let $X$ be a Riemannian manifold and let $\mathscr{G}^{\mathrm{Pin}_d^-}_{\mathrm{O}_d}(X)$ be the lifting gerbe for lifting the structure group of the frame bundle from $\mathrm{O}_d$ to $\mathrm{Pin}_d^-$. Then, we have a canonical 1-isomorphism of 2-vector bundles
  \begin{equation*}
        \mathscr{G}^{\mathrm{Pin}_d^-}_{\mathrm{O}_d}(X) \cong \Cl(TX) \otimes \Cl_{-d}.
\end{equation*}
\end{corollary}

The second part of \cref{Thm:CanonicalIsomorphism:Representation} gives in the present situation a new proof of the well-known fact that the category of  $\mathrm{Pin}_d^-$-structures on $X$ is equivalent to the category of invertible $(\Cl(TX) \otimes \Cl_{-d})$-module bundles.  The advantage of our new proof is that this  abstract equivalence is exhibited as an equivalence between Hom-categories in a single bicategory, established by composition with a 1-isomorphism. 
Passing through our constructions, it turns out that this sends a $\mathrm{Pin}_d^-$-structure on $X$ to the real spinor bundle of Lawson and Michelsohn \cite{Lawson1989}.
In case that $X$ is oriented, in which case the structure group of the frame bundle is already reduced to $\mathrm{SO}_d$, there is a variation of this statement where $G = \mathrm{SO}_d$ and $\hat{G} = \mathrm{Spin}_d$. 
In this case, we have an analogous statement for the obstruction gerbe for the lift from $\mathrm{SO}_d$ to $\mathrm{Spin}_d$, which reads
\begin{equation*}
  \mathscr{G}^{\mathrm{Spin}_d}_{\mathrm{SO}_d}(X) \cong \Cl(TX) \otimes \Cl_{-d}.
\end{equation*}
\begin{comment}
Moreover, the category of  $\Spin_d$-structures on $X$ is equivalent to the category of invertible $(\Cl(TX) \otimes \Cl_{-d})$-module bundles, and this equivalence takes a $\Spin_d$-structure to the real spinor bundle of $X$.
\end{comment}

There is also a variation for complex scalars. 
Here $G = \mathrm{O}_d$ and $\hat{G} = \mathrm{Pin}^\C_d$. If $d$ is even, we set $A = \CCl_d$ and $F$ is one of the spin representations $\Delta_d^\pm$, which is a Morita equivalence to $k= \C$. Again, the action of $G$ on $A$ is the standard one and $\mathrm{Pin}^\C_d$ acts on $F$ through the inclusion $\mathrm{Pin}_d^\C \subset \CCl_d$. If $d$ is odd, then we take $A = \CCl_{d+1} = \CCl_d \otimes \CCl_1$ (on which $G$ acts on the first factor only), together with $F = \Delta_{d+1}^\pm$. We obtain the following.

\begin{corollary}
\label{co:pinc}
Let $X$ be a $d$-dimensional Riemannian manifold and let $\mathscr{G}^{\mathrm{Pin}^\C_d}_{\mathrm{O}_d}(X)$ be the lifting gerbe for lifting the structure group of the frame bundle from $\mathrm{O}_d$ to $\mathrm{Pin}_d^\C$. Then, each of the two possible choices  of spin representation provides a canonical isomorphism
\begin{equation*}
  \mathscr{G}^{\mathrm{Pin}^\C_d}_{\mathrm{O}_d}(X) \cong \begin{cases} \CCl(TX)  & \text{if } $d$ \text{ is even;} \\ \CCl(TX) \otimes \CCl_1 & \text{if } $d$ \text{ is odd.} \end{cases}
\end{equation*} 
\end{corollary}

Again, the second part of \cref{Thm:CanonicalIsomorphism:Representation} gives in the present situation (say, $d$ even) a new proof of the well-known fact that  the category of  $\mathrm{Pin}_d^\C$-structures on $X$ is equivalent to the category of invertible $\CCl(TX)$-module bundles. For $F=\Delta_d^{+}$, this equivalence takes a $\mathrm{Pin}_d^\C$-structure on $X$ to the graded spinor bundle, and for $F=\Delta_d^{-}$, to the its grading reversal. If $d$ is odd, then the category of  $\mathrm{Pin}_d^\C$-structures on $X$ is equivalent to the category of invertible $(\CCl(TX) \otimes \CCl_1)$-module bundles. 
For $F=\Delta_{d+1}^{+}$, this equivalence takes a $\mathrm{Pin}_d^\C$-structure $P$ on $X$ to the graded $(\CCl(TX) \otimes \CCl_1)$-module bundle $\mathcal{F} = P \times_{\mathrm{Pin}_d^\C} \Delta_{d+1}^{+}$.
To obtain the usual ungraded spinor bundle in odd dimensions this way, one takes the even subbundle $\mathcal{F}_0$ of $\mathcal{F}$, which has an action of $\CCl(TX)$ given by modifying the previous action by the extra vector $e \in \CCl_1$ (Explicitly, it is given by $v \lact_\mathrm{mod} \psi := e \lact v \lact \psi$ for $v \in TX$ and $\psi \in \mathcal{F}_0$.)
\\
As before, there is a variation on this statement in the case that $X$ is oriented and $G = \mathrm{SO}_d$, $\hat{G} = \mathrm{Spin}_d^\C$.

We finish this section with some infinite-dimensional examples. 
These do not, strictly speaking, fit our setup.
Instead, one would need to define super 2-vector bundles modeled on suitable bicategories of C$^*$-algebras or von Neumann algebras instead of the category of finite-dimensional algebras used here (see \cite{Pennig2011} or \cite{Brouwer2003}).
However, due to issues with smoothness in this operator algebraic setup, this will often only yield \emph{continuous} 2-vector bundles.
\\
Let $H$ be a separable complex Hilbert space and let $G = \mathrm{PU}(H) = \mathrm{U}(H) /\mathrm{U}(1)$ be the corresponding projective unitary group. $\mathrm{PU}(H)$ acts by conjugation on the algebra $A = \mathbb{K}(H)$ of compact operators on $H$. Now, $F=H$ is a module for $\mathbb{K}(H)$, but $\mathrm{PU}(H)$ does not act on $H$. Instead, the central extension
\begin{equation*}
\U(1) \to \U(H) \to \PU(H)
\end{equation*}
together with the standard action of $\U(H)$ on $H$ turns $H$ into (a topological version of) a $\U(H)$-equivariant $\mathbb{K}(H)$-module;  in particular, condition  \eqref{eq:IntertwinerCondition} is satisfied.
Suppose now that $P$ is a principal $\PU(H)$-bundle over $X$. We consider the lifting gerbe $\mathscr{G}_P$ and the associated bundle $\mathcal{A}:=P \times_{\mathrm{PU}(H)} \mathbb{K}(H)$ of compact operators. It is well known that the (bundle-gerbe-theoretic) Dixmier-Douady class of $\mathscr{G}_P$ coincides with the (operator algebraic) Dixmier-Douady class of $\mathcal{A}$. In a setting of continuous 2-vector bundles, this coincidence obtains a nice new explanation, via (a continuous version of) \cref{Thm:CanonicalIsomorphism:Representation}.  
\\
Indeed, \cref{Thm:CanonicalIsomorphism:Representation} provides a $\mathscr{G}_P$-twisted $\mathcal{A}$-module bundle $\mathscr{H}$, which gives a 1-morphism
\begin{equation*}
\mathscr{H}:\mathscr{G}_P \to \mathcal{A}\text{.}
\end{equation*} 
In fact, if the bicategory of infinite-dimensional algebras used to define notion of topological 2-vector bundles is such that $H$ provides a Morita equivalence from $\mathbb{K}(H)$ to $\C$, then $\mathscr{H}$ is in fact an isomorphism. 
For example, this is the case in the category of C$^{\ast}$-algebras and Hilbert modules discussed in the work of Pennig \cite{Pennig2011}.
Then, the 1-morphism $\mathscr{H}: \mathscr{G}_P \to \mathcal{A}$ is a 1-isomorphism, and it is clear that the 2-vector bundles $\mathscr{G}_P$ and $\mathcal{A}$ must have the same Dixmier-Douady class.

Another example is the treatment of spinor bundles on the loop space $LM$ of a string manifold $M$, which is carried out rigorously in a Fr\'echet setting in \cite{kristel2020smooth}.
There, $P$ is the frame bundle of $LM$, which is a Fr\'echet principal bundle for the loop group $G=L\Spin(d)$, where $d=\dim (M)$. Let $V$ be a  Hilbert space of \quot{spinors on the circle}, which is equipped with a real structure, and with a Lagrangian subspace $L \subset V$. The unitary group of $V$ has a famous subgroup, the restricted orthogonal group $\O_L(V)$ of unitary operators that commute with the real structure and fix the equivalence class of $L$.
The restricted orthogonal group is a Banach Lie group, and there are Lie group homomorphisms
\begin{equation*}
L\Spin(d) \to L\SO(d) \to \O_L(V)\text{.}
\end{equation*}
We consider the Clifford C$^{*}$-algebra $A=\Cl(V)$ of Plymen-Robinson \cite{PR95}, on which $\O_L(V)$ acts by Bogoliubov automorphisms. Thus, the group $G=L\Spin(d)$ acts on $A$.
Next, we consider the Fock space $F=F_L$ associated to the Lagrangian subspace $L$, which is a $\Cl(V)$-module. The universal central extension
\begin{equation*}
\U(1) \to \widetilde{L\Spin}(d) \to L\Spin(d)
\end{equation*}
acts $F$ by unitary operators, turning $F$ into a $\widetilde{L\Spin}(d)$-equivariant $\Cl(V)$-module. It  satisfies condition \cref{eq:IntertwinerCondition}, which is in this case traditionally written as
\begin{equation*}
\theta_g(a) = UaU^{*}\text{,}
\end{equation*}
where $g\in L\Spin(d)$, $\theta_g$ denotes its Bogoliubov automorphism,  $a\in \Cl(V)$, and $U\in \widetilde{L\Spin}(d)$ projecting to $g$. 
The precise infinite-dimensional analogue of the general setting is explained in \cite[\S 2.4]{kristel2020smooth}, and the fact that the constructions outlined above match this setting is proved in \cite[Thm.~3.2.9]{kristel2020smooth}. 
\\
 The bundle gerbe $\mathscr{S}:=\mathscr{G}_P$ in this situation is the \emph{spin lifting gerbe} on loop space. The algebra bundle $\mathrm{Cl}(P):=\mathcal{A} = P \times_{L\Spin(d)} \Cl(V)$ is the Clifford algebra bundle on loop space. The $\mathscr{S}$-twisted $\mathcal{A}$-module bundle $\mathscr{F}$ is the \emph{twisted spinor bundle on loop space}. While all these structures have a well-defined meaning in the setting  rigged C$^{*}$-algebras and rigged Hilbert space bundles, we currently do not have a complete setting of infinite-dimensional smooth 2-vector bundles. 
In such a suitable setting, we would be able to interpret the twisted spinor bundle $\mathscr{F}$ as a 1-morphism
\begin{equation*}
\mathscr{F}: \mathscr{S} \to \Cl(P) 
\end{equation*}
of 2-vector bundles over $LM$, in analogy to the finite-dimensional cases discussed in \cref{co:pin-,co:pinc}.

\begin{appendix}

\setsecnumdepth{1}

\section{Cohomology with values in crossed modules}

\label{sec:crossedmodules}

\label{sec:butterflies}

A \emph{crossed module $\Gamma$ of Lie groups}  of a Lie group homomorphism $t:H \to G$ between two Lie groups and of a left action $h\mapsto \prescript{g}{}h$ of $G$ on $H$ by group homomorphisms, such that $t(\prescript{g}{}h)=g t(h)g^{-1}$ and $\prescript{t(x)}{}h=xhx^{-1}$ hold for all $g\in G$ and $h,x\in H$ \cite{Mackenzie1989}. 

\begin{example}
An \emph{abelian} Lie group $A$ induces a crossed module, denoted $\mathscr{B}A$, with groups  $H:=A$ and $G:=\{e\}$. Any Lie group $G$ induces another crossed module, denoted $G_{dis}$, with groups $G$ and $H:= G$, $t=\id_G$ and the conjugation action of $G$ on itself. 
\end{example}

\label{sec:cmc}

We need in \cref{sec:nonabcohfa} and below the passage from a crossed module $\Gamma$ to the corresponding Lie 2-groupoid, denoted $\mathscr{B}\Gamma$. 
The rationale here is that smooth crossed modules correspond to strict Lie 2-groups, which in turn correspond to Lie 2-groupoid with a single object.
  Suppose  that a crossed module $\Gamma$ is given by Lie groups $G$ and $H$ and a Lie group homomorphism $t:H \to G$. The associated Lie 2-groupoid $\mathscr{B}\Gamma$ has a single object, its manifold of 1-morphisms is $G$, and its manifold of 2-morphisms is $H \times G$, where $(h,g)$ is considered as a 2-morphism from $g$ to $t(h)g$. The composition of 1-morphisms of the multiplication of $G$,  the vertical composition of 2-morphisms is $(h',g') \circ (h,g)= (h'h,g)$, and the horizontal composition of 2-morphisms is given by the semi-direct product formula
\begin{equation}
\label{eq:semidirect}
(h_2,g_2) \bullet (h_1,g_1) = (h_2 \prescript{g_2}{}h_1,g_2g_1)\text{.} 
\end{equation}

\begin{definition}[Cohomology with values in a crossed module]
\label{DefNonAbChechCohom}
Let $\Gamma$ be a crossed module of Lie groups, and let $X$ be a smooth manifold.
The \emph{\v Cech cohomology of $X$ with coefficients in a crossed module $\Gamma$} is
\begin{equation*}
\check{\mathrm{H}}^1(X,\Gamma):=\mathrm{h}_0(\underline{\mathscr{B}\Gamma}^{+}(X))\text{.}
\end{equation*}
\end{definition}

That is, we consider the 2-groupoid $\mathscr{B}\Gamma$ 
with a single object, the presheaf of bicategories  $\underline{\mathscr{B}\Gamma}$ of level-wise smooth functions to $\mathscr{B}\Gamma$,  2-stackify using the plus-construction, evaluate at $X$, and then take its set of isomorphism classes of objects.
See \cite[\S A.3]{Nikolause} for this elegant definition. 

In \cite[\S A.3]{Nikolause} it is explained how to spell out  \cref{DefNonAbChechCohom} in terms of concrete cocycles. For this purpose, one performs the plus construction with respect to a surjective submersion obtained as the disjoint union of the members of an open cover $\{A_{\alpha}\}_{\alpha\in A}$ of $X$. Then, a class in $\check{\mathrm{H}}^1(X,\Gamma)$ is represented by a pair $(g,a)$ where $g$ is a collection of smooth maps $g_{\alpha\beta}:U_{\alpha} \cap U_{\beta} \to G$ and $a$ is a collection of smooth maps $a_{\alpha\beta\gamma}:U_{\alpha} \cap U_{\beta}\cap U_{\gamma} \to H$, such that the cocycle conditions 
\begin{equation} 
\label{CocycleConditionsNonAb2}
t(a_{\alpha\beta\gamma})\cdot  g_{\beta\gamma}  \cdot g_{\alpha\beta}=g_{\alpha\gamma}
\quad\text{ and }\quad
a_{\alpha\gamma\delta}\cdot \prescript{g_{\gamma\delta}}{}a_{\alpha\beta\gamma}=a_{\alpha\beta\delta}\cdot a_{\beta\gamma\delta}
\end{equation}  
are satisfied. Two cocycles $(g,a)$  and $(g',a')$ are equivalent, if -- after passing to a common refinement of the open covers -- there exist smooth maps $h_{\alpha}:U_{\alpha} \to G$ and $e_{\alpha\beta}:U_{\alpha}\cap U_{\beta}\to H$ satisfying
\begin{equation}
\label{CoboundaryConditionsNonAb}
t(e_{\alpha\beta})\cdot h_{\beta}\cdot g_{\alpha\beta}=g'_{\alpha\beta}\cdot h_{\alpha}
\quad\text{ and }\quad
a'_{\alpha\beta\gamma}\cdot \prescript{g'_{\beta\gamma}}{}{e_{\alpha\beta}}\cdot e_{\beta\gamma}= e_{\alpha\gamma}\cdot \prescript{h_{\gamma}}{}a_{\alpha\beta\gamma}\text{.}
\end{equation}

\begin{remark}\label{RemarkCheckCohomologyAppendix}
It is straightforward to see, either from the definition above or from the cocycle description, that for a Lie group $G$ we obtain $\check{\mathrm{H}}^1(X,G_{dis})= \check{\mathrm{H}}^1(X,\underline{G})$, \ie the ordinary \v Cech cohomology with values in the sheaf of smooth $G$-valued functions, and for an abelian Lie group $A$ we obtain $\check{\mathrm{H}}^1(X,\mathscr{B}A)=\check{\mathrm{H}}^2(X,\underline{A})$. \end{remark}

\begin{remark}
The \v Cech cohomology $\check{\mathrm{H}}^1(X,\Gamma)$ of $X$ with values in $\Gamma$ is often called the \quot{non-abelian (\v Cech)} cohomology. Indeed, in contrast to ordinary cohomology,  $\check{\mathrm{H}}^1(X,\Gamma)$ is not a group, but only a pointed set. 
\end{remark}

The easiest and most natural kind of morphism one can consider between crossed modules, called \emph{strict homomorphism}, is a pair of group homomorphisms $G \to G'$ and $H \to H'$ that strictly respect all structure of the crossed modules. 
A strict homomorphism of crossed modules induces in an obvious way a map between the corresponding cohomologies, by just composing cocycles with the group homomorphisms. 
However, this does not give the correct notion of isomorphism between the associated 2-stacks $\sheaf{\mathscr{B}\Gamma}^{+}$.
 It is proved by Aldrovandi-Noohi  \cite{Aldrovandi2009} that \emph{invertible butterflies} between crossed modules $\Gamma_1$ and $\Gamma_2$ give the correct notion. An invertible butterfly between crossed modules $t_1: H_1 \to G_1$ and $t_2: H_2 \to G_2$ consists of a Lie group $K$ together with Lie group homomorphisms  that make up a commutative diagram 
\begin{equation}
\label{eq:butterfly}
\begin{aligned}
\xymatrix{H_1 \ar[dd]_{t_1} \ar[dr]^{i_1} && H_2 \ar[dd]^{t_2}\ar[dl]_{i_2} \\ & K \ar[dl]^{p_1}\ar[dr]_{p_2} \\ G_1 && G_2\text{,}}
\end{aligned}
\end{equation}
such that both diagonal sequences are short exact sequences of Lie groups, and the equations
\begin{equation} \label{RelationsButterfly}
i_1(\prescript{p_1(x)}{}h_1) = xi_1(h_1)x^{-1}
\quad\text{ and }\quad
i_2(\prescript{p_2(x)}{}h_2) = xi_2(h_2)x^{-1}
\end{equation} 
hold for all $h_1\in H_1$, $h_2\in H_2$ and $x\in K$. Since an invertible butterfly establishes an equivalence of 2-stacks $\sheaf{\mathscr{B}\Gamma_1}^{+}\cong \sheaf{\mathscr{B}\Gamma_2}^{+}$, see \cite{Aldrovandi2009}, we obtain due to \cref{DefNonAbChechCohom} immediately the following result.

\begin{proposition}
\label{prop:butterflyco}
Any invertible butterfly between crossed modules $\Gamma_1$ and $\Gamma_2$ induces a bijection
\begin{equation*}
\check{\mathrm{H}}^1(X,\Gamma_1) \cong \check{\mathrm{H}}^1(X,\Gamma_2)\text{.}
\end{equation*} 
\end{proposition}

More difficult is to see how the isomorphism of \cref{prop:butterflyco} is described explicitly on a cocycle level; since we have  not found a reference about this in the literature, we shall describe this now. 
Let $(g,a)$ be a cocycle representing a class in $\check {\mathrm{H}}^1(X,\Gamma_1)$ with respect to an open cover $\{U_{\alpha}\}_{\alpha\in A}$. 
\begin{comment}
Thus, $g$ is a collection of smooth maps $g_{\alpha\beta}:U_{\alpha} \cap U_{\beta} \to G_1$ and $a$ is a collection of smooth maps $a_{\alpha\beta\gamma}:U_{\alpha} \cap U_{\beta}\cap U_{\gamma} \to H_1$, such that the cocycle conditions \cref{CocycleConditionsNonAb2} are satisfied. 
\end{comment}
By passing to a smaller open cover, we may assume that $g_{\alpha\beta}$ lift along $p_1:K \to G_1$ to smooth maps $\tilde g_{\alpha\beta}: U_{\alpha}\cap U_{\beta} \to K$. Then, we consider $f_{\alpha\beta} := p_2\circ \tilde g_{\alpha\beta}: U_{\alpha}\cap U_{\beta} \to G_2$. The first cocycle condition for $(g,a)$ shows that
\begin{equation*}
p_1(\tilde g_{\alpha\gamma}\tilde g_{\alpha\beta}^{-1}\tilde g_{\beta\gamma}^{-1}i_1(a   _{\alpha\beta\gamma})^{-1})=1\text{,}
\end{equation*}
and hence, by exactness of the butterfly's NE-SW sequence, there exist unique smooth maps $b_{\alpha\beta\gamma}: U_{\alpha}\cap U_{\beta}\cap U_{\gamma} \to H_2$ such that
\begin{equation*}
i_2(b_{\alpha\beta\gamma}) = \tilde g_{\alpha\gamma}\tilde g_{\alpha\beta}^{-1}\tilde g_{\beta\gamma}^{-1}i_1(a _{\alpha\beta\gamma})^{-1}\text{.}
\end{equation*}
It is straightforward to show that  $(f,b)$ is a cocycle with values in $\Gamma_2$.
Suppose another lift $\tilde g'_{\alpha\beta}$ is chosen, and let $(f',b')$ be the corresponding cocycle with values in $\Gamma_2$.
Again  by exactness of the NE-SW sequence there exist smooth maps $e_{\alpha\beta}:U_{\alpha}\cap U_{\beta}\to H_2$ such that $\tilde g'_{\alpha\beta}=i_2(e_{\alpha\beta})\tilde g_{\alpha\beta}$. It i then straightforward to show that the cocycles $(f,b)$ and $(f',b')$ are equivalent via a coboundary $(1,e_{\alpha\beta})$. 
\begin{comment}
Then, $f_{\alpha\beta}'=i(e_{\alpha\beta})f_{\alpha\beta}$. Moreover, we have
\begin{align*}
i_2(b'_{\alpha\beta\gamma}) &= \tilde g'_{\alpha\gamma}\tilde {g'}_{\alpha\beta}^{-1}\tilde {g'}_{\beta\gamma}^{-1}i_1(a _{\alpha\beta\gamma})^{-1}
\\&= i_2(e_{\alpha\gamma})\tilde g_{\alpha\gamma}\tilde g_{\alpha\beta}^{-1}i_2(e_{\alpha\beta})^{-1}\tilde {g'}_{\beta\gamma}^{-1}i_1(a _{\alpha\beta\gamma})^{-1}
\\&= i_2(e_{\alpha\gamma})\tilde g_{\alpha\gamma}\tilde g_{\alpha\beta}^{-1}\tilde {g'}_{\beta\gamma}^{-1}\tilde {g'}_{\beta\gamma}i_2(e_{\alpha\beta})^{-1}\tilde {g'}_{\beta\gamma}^{-1}i_1(a _{\alpha\beta\gamma})^{-1}
\\&= i_2(e_{\alpha\gamma})\tilde g_{\alpha\gamma}\tilde g_{\alpha\beta}^{-1}\tilde {g'}_{\beta\gamma}^{-1}i_2(^{f'_{\alpha\beta}}e_{\alpha\beta})^{-1}i_1(a _{\alpha\beta\gamma})^{-1}
\\&= i_2(e_{\alpha\gamma})\tilde g_{\alpha\gamma}\tilde g_{\alpha\beta}^{-1}\tilde g_{\beta\gamma}^{-1}i_2(e_{\beta\gamma})^{-1}i_2(^{f'_{\alpha\beta}}e_{\alpha\beta})^{-1}i_1(a _{\alpha\beta\gamma})^{-1}
\\&= i_2(e_{\alpha\gamma})\tilde g_{\alpha\gamma}\tilde g_{\alpha\beta}^{-1}\tilde g_{\beta\gamma}^{-1}i_1(a _{\alpha\beta\gamma})^{-1}i_2(e_{\beta\gamma})^{-1}i_2(^{f'_{\alpha\beta}}e_{\alpha\beta})^{-1}
\\&= i_2(e_{\alpha\gamma})i_2(b_{\alpha\beta\gamma})i_2(e_{\beta\gamma})^{-1}i_2(^{f'_{\alpha\beta}}e_{\alpha\beta})^{-1}
\end{align*}
These two equalities show that $(f,b)$ and $(f',b')$ are equivalent cocycles, via equivalence data $(1,e_{\alpha\beta})$.
\end{comment}
Thus, we have a well defined assignment of of cohomology classes in $\check{\mathrm{H}}^1(X,\Gamma_2)$ to cocycles with values in $\Gamma_1$. Next, we suppose that  $(g,a)$ and $(g',a')$ are equivalent cocycles with values in $\Gamma_1$, i.e., there are smooth maps $h_{\alpha}:U_{\alpha}\to G_1$ and $e_{\alpha\beta}:U_{\alpha}\cap U_{\beta} \to H_1$ satisfying \cref{CoboundaryConditionsNonAb}.
Suppose that we have chosen the lifts $\tilde g_{\alpha\beta}$ of $g_{\alpha\beta}$. Let $\tilde h_{\alpha}:U_{\alpha} \to K$ be lifts of $h_{\alpha}$ along $p_1$. Then, $\tilde g'_{\alpha\beta} := i_1(e_{\alpha\beta})\tilde h_{\beta}\tilde g_{\alpha\beta}\tilde h^{-1}_{\alpha}$ is a valid lift for $g'_{\alpha\beta}$. Let $(f',b')$ be the corresponding cocycle with values in $\Gamma_2$.
It is then straightforward to check that $(f,b)$ and $(f',b')$ are equivalent via the coboundary $(h',1)$, where $h'$ consists of the maps $h'_{\alpha}:= p_2 \circ \tilde h_{\alpha}$.
Thus, that we obtain a well-defined map
\begin{equation*}
\check {\mathrm{H}}^1(X,\Gamma_1)\to \check {\mathrm{H}}^1(X,\Gamma_2)\text{.}
\end{equation*}         
Since the invertible butterfly is symmetric, we obtain in the same way a map in the opposite direction. It is easy to see that these maps are inverses of each other.
\begin{comment}
Starting with a cocycle $(g,a)$, we choose lifts $\tilde g_{\alpha\beta}$ of $g_{\alpha\beta}$ along $p_1$ and pass to the cocycle $(f,b)$ with $f_{\alpha\beta} := p_2\circ \tilde g_{\alpha\beta}$ and $i_2(b_{\alpha\beta\gamma}) = \tilde g_{\alpha\gamma}\tilde g_{\alpha\beta}^{-1}\tilde g_{\beta\gamma}^{-1}i_1(a _{\alpha\beta\gamma})^{-1}$. Performing the inverse, we may lift $f_{\alpha\beta}$ to the same functions $\tilde g_{\alpha\beta}$, and obtain back the very same $g_{\alpha\beta}$. Likewise, the equality
\begin{equation*}
i_2(a_{\alpha\beta\gamma}) = \tilde g_{\alpha\gamma}\tilde g_{\alpha\beta}^{-1}\tilde g_{\beta\gamma}^{-1}i_1(b _{\alpha\beta\gamma})^{-1}
\end{equation*} 
which follows from the one above since the images of $i_1$ and $i_2$ in $K$ commute, shows that we get back the very same $a_{\alpha\beta\gamma}$.
\end{comment}

\end{appendix}

\bibliography{bibfile}
\bibliographystyle{kobib}

\end{document}